\documentclass{cedram-aif}

\usepackage{geometry}
\usepackage{microtype}
\usepackage{graphicx}
\usepackage{amssymb}
\usepackage{epstopdf}
\usepackage{amsmath,amscd}
\usepackage{amsthm}
\usepackage{enumerate}
\usepackage{url,verbatim}
\usepackage{subfig}
\usepackage{xcolor}
\usepackage{comment}
\theoremstyle{definition}
\newtheorem{example}{Example}
\newtheorem{assumption}{Assumption}
\theoremstyle{remark}
\newtheorem{construction}{Construction}

\numberwithin{equation}{section}
\numberwithin{figure}{section}

% ...

% ...

\newcommand\ol{\overline}

\newcommand\EE{{\mathbb E}}

\newcommand\RR{{\mathbb R}}
\newcommand\ZZ{{\mathbb Z}}
\newcommand\NN{{\mathbb N}}

\newcommand\PP{{\mathbb P}}
\newcommand\HH{{\mathbb H}}
\newcommand\GUE{{\mathbb {GUE}}}

\renewcommand\ell{l}

\newcommand\GT{\mathbb{G}\mathbb{T}}
\newcommand\CC{\mathbb{C}}
\newcommand\bm{\mathbf{m}}
\newcommand\SH{\mathrm{SH}}

\equalenv{proposition}{prop}
\equalenv{corollary}{coro}

\equalenv{definition}{defi}

\equalenv{remark}{rema}

\title
[perfect matchings on contracting square-hexagon lattices]
{Limit shape and height fluctuations of random perfect matchings on square-hexagon lattices}

\alttitle{Forme limite et fluctuations de hauteur pour les couplages parfaits
aléatoires des graphes carrés-hexagones}

\author{C\'edric Boutillier}
\address{Laboratoire de Probabilit\'es, Statistique et Mod\'elisation, Sorbonne
  Universit\'e, CNRS, 4  place
Jussieu, F-75005 Paris}

\email{cedric.boutillier@sorbonne-universite.fr}
\urladdr{\url{https://www.lpsm.paris/pageperso/boutil/}}

\author{Zhongyang Li}
\address{Department of Mathematics,
University of Connecticut,
Storrs, Connecticut 06269-3009, USA}
\email{zhongyang.li@uconn.edu}
\urladdr{\url{https://mathzhongyangli.wordpress.com}}

\thanks{The authors thank the Institut Henri Poincar\'e for hospitality. The authors would also like to express gratitude to Alexei Borodin, Alexey Bufetov, Philippe Di Francesco, Vadim Gorin,  Richard Kenyon, Alisa Knizel, Jonathan Novak for helpful discussions, and the anonymous reviewers for careful reading of the paper.
CB acknowledges support from ANR-18-CE40-0033.
ZL acknowledges support from NSF DMS 1608896 and NSF DMS 1643027. }

  
\altkeywords{dimères, couplage parfait, forme limite, champ libre gaussien,
fonction de Schur}


\begin{document}
\begin{abstract}
  We study perfect matchings on the contracting square-hexagon lattice, constructed row by row from a row of
  either the square grid or the hexagonal lattice. Given $1\times n$ periodic weights to edges, we consider the
  probabilities of dimers proportional to the product of edge weights. We show that the
  partition function equals a Schur function of the edge weights.
 We then prove the Law of Large Numbers
  (limit shape) and the Central Limit Theorem (convergence to the Gaussian free
  field) for the corresponding height functions. We also show that 
  certain type of dimers near the turning corner converge in distribution to
  the eigenvalues of Gaussian Unitary Ensemble, and that in the scaling limit
  when each segment of the bottom boundary grows
  linearly with respect to the dimension of the graph, the frozen boundary is a
  cloud curve with multiple tangent points (depending on the period) along each horizontal boundary segment.
\end{abstract}

\begin{altabstract}
  Nous étudions les couplages parfaits de graphes, construits
  en prenant, pour chaque ligne, une ligne soir du réseau carré, soit du réseau
  hexagonal. Étant donnés des poids sur les arêtes avec une période $1\times n$,
   la fonction de partition est une fonction de
  Schur dépendant des poids. Nous
  obtenons dans la limite des grands systèmes une loi des grands nombres
  (forme limite) et un théorème central limite (convergence vers le champ libre)
  pour la fonction de hauteur associée. La
  distribution de certains dimères près du point de contact au
  bord converge vers celle des valeurs propres de l'ensemble unitaire gaussien.
  De plus, dans la limite d'échelle de systèmes pour lesquels
  chaque segment du bord croît linéairement
  avec la taille du graphe, le bord de la zone gelée est une \emph{courbe nuage}
avec des points de contact sur chaque segment du bord inférieur
dont le nombre dépend de la période.
\end{altabstract}

\maketitle

\section{Introduction}

A \emph{perfect matching}, or a \emph{dimer configuration}, is a subset of edges
of a graph such that each vertex is incident to exactly one edge. We study the
asymptotic behavior of periodically weighted random perfect matchings on a
class of domains called the contracting square hexagon lattice. Each row of the
lattice is either obtained from a row of a square grid or that of a hexagon
lattice; see Figure~\ref{fig:SH} for an example. On such a graph we shall assign
edge weights, satisfying the condition that the edge weights are invariant under
under horizontal translations, while changing row by row. We define a
probability measure for dimer configurations on such a graph to be proportional
to the product of edge weights.

When all the edge weights are 1, the underlying probability measure is the
uniform measure. The uniform perfect matchings on a square grid or a hexagonal
lattice have been studied extensively in the past few decades;
see~\cite{LP12,VP15,DM17} for recent results about uniform perfect matchings on the
hexagonal lattice, and~\cite{bk} for recent results about uniform perfect
matchings on the square grid. These results are obtained by applying and
re-developing the recent techniques developed to study the Schur processes;
see~\cite{OR01,OR07,AB07,AB11,bg2,bg}. Dimer model on a more general graph,
called the \emph{rail-yard graph}, may also be studied by techniques of Schur
processes; see~\cite{bbccr}.

Among the problems concerning the asymptotic behavior of perfect matchings on
larger and larger graphs, two of them are of special interest: the Law of Large
Numbers and the Central Limit Theorem. More precisely, when the underlying
finite graphs on which the dimer configurations are defined become larger and
larger whose rescaled version approximate a certain domain in the plane, the
rescaled height functions (which is a random function defined on faces of the
graph associated to each random perfect matching) are expected to converge to a
deterministic function (limit shape); and the non-rescaled height function is
expected to have Gaussian fluctuation. The limit shape behavior was first
observed from the arctic circle phenomenon for dimer models on large Aztec
diamond (which is a finite subgraph of the square grid with certain boundary
conditions); see~\cite{JPS98,Joh05}. In each component outside the inscribed the
circle, with probability exponentially close to 1, all the present edges of the
dimer configuration are along the same direction. This is called the frozen
region. Inside the circle, the probability that an edge of any certain direction
appears in the dimer configuration is non-degenerate and lies in the open
interval $(0,1)$; this is called the liquid region. The limit shape of
non-uniform dimer models on square grids, with more general boundary conditions,
was studied in~\cite{ckp00}, and the technique may be generalized to obtain a
variational principle, limit shape, and equation of frozen boundary for dimer
models on general periodic bipartite graphs; see~\cite{KO07}. The Aztec diamond
with $2\times 2$ period was studied in~\cite{CJ14}, with $2\times n$ period was
studied in~\cite{PS14}.

A square-hexagon lattice may be constructed row by row from either a row of a
square grid or a row of a hexagonal lattice.
In this paper, we assign positive weights to edges of the square-hexagon lattice
in such a way that the edge weights change row by row with a fixed finite
period. We then consider a special finite subgraph of the square-hexagon
lattice, called a contracting square-hexagon lattice. With the help of the
branching formula for the Schur function, we then show that the  
partition function of dimer configurations on a contracting square hexagon lattice, can be computed by a
Schur function depending on edge weights.

Note that Markov chains for sampling those random
dimer configurations on finite square-hexagon lattices with certain boundary
conditions (i.e.\ random tilings of tower graphs) were studied in~\cite{bf15}.

We then study the limit shape of the
dimer configurations when the mesh size of the graph goes to zero, and show
that the height function converges a deterministic function with an explicit
formula. We then find the equation of the frozen boundary, and show that the
frozen boundary is again a cloud curve (similar results was obtained
in~\cite{KO07} for the hexagonal lattice,  and obtained in~\cite{bk} for the
square grid), whose number of tangent points to the bottom boundary depend not
only on the number of segments with distinct boundary conditions on the bottom
boundary, but also on the size of the period of edge weights. In particular,
given our assignments of edge weights, the liquid region is a simply-connected
domain, i.e.\ there are no ``floating bubbles'' in the liquid region. We then
study the fluctuations of non-rescaled height function, and show that after a
homeomorphism from the liquid region to the upper half plane, the law of non-rescaled
height fluctuations are given in the limit by the Gaussian free field.
This extends the framework in which such a result is available for non-flat
boundary conditions. See~\cite{BF14,MD1} for the first results of this type,
and~\cite{MD2} for another method to show that a large class of models have this
kind of fluctuations.
We also study
the distribution of present edges joining a row with odd index to a row with
even index above it, and show that near the top boundary, these edges have the
same distribution as the eigenvalues of a GUE random matrix, which was
established in~\cite{OR06} for plane partitions and in~\cite{JN06} for the Aztec
diamond. In~\cite{zl1,zl2}, the case when the periodic edge weights decay polynomially 
with respect to the size of the graph is investigated, the liquid region is proved to split to finitely
many disconnected components, and the height fluctuation in each component of the liquid region
is proved to be an independent Gaussian free field in the scaling limit. 

The organization of the paper is as follows. In Section~\ref{sec:cb}, we define
the contracting square-hexagon lattice and prove the formula to compute the
partition function of dimer configurations on such a lattice via Schur functions
depending on edge weights. In Section~\ref{sec:ls}, we prove an explicit formula
for the limit of the rescaled height function. In Section~\ref{s3}, we prove an
explicit formula for the density of the limit counting measure associated to the
dimer configurations on each row of the contracting square-hexagon lattice, and
define the frozen region to be the region whenever the density is 0 or 1. In
Section~\ref{sec:fb}, we prove an explicit formula for the frozen boundary (the
boundary of the frozen region) and show that the frozen boundary a cloud curve.
In Section~\ref{sec:gue}, we show that  the distribution of present edges
joining an row with odd index to a row with even index above it near the top
boundary is the same as that of eigenvalues of a GUE random matrix. In
Section~\ref{sec:gff}, we show that the fluctuation of the non-rescaled height
function is a homeomorphism of the Gaussian free field in the upper half plane.
In Section~\ref{sec:ex}, we give simulations of the distribution of dimer models
on the contracting square-hexagon lattice, and draw pictures of the limit shape.

\section{Combinatorics}\label{sec:cb}

In this section, we define the contracting square-hexagon lattice on which we
shall study the perfect matching, or the dimer model. By an explicit bijection
between perfect matchings on the contracting square-hexagon lattice and
sequences of certain Young diagrams, we express the probability measure on
perfect matchings, in which the probability of each configuration is
proportional to product of edge weights, in terms of Schur functions. As a
result, the partition function of dimer configurations on such contracting
square-hexagon lattice can also be expressed in term of Schur functions.
This extends known results for the dimer model on the square grid~\cite{bk} and
hexagonal
lattice~\cite{LP12,bg}, where the
underlying measure is uniform or a $q$-deformation of the uniform measure.

\subsection{Square-hexagon Lattices}
Consider a doubly-infinite binary sequence indexed by integers
$\ZZ=\{\ldots,-2,-1,0,1,2,\ldots\}$.
\begin{equation*}
  \check{a}=(\ldots,a_{-2},a_{-1},a_0,a_1,a_2,\ldots)\in\{0,1\}^{\ZZ}.
\end{equation*}

The \emph{whole-plane square-hexagon lattice} associated with the  sequence
$\check{a}$, is a bipartite plane graph $\mathrm{SH}(\check{a})$ defined as
follows.
Each vertex of $\mathrm{SH}(\check{a})$ is either black or white, and we
identify the vertices with points on the plane.
Its vertex set is a subset of $\frac{\ZZ}{2}\times \frac{\ZZ}{2} $.
For $m\in\frac{\ZZ}{2}$, the vertices with ordinate $m$ correspond to the $2m$th
row  of the graph. Vertices on even rows (for $m$ integer) are colored in black.
Vertices on odd rows (for $m$ half integer) are colored in white.
 \begin{itemize}
\item  each black vertex on the $(2m)$th row is adjacent to two white vertices in the $(2m+1)$th row; and
\item if $a_m=1$, each white vertex on the $(2m-1)$th row is adjacent to exactly one black vertex in the $(2m)$th row; 
 if $a_m=0$, each white vertex on the $(2m-1)$th row is adjacent to two black vertices in the $(2m)$th row. 
  \end{itemize}
 See Figure~\ref{lcc}.
Such a graph is also related to the
rail-yard graph; see~\cite{bbccr}.

\begin{figure}
  \subfloat[Structure of $\mathrm{SH}(\check{a})$ between the $(2m)$th row and the $(2m+1)$th row]{\includegraphics[width=.6\textwidth]{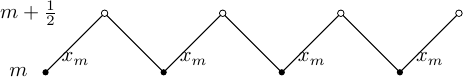}}\\
  \subfloat[Structure of $\mathrm{SH}(\check{a})$ between the $(2m-1)$th row and the $(2m)$th row when $a_m=0$]{\includegraphics[width = .6\textwidth]{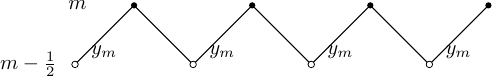}}\\
  \subfloat[Structure of $\mathrm{SH}(\check{a})$ between the $(2m-1)$th row and the $(2m)$th row when $a_m=1$]{\includegraphics[width = .55\textwidth]{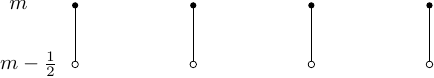}}
  \caption{Graph structures of the square-hexagon lattice on the $(2m-1)$th,
    $(2m)$th, and $(2m+1)$th rows depend on the values of $(a_m)$. Black vertices
    are along the $(2m)$th row, while white vertices are along the $(2m-1)$th and
  $(2m+1)$th row.}
  \label{lcc}
\end{figure}

We shall assign edge weights to the whole-plane square-hexagon lattice
$\SH(\check{a})$ as follows.

\begin{assumption}\label{apew}
For $m\geq 1$, we assign weight $x_m>0$ to each NE-SW edge joining the
$(2m)$th row  to the $(2m+1)$th row of $\mathrm{SH}(\check{a})$. We assign
weight $y_m>0$ to each NE-SW edge joining the $(2m-1)$th row to the $(2m)$th
 row of $\mathrm{SH}(\check{a})$, if such an edge exists. We assign weight $1$
to all the other edges. 
\end{assumption}

It is straightforward to check the following lemma describing the faces of a
whole-plane square-hexagon lattice.

\begin{lemma}Each face of $\mathrm{SH}(\check{a})$ is either a square (degree-4 face) or a hexagon (degree-6 face). Let $m\geq 1$ be a positive integer.
\begin{enumerate}
\item There exists a degree-6 face including both black vertices in the $(2m)$th row and black vertices in the $(2m+2)$th row if and only if $a_{m+1}=1$.
\item There exists a degree-4 face including both black vertices in the $(2m)$th row and black vertices in the $(2m+2)$th row if and only if $a_{m+1}=0$.
\end{enumerate}
\end{lemma}

A \emph{contracting square-hexagon lattice} is built from a whole-plane
square-hexagon lattice as follows:

\begin{definition}\label{dfr}
  Let $N\in \NN$. Let $\Omega=(\Omega_1,\ldots,\Omega_N)$ be an $N$-tuple of
  positive integers, such that  $1=\Omega_1<\Omega_2<\cdots<\Omega_{N}$. Set
  $m=\Omega_N-N$.
  The contracting square-hexagon lattice $\mathcal{R}(\Omega,\check{a})$ is a
  subgraph of $\mathrm{SH}(\check{a})$ built of $2N$ or $2N+1$ rows.
  We shall now enumerate the rows of $\mathcal{R}(\Omega,\check{a})$
  inductively, starting from the bottom as follows:
  \begin{itemize}
    \item The first row consists of vertices $(i,j)$ with $i=\Omega_1-\frac{1}{2},\ldots,\Omega_N-\frac{1}{2}$ and $j=\frac{1}{2}$. We call this row the boundary row of $\mathcal{R}(\Omega,\check{a})$.
    \item When $k=2s$, for $s=1,\ldots N$,  the $k$th row consists of vertices $(i,j)$ with $j=\frac{k}{2}$ and incident to at least one vertex in the $(2s-1)th$ row of the whole-plane square-hexagon lattice $\mathrm{SH}(\check{a})$ lying between the leftmost vertex and rightmost vertex of the $(2s-1)$th row of $\mathcal{R}(\Omega,\check{a})$
    \item When $k=2s+1$, for $s=1,\ldots N$,  the $k$th row consists of vertices $(i,j)$ with $j=\frac{k}{2}$ and incident to two vertices in the $(2s)$th row of  of $\mathcal{R}(\Omega,\check{a})$.
  \end{itemize}
\end{definition}

  The transition from an odd row to the next even row in a contracting
  square-hexagon lattice can be of two kinds
  depending on whether vertices are connected to one or two vertices of the row
  above them. See Figures \ref{fig:raz}, \ref{fig:HD}, and \ref{fig:SH} for examples of contracting square-hexagon lattices.

  \begin{definition}
    \label{defI1I2}
  Let $I_1$ (resp.\@ $I_2$) be the set of indices $j$ such that vertices
  of the $(2j-1)$th row are connected to one vertex (resp.\@ two vertices) of
  the $(2j)$th row.
  In terms of the sequence $\check{a}$,

  \begin{equation*}
    I_1=\{k\in \{1,\dots,N\}\ |\ a_k=1\},\quad
    I_2=\{k\in \{1,\dots,N\}\ |\ a_k=0\}.
  \end{equation*}
\end{definition}

  The sets $I_1$ and $I_2$ form a partition of $\{1,\dots,N\}$, and we have
  $|I_1|=N-|I_2|$.


\subsection{Perfect Matching}

\begin{definition}\label{dfvl}
  A \emph{dimer configuration}, or a \emph{perfect matching} $M$ of a
  contracting square-hexagon lattice $\mathcal{R}(\Omega,\check{a})$ is a set of
  edges $((i_1,j_1),(i_2,j_2))$, such that each vertex of
  $\mathcal{R}(\Omega,\check{a})$ belongs to a unique edge in $M$.

  The set of perfect matchings of $\mathcal{R}(\Omega,\check{a})$ is denoted by
  $\mathcal{M}(\Omega,\check{a})$.
\end{definition}

\begin{definition}
  Let $M\in \mathcal{M}(\Omega,\check{a})$ be a perfect matching of
  $\mathcal{R}(\Omega,\check{a})$. We call an edge $e=((i_1,j_1),(i_2,j_2))\in
  M$ a \emph{$V$-edge} if $\max\{j_1,j_2\}\in\NN$ (i.e.\@ if its higher
  extremity is black) and we call it a \emph{$\Lambda$-edge}
  otherwise. In other words, the edges going upwards starting from an odd row
  are $V$-edges and those ones starting from an even row are $\Lambda$-edges. We
  also call the corresponding vertices-$(i_1,j_1)$ and $(i_2,j_2)$ $V$-vertices
  and $\Lambda$-vertices accordingly.
\end{definition}

\begin{lemma}
  Let $M\in \mathcal{M}(\Omega,\check{a})$ be a perfect matching of
  $\mathcal{R}(\Omega,\check{a})$. For each $1\leq i\leq N$, the number of
  $V$-edges joining the $(2i-1)$th row and the $(2i)$th row is one more than the
  number of $V$ edges joining the $(2i)$th row and the $(2i+1)$th row.
\end{lemma}
\begin{proof}

For $j\in\{2i,2i+1\}$, let $t_j$ be the number of vertices in the $j$th row of
$\mathcal{R}(\Omega,\check{a})$. From the construction of
$\mathcal{R}(\Omega,\check{a})$ in Definition~\ref{dfr}, we have
\begin{equation}
t_{2i+1}=t_{2i}-1.
\label{dcv}
\end{equation}
 Let $s$ be the number of $V$-edges joining the $(2i-1)$th row and $(2i)$th row.
 Then there exists $(t_{2i}-s)$ $\Lambda$-edges joining the $(2i)$th row to the
 $(2i+1)$th row. Hence there are $(t_{2i+1}-t_{2i}+s)$ $V$-edges joining the
 $(2i+1)$th row and $(2i+2)$th row. Then the lemma follows from \eqref{dcv}.
\end{proof}


\begin{example}\mbox{}
\begin{enumerate}
\item If for each $i\geq 1$, each vertex on the $i$th row of
  $\mathrm{SH}(\check{a})$ is adjacent to two vertices on the $(i+1)$th row of
  $\mathrm{SH}(\check{a})$, then the construction in Definition~\ref{dfr} gives
  us the rectangular Aztec diamond studied in~\cite{bk}.
\item For each $i\geq 1$, each vertex on the $(2i-1)$th row of
  $\mathrm{SH}(\check{a})$ is adjacent to one vertex on the $(2i)$th row of
  $\mathrm{SH}(\check{a})$, then the construction in Definition~\ref{dfr} gives
  us the contracting hexagonal lattice studied in~\cite{LP12,bg2}.
\end{enumerate}
\end{example}

\subsection{Partitions and Young Diagrams}
\label{sec:pyd}

Following~\cite{bk}, we will use \emph{signatures} to encode the perfect matchings of 
of the contracting square-hexagons.

\begin{definition}
  A \emph{signature} of length $N$ is a sequence of nonincreasing integers
  $\mu=(\mu_1\geq \mu_2\geq \ldots \geq\mu_N)$. Each $\mu_k$ is a
  \emph{part} of the signature $\mu$. The \emph{length} $N$ of the signature $\mu$
  is denoted by $\ell(\mu)$. We say that $\mu$ is \emph{non-negative} if
  $\mu_N\geq 0$. The \emph{size} of a non-negative signature $\mu$ is
  \begin{equation*}
    |\mu| = \sum_{i=1}^N \mu_i.
  \end{equation*}
  $\GT_N$ denotes the set of signatures of length $N$, and $\GT^+_N$ is the
  subset of non-negative signatures.
\end{definition}

To the boundary row $\Omega=(\Omega_1<\cdots<\Omega_N)$ of a contracting
square-hexagon lattice is naturally associated a non-negative signature $\omega$
of length $N$ by:
\begin{equation*}
  \omega=(\Omega_N-N,\dotsc,\Omega_1-1).
\end{equation*}

Non-negative signatures are the convenient objects to talk about integer
partitions with a given number of zero parts. Most of the objects 
constructed from partitions are available for non-negative signatures, in
particular Young diagrams, and interlacement relations which we recall now.

A graphic way to represent a non-negative signature $\mu$ is through its
\emph{Young diagram} $Y_\mu$, a collection of $|\mu|$ boxes arranged on
non-increasing rows aligned on the left: with
$\mu_1$ boxes on the first row, $\mu_2$ boxes on the second row,\dots $\mu_N$
boxes on the $N$th row. Some rows may be empty if the corresponding $\mu_k$ is
equal to 0. The correspondence between non-negative signatures of length $N$ and
Young diagrams with $N$ (possibly empty) rows is a bijection.

If all the parts of a non-negative signature $\mu$ are equal (say $N$ parts equal to
$m$), the Young diagram $Y_\mu$ has a rectangular shape. We then say that $\mu$
is \emph{rectangular}, and note $\mu=N\times m$, and $Y_{N\times m}$ for its Young
diagram.

Young diagrams included in $Y_{N\times m}$ are those corresponding to
non-negative signatures of length $N$ and parts bounded by $m$.

\begin{definition}
 Let $Y,W$ be two Young diagrams. We say that $Y\subset W$ \emph{differ by a
 horizontal strip} if the
  collection of boxes in $Z=W\setminus Y$ contains at most one box in every
  column. We say that they \emph{differ by a vertical strip} if $Z$ contains at
  most one box in every row.

  We say that two non-negative signatures $\lambda$ and $\mu$ \emph{interlace}, and
  write $\lambda \prec \mu$ if $Y_\lambda\subset Y_\mu$ differ by a horizontal
  strip. We say they \emph{co-interlace} and write $\lambda\prec'\mu$ if
  $Y_\lambda\subset Y_\mu$ differ by a vertical strip.
\end{definition}

Another way to graphically represent signatures is to use \emph{Maya diagrams}, which
usually represent a collection of white and black particles
(here squares $\square$, $\blacksquare$)
on the 1-dimensional
lattice~$\mathbb{Z}$. For our purposes, since we will work with non-negative
signatures with Young diagram included in a rectangle of a given size,
we will need finite version of Maya diagrams, defined below.

\begin{definition}\label{dmd}
  A \emph{finite Maya diagram} $\mathbf{m}$ of length $n$ is an element of
  $\{\square,\blacksquare\}^n$. The \emph{origin} of the Maya diagram is a
  position between two successive elements of the sequence, such that the number
  of elements on the left (resp.\@ on the right) of this position is equal to
  the number of $\blacksquare$ (resp.\@ $\square$) particles.
\end{definition}

Non-negative signatures $\mu$ of length $N$ with parts bounded by $m$
corresponds bijectively to finite Maya diagrams $\emph{m}_\mu$ of length
$N+m$ and exactly $N$ black particles, by the following coding of the
non trivial part of the
boundary of
$Y_\mu$ seen as a lattice path of length $N+m$ connecting two opposite corners
of $Y_{N\times m}$. A vertical step corresponds to a $\blacksquare$ and a horizontal
step corresponds to a $\square$. See Figure~\ref{fig:young_maya}.

This way, the signature of length $N$ with all
parts equal to $0$ (resp.\@ equal to $m$) corresponds to the Maya diagram where
the $N$ black particles are on the left (resp.\@ on the right) of the $m$ white
particles.

\begin{figure}
  \centering
  \includegraphics[width=10cm]{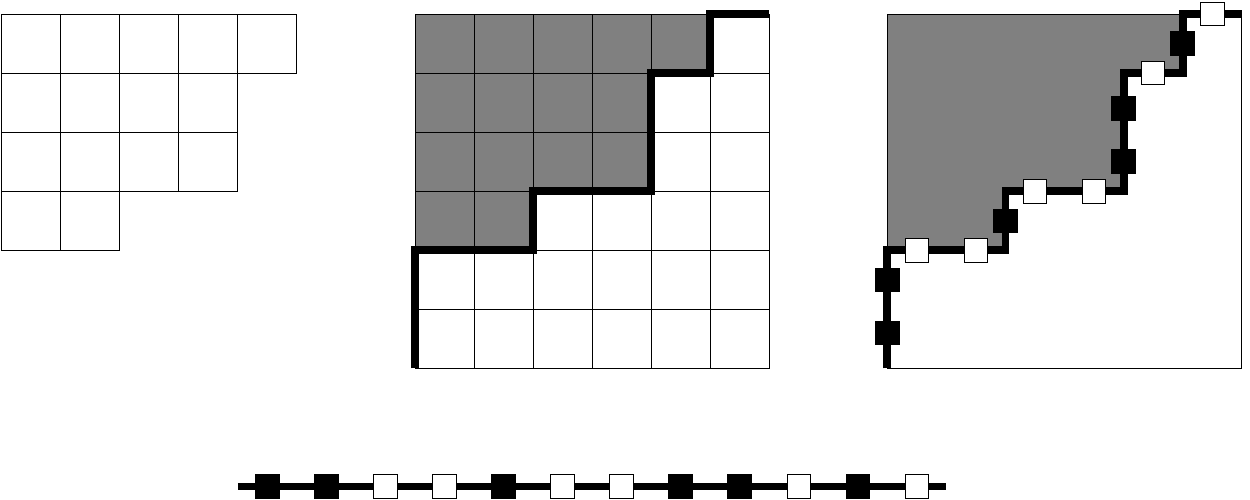}
  \caption{Top: the Young diagram of $(5,4,4,4,2,0,0)$, seen as a non-negative
    signature with 6 parts, all bounded by 6 (left); the same Young diagram
    included in the rectangle $6\times 6$ and the path representing the boundary
    of the Young diagram from the lower left to the upper right corner (middle),
    the encoding of the steps of path with black and particles. Bottom, the
    actual finite Maya diagram of size $6+6$.}
    \label{fig:young_maya}
\end{figure}

We shall associate to each perfect matching in $\mathcal{M}(\Omega,\check{a})$
a sequence of non-negative signatures, one for each row of the graph.

\begin{construction}\label{ct}
  Let $j\in\{1,\dots,2N+1\}$. Assume that the $j$th row of
  $\mathcal{R}(\Omega,\check{a})$ has $n_j$ V-vertices and $m_j$
  $\Lambda$-vertices. Then we first associate a finite Maya diagram of length
  $n_j+m_j$ with $n_j$ black particles: every $\Lambda$-vertex (resp.\@ $V$-vertex)
  is mapped to a white (resp.\@ black) particle. This Maya diagram corresponds
  then to a Young diagram of a non-negative signature of length $n_j$ and parts
  bounded by $m_j$, which in turn has a Young diagram $Y_j$ fitting in a
  $n_j\times m_j$ rectangle.

  Note that to perform this construction for the
  boundary row (i.e.\@ $j=1$), vertices with coordinates between $\Omega_1$ and
  $\Omega_N$, which are not present in the graph are considered a (virtual)
  $\Lambda$-vertices, and should be taken into account to compute $n_1$ and $m_1$.
\end{construction}

The encoding of dimer configurations of finite contracting square-hexagon graphs
with Maya diagrams allows then for a bijective correspondence with sequences of
interlaces signatures. More precisely:
\begin{theorem}[\cite{bk} Theorem 2.9,~\cite{bbccr}]
  For given $\Omega$, $\check{a}$, let $\omega$ be the signature associated to
  $\Omega$. Then the construction~\ref{ct} defines a
  bijection between the set of perfect matchings
  $\mathcal{M}(\Omega,\check{a})$ and the set $S(\omega,\check{a})$ of
  sequences of non-negative signatures
  \begin{equation*}
    \{(\mu^{(N)},\nu^{(N)},\dots,\mu^{(1)}, \nu^{(1)}, \mu^{(0)}\}
  \end{equation*}
  where the signatures satisfy the following properties:
  \begin{itemize}
    \item All the parts of $\mu^{(0)}$ are equal to 0;
    \item The signature $\mu^{(N)}$ is equal to $\omega$;
    \item The signatures satisfy the following (co)interlacement relations:
      \begin{equation*}
        \mu^{(N)} \prec' \nu^{(N)} \succ \mu^{(N-1)} \prec' \cdots
        \mu^{(1)} \prec' \nu^{(1)} \succ \mu^{(0)}.
      \end{equation*}
  \end{itemize}
  Moreover, if $a_m=1$, then $\mu^{(N+1-k)}=\nu^{(N+1-k)}$.
  \label{myb}
\end{theorem}

\begin{remark}
  The interlacing relations have the following implications on the signatures
  and their Young diagrams:
  \begin{itemize}
    \item for all $i$, $\ell(\mu^{(i)}) = \ell(\nu^{(i)})=i$;
    \item for all $i$, parts of $\mu^{(i)}$ (resp.\@ $\nu^{(i)}$) are all
      bounded by $\Omega_N+i-t(i)$ (resp.\@ $\Omega_N-t(i+1)+i+1$).
  \end{itemize}
  where
  \begin{equation*}
    t(i)=\#(I_1\cap\{1,\dots,i-1\})
  \end{equation*}
  is the number of odd rows below with ordinate less than $i$, where vertices
  are connected to a single vertex of the row above them.
\end{remark}

The following lemma relates the size of the signatures associated to rows of the
graph with the number of NE-SW dimers connecting these rows:
\begin{lemma}\label{la116}
  \begin{enumerate}
      Let $1\leq i\leq N$.
    \item If in the $(2i)$th row of $\mathcal{R}(\Omega,\check{a})$, the dimer
      configuration is given by the signature $\nu^{(N-i+1)}$; and in the
      $(2i+1)$th row, the dimer  configuration is given by the signature
      $\mu^{(N-i)}$, then the number of present NE-SW edges joining the $(2i)$th
      row to the $(2i+1)$th row is $|\nu^{(N-i+1)}|-|\mu^{(N-i)}|$.
    \item Assume that each vertex in the $(2i-1)$th row of
      $\mathrm{RH}(\check{a})$ is adjacent to two vertices in the $(2i)$th row.
      If in the $(2i-1)$th row of $\mathcal{R}(\Omega,\check{a})$, the dimer
      configuration is given by the signature $\mu^{(N-i+1)}$; and in the
      $(2i)$th row, the dimer  configuration is given by the signature
      $\nu^{(N-i+1)}$, then the number of present NE-SW edges joining the
      $(2i-1)$th row to the $(2i)$th row is $|\nu^{(N-i+1)}|-|\mu^{(N-i+1)}|$.
\end{enumerate}
\end{lemma}

\begin{proof}
  Edges present in a dimer configuration between rows $2i$ and $2i+1$ are $\Lambda$-edges, connecting
  $\Lambda$-vertices which in terms of Maya diagram are $\square$-particles.

  Let $M_{\nu,i}$ (resp.\ $M_{\mu,i}$) be Maya diagrams corresponding to $\nu^{(N+1-i)}$ and $\mu^{(N-i)}$.
  Note that $M_{\nu,i}$ (resp.\ $M_{\mu,i}$) has exactly $N+1-i$ (resp.\@ $N-i$) boxes to the left of its origin, and
  both $M_{\nu,i}$ and $M_{\mu,i}$ have the same number of boxes to the right or their origins.
  This follows from the fact that $M_{\nu,i}$ (resp.\ $M_{\mu,i}$) has $N+1-i$
  (resp.\@ $N-i$) black squares, while 
  both $M_{\nu,i}$ and $M_{\mu,i}$ have the same number of white squares; see Definition~\ref{dmd}.
  If we look at the $M_{\nu,i}$ and $M_{\mu,i}$ with their origin aligned, the
  presence of a NE-SW edge corresponds to a $\square$-particle in $M_{\nu,i}$ jumping to the
  right by one step in $M_{\mu,i}$ (whereas a NW-SE edge would correspond to a
  $\square$-particle
  staying at the same place in both $M_{\nu,i}$ and $M_{\mu,i}$; see
  Figure~\ref{fig:bwd}.

  The number of NE-SW edges between these two rows is thus the total
  displacement of $\square$-particles. Since in a Maya diagram,
  moving a $\square$-particle to the right corresponds to removing a box in the Young diagram, thus
  decreasing the size of the partition by 1, it follows that the total
  displacement of the $\square$-particles is equal to
  $|\nu^{(N+1-i)}|-|\mu^{(N-i)}|$.

  The second part is proved analogously.
\end{proof}
\begin{figure}
  \centering
  \includegraphics[width=8cm]{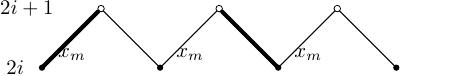}\\
  \bigskip
  \bigskip
  \includegraphics[width=2cm]{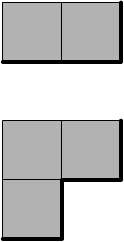}\qquad\qquad   \includegraphics[width=4cm]{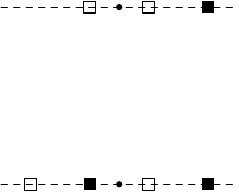}
  \caption{%
    The top graph represents part of a dimer configuration of a square-hexagon
    lattice, more precisely, the configuration of edges between the $2i$th row
    and the $(2i+1)$th row. Present edges in the figure are $\Lambda$-edges.
    Endpoints of present edges in the figure are $\Lambda$-vertices; all the
    other vertices are $V$-vertices. The mid-left graph is the Young diagram
    associated to the $(2i+1)$th row; and the bottom-left graph is the Young
    diagram associated to the $2i$th row. The mid-right graph is the Maya
    diagram corresponding to the $(2i+1)$th row; and the bottom-right graph is
    the Maya diagram associated to the $2i$th row.
  }\label{fig:bwd}
\end{figure}

A simple and direct consequence of Lemma~\ref{la116} is the following:

\begin{corollary}
  \label{cor:nb_ne_sw}
  The total number of NE-SW edges in a perfect matching of
  $\mathcal{R}(\Omega,\check{a})$ is equal to $|\omega|$, the size of the
  non-negative signature corresponding to the boundary row $\Omega$.
\end{corollary}

\subsection{Schur Functions and Partition Function of Perfect Matchings}

Recall that the \emph{partition function} of the dimer model of a finite graph
$G$ with edge weights $(w_e)_{e\in E(G)}$ is given by
\begin{equation*}
Z=\sum_{M\in \mathcal{M}}\prod_{e\in M}w_e,
\end{equation*}
where $\mathcal{M}$ is the set of all perfect matchings of $G$. The
\emph{Boltzmann dimer probability measure} on $M$ induced by the weights $w$ is
thus defined by declaring that probability of a perfect matching is equal to
\begin{equation*}
  \frac{1}{Z}\prod_{e\in M} w_e.
\end{equation*}

In this section, we prove a formula which express the partition function of
perfect matchings on a contracting square-hexagon lattice
$\mathcal{R}(\Omega,\check{a})$ as a Schur function depending on the boundary
configuration $\Omega$ and the edge weights.


\begin{definition}
  Let $\lambda\in \GT_N$. The \emph{rational Schur function} $s_\lambda$
  associated to $\lambda$ is the homogeneous symmetric function of degree
  $|\lambda|$ in $N$ variables defined by:
\begin{equation*}
  s_{\lambda}(u_1,\ldots,u_N)=
  \frac{\det_{i,j=1,\ldots,N}(u_i^{\lambda_j+N-j})}{\prod_{1\leq
  i<j\leq N}(u_i-u_j)}.
\end{equation*}
\end{definition}

Let $\mu,\nu\in\GT_n^+$ be two non-negative signature of length $n$. It
is well-known that Schur functions form a basis for the algebra of symmetric
functions. Let $\beta=(\beta_1,\ldots,\beta_n)\in \CC^n$. We define
as in~\cite{bk} the coefficients
$\mathrm{st}_{\beta}\left(\mu\rightarrow\nu\right)$ and
$\mathrm{pr}_{\beta}\left(\lambda^{(i)}\rightarrow \mu^{(i-1)}\right)$
as follows:
\begin{equation}
\mathrm{pr}_{\beta}(\nu\rightarrow\lambda)=
\begin{cases}
\beta_1^{|\nu|-|\lambda|}
\frac{{s_{\lambda}(\beta_2,\ldots,\beta_n)}}{s_{\nu}(\beta_1,\ldots,\beta_n)}
&\text{if $\lambda\prec\nu$}
\\
0 & \text{otherwise}
\end{cases},
\label{d3}
\end{equation}
and 
\begin{equation}
\mathrm{st}_{\beta}(\mu\rightarrow\lambda)=
\begin{cases}
  \frac{1}{\prod_{j=1}^{n}(1+\beta_j)}
  \frac{{s_{\lambda}(\beta_1,\ldots,\beta_n)}}{s_{\mu}(\beta_1,\ldots,\beta_n)}
  &\text{if $\mu\prec'\lambda$}
  \\
  0 &\text{otherwise}
  \label{d4}
\end{cases}.
\end{equation}

By the branching formula for Schur polynomials, and the same argument
as~\cite[Lemma 2.12]{bk} we have the following identities
\begin{align}
  \frac{s_{\mu}(u_1,\ldots,u_n)}{s_{\mu}(\beta_1,\ldots,\beta_n)}
  \prod_{j=1}^{n}\frac{(1+u_j)}{(1+\beta_j)}
  &=\sum_{\lambda\in \GT_n}\mathrm{st}_{\beta}\left(\mu\rightarrow\lambda\right)
  \frac{s_{\lambda}(u_1,\ldots,u_n)}{s_{\lambda}(\beta_1,\ldots,\beta_n)},
  \label{d1}
  \\
  \frac{s_{\nu}(\beta_1,u_2,\ldots,u_n)}{s_{\nu}(\beta_1,\beta_2\ldots,\beta_n)}
  &=\sum_{\lambda\in \GT_{n-1}}\mathrm{pr}_{\beta}\left(\nu\rightarrow\lambda\right)
  \frac{s_{\lambda}(u_2,\ldots,u_n)}{s_{\lambda}(\beta_2,\ldots,\beta_n)}
  \label{d2}.
\end{align}
from which we deduce that the following holds:
\begin{equation*}
\sum_{\lambda\prec \nu}
\mathrm{pr}_{\beta}(\nu\rightarrow\lambda)=1,
\qquad
\sum_{\lambda\succ'\mu}
\mathrm{st}_{\beta}(\mu\rightarrow\nu)=1.
\end{equation*}

For $i\in \{1,2,\ldots,i\}$, define 
\begin{equation}
C_i=\left(x_i,x_{i+1},\ldots,x_N\right)\in \RR^{N-i+1},
\label{dci}
\end{equation}
and for $i\in I_2$,  define
\begin{equation}
B_i=y_i C_i=\left(y_{i}x_{i},y_{i}x_{i+1}\ldots,y_{i}x_{N}\right)\in \RR^{N-i+1}
\label{dbi}
\end{equation}

By homogeneity of the Schur functions, for each $i\in I_2$, and $\lambda\in \GT_{N-i}^+$, we have
\begin{eqnarray}
s_{\lambda}(B_{i})=\left(y_{i}\right)^{|\lambda|}s_{\lambda}(C_{i}).\label{slbc}
\end{eqnarray}

For $i\in I_2$, define
\begin{eqnarray}
\Gamma_i=\prod_{t=i+1}^{N}\left(1+y_{i}x_{t}\right).\label{gi}
\end{eqnarray}

Recall that $\mathcal{S}^N_{\omega}(\check{a})$, as defined in Theorem~\ref{myb}, is the set of all the sequences of partitions in bijection with the set $\mathcal{M}(\Omega,\check{a})$, which consists of all the perfect matchings on the contracting square-hexagon lattice with bottom boundary condition $\Omega$ and structures on rows given by $\check{a}$.  We now define a probability measure on $\mathcal{S}^N_{\omega}(\check{a})$ as follows:
\begin{multline}
\mathbb{P}_{\omega}^N\left(\mu^{(N)}, \nu^{(N)},\ldots,\mu^{(1)},
\nu^{(1)},\mu^{(0)}\right)=\\
1_{\{\mu^{(N)}=\omega\}}\times
\prod_{j\in
I_2}\mathrm{st}_{B_{j}}\left(\mu^{(N-j+1)}\rightarrow\nu^{(N-j+1)}\right)\prod_{i=1}^{N}\mathrm{pr}_{C_{i}}\left(\nu^{(N-i+1)}\rightarrow\mu^{(N-i)}\right).
\label{pb}
\end{multline}

The following proposition connects this measure with the Boltzmann measure on
dimers configurations of the associated contracting square-hexagon graph:
\begin{proposition}
  \label{p16}
  The bijection described in Theorem~\ref{myb} transports the probability
  measure~\eqref{pb} on $\mathcal{S}^N_{\omega}(\check{a})$ to a Boltzmann dimer
  measure on the perfect matchings of $\mathcal{R}(\Omega,\check{a})$, with
  the following weights
  \begin{itemize}
    \item each  NE-SW edge joining the $(2i)$th row to the $(2i+1)$th row has weight $x_i$; and
    \item each  NE-SW edge joining the $(2i-1)$th row to the $2i$th row has weight $y_i$, if such an edge exists;
    \item All the other edges have weight 1.
  \end{itemize}

  Moreover, the dimer partition function on $\mathcal{R}(\Omega,\check{a})$ for
  these weights is given by
  \begin{equation*}
    Z=\left[\prod_{i\in I_2}\Gamma_i\right] s_{\omega}(x_{1},\ldots,x_{N})
  \end{equation*}
  where $\omega$ is the $N$-tuple corresponding to the boundary row of
  $\mathcal{R}(\Omega,\check{a})$, and $\Gamma_i$ is defined as in \eqref{gi}.
\end{proposition}

\begin{proof}
  By \eqref{pb}, \eqref{d3}, \eqref{d4} and \eqref{slbc}, we have
  \begin{multline}
    \mathbb{P}_{\omega}^N\left(\mu^{(N)}, \nu^{(N)},\ldots,\mu^{(1)},
    \nu^{(1)},\mu^{(0)}\right)=\\
    1_{\{\mu^{(N)}=\omega\}}
    \frac{%
      \prod_{i\in I_2}\left[%
        \left(y_{i}\right)^{|\nu^{(N-i+1)}|-|\mu^{(N-i+1)}|}
      \right]
      \prod_{j=1}^{N}\left[%
        \left(x_{j}\right)^{|\nu^{(N-j+1)}|-|\mu^{(N-j)}|}
      \right]
    }{%
      \left[\prod_{i\in I_2}\Gamma_i\right]  s_{\omega}(x_{1},\ldots,x_{N})
    }
    \label{pb1}.
  \end{multline}
  When $\mu^{(N)}=\omega$, the numerator of \eqref{pb1} is exactly  by
  Lemma~\ref{la116} the product of
  weights of present edges in the perfect matching corresponding to the sequence
  of non-negative signatures
  \begin{equation*}
    \left(\mu^{(N)}, \nu^{(N)},\ldots,\mu^{(1)}, \nu^{(1)},\mu^{(0)}\right)
  \end{equation*}
 Then the proposition follows.
\end{proof}

\subsection{Examples}

In this section, we provide a few examples of contracting square-hexagon
lattices, compute the partition functions of dimer configurations on these
graphs explicitly, and verify that these partition functions are equal to the
formula given by Proposition~\ref{p16}.

\subsubsection{Square Grid}
Consider perfect matchings on a square grid with edge weights assigned as in the
Figure~\ref{fig:raz}.

\begin{figure}
  \includegraphics[width=8cm]{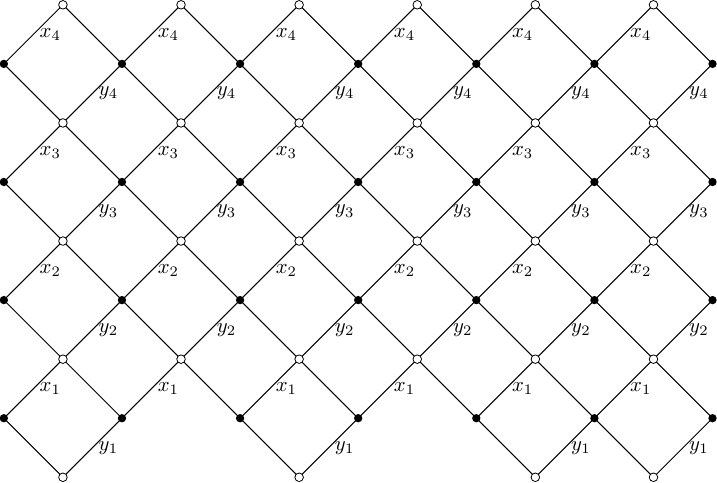}
\caption{Rectangular Aztec diamond with $N=4$, $m=2$, $\Omega=(1,3,5,6)$, and $a_i=0$.}
\label{fig:raz}
\end{figure}

\begin{corollary}
  \label{c19}
  Let $\mathcal{R}(\Omega,\check{a})$ be a contracting square grid, with edge
  weights $x_1,x_2,\ldots,$ on NE-SW edges; $a_i=0$ for all $i\geq 1$; and
  $\Omega$ is an $N$-tuple of integers. Then the partition function for perfect
  matchings on $\mathcal{R}(\Omega,\mathbf{0})$ is given by
  \begin{equation*}
    Z=\left[\prod_{i=1}^{N}\prod_{j=i}^{N}(1+y_ix_{j})\right]
    s_{\omega}(x_{1},\ldots,x_{N})
  \end{equation*}
  where $\omega$ is the $N$-tuple corresponding to the boundary row of
  $\mathcal{R}(\Omega,\mathbf{0})$, and $\Gamma_i$ is defined as in \eqref{gi}.
\end{corollary}

\begin{proof}
  Note that when $\check{a}=\mathbf{0}$, the graph is a square grid. When
  $\Omega$ is an $N$-tuple of integers, we have $I_2=\{1,2,\ldots,N\}$.
  Corollary~\ref{c19} follows from Proposition~\ref{p16}.
\end{proof}

The case when all $x_i$ and $y_i$ are 1 is the one covered by~\cite{bk}.

\subsubsection{Hexagon Lattice}

\begin{corollary}
  \label{ch}
  Let $\mathcal{R}(\Omega,\check{a})$ be a contracting hexagonal lattice such
  that $a_i=1$ for all $i\geq 1$ and $\Omega$ is an $N$-tuple of integers, with
  edge weights $x_1,x_2,\ldots,$ on NE-SW edges. Then  the partition function
  for perfect matchings on $\mathcal{R}(\Omega,\mathbf{1})$ is given by
  \begin{equation*}
    Z= s_{\omega}(x_{1},\ldots,x_{N})
  \end{equation*}
  where $\omega$ is the $N$-tuple corresponding to the boundary row of
  $\mathcal{R}(N,\Omega,m)$, and $\Gamma_i$ is defined as in \eqref{gi}.
\end{corollary}

\begin{proof}
  Note that when $\check{a}=\mathbf{1}$, the graph is a hexagon lattice. When
  $\Omega$ is an $N$-tuple of integers, we have $I_2=\emptyset$.
  The Corollary~\ref{ch} follows from Proposition~\ref{p16}.
\end{proof}

The case when all the weights $x_i$ are equal to 1 is the context
of~\cite{LP12}, although the results there were obtained through a
$q$-deformation of the measure by setting $x_i=q^{-i}$ and taking the limit
$q\to 1$.

\begin{figure}
  \includegraphics[width=8cm]{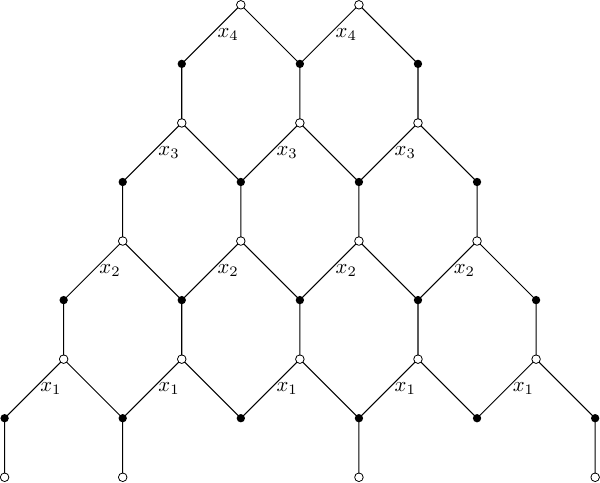}
\caption{Contracting hexagon lattice with $N=4$, $m=2$, $\Omega=(1,2,4,6)$, and
$a_i=1$.}
\label{fig:HD}
\end{figure}

\subsubsection{A Square-Hexagon Lattice}

\begin{example}
  The partition function of dimer configurations on a square-hexagon lattice as
  illustrated in Figure~\ref{fig:SH} is
\begin{multline}
  \label{zp}
  Z=(1+y_2x_2)(1+y_2x_3)\\
  \times[x_1^3x_2+x_1^3x_3+x_1x_2^3+x_1x_3^3+x_2^3x_3+x_2x_3^3\\
  +x_1^2x_2^2+x_1^2x_3^2+x_2^2x_3^2+2x_1x_2x_3(x_1+x_2+x_3)].
\end{multline}
  Indeed, in the graph shown in Figure~\ref{fig:SH}, we have $I_2=\{2\}$
  and the boundary signature is $\omega=(3,1,0)$.
  Then the partition function can be computed by applying Proposition~\ref{p16}. More precisely
  \begin{equation*}
    Z=(1+y_2x_2)(1+y_2x_3)s_{\omega}(x_1,x_2,x_3).
  \end{equation*}
  Expanding $s_{\omega}(x_1,x_2,x_3)$, we obtain exactly~\eqref{zp}.
\end{example}

For this case, as well as the previous cases, an alternative way to derive the
partition function would be to apply Kasteleyn--Percus
theory~\cite{PWK61,Percus} and write it as the determinant of a sign twisted,
weighted, bipartite adjacency matrix of the graph, and get the same polynomials.
But for this class of graphs, the machinery of symmetric functions
gives a shorter derivation of the partition function.

\begin{figure}
  \includegraphics{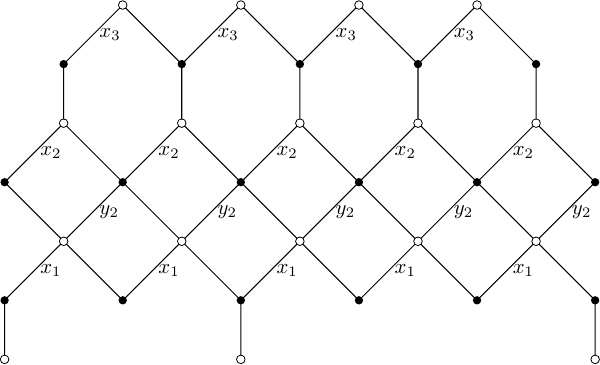}
\caption{Contracting square-hexagon lattice with $N=3$, $m=3$, $\Omega=(1,3,6), (a_1,a_2,a_3)=(1,0,1)$.}
\label{fig:SH}
\end{figure}

\subsection{Convergence of the free energy}

We state now a result about the asymptotic behavior of the partition function Z
of the dimer model on contracting square-hexagon graphs (defined in
Proposition~\ref{p16}), subject to some regularity
for the sequence of signatures describing the boundary of the graph.
This partition function then grows exponentially with $N^2$, where $N$ is the
size of the graph, and the exponential
growth rate
\begin{equation*}
  \lim_{N\to\infty} \frac{1}{N^2} \log Z
\end{equation*}
is called the \emph{free energy}.

Let us introduce first some definition to state the hypotheses for the
convergence result:

Let $\lambda\in\GT_N$ be a non-negative signature. We define the
\emph{counting measure} $m(\lambda)$ corresponding to $\lambda$ as follows:
\begin{equation*}
m(\lambda)=\frac{1}{N}\sum_{i=1}^{N}\delta\left(\frac{\lambda_i+N-i}{N}\right).
\end{equation*}

Let $\rho$ be a probability measure on the set $\GT_N$ of all signatures. The
push-forward of $\rho$ with respect to the map $\lambda\mapsto m(\lambda)$
defines a random probability measure on $\RR$ denoted by $m(\rho)$.

One natural setting for which one can prove that the free energy exists is when
the corresponding sequence of signatures describing the boundary of our sequence
of contracting square-hexagon graphs is \emph{regular}~\cite{VP15}, in the
following sense:

\begin{definition}[\cite{VP15}]
  \label{df23}
  A sequence of signatures $\lambda(N)\in\GT_N$ is called \emph{regular}, if there
  exists a piecewise continuous function $f(t)$ and a constant $C>0$ such that
  \begin{equation*}
    \lim_{N\rightarrow\infty} \frac{1}{N}
    \sum_{j=1}^{N}
    \left|\frac{\lambda_j(N)}{N}-f\left(\frac{j}{N}\right)\right| = 0,
    \quad
    \text{and}
    \quad
    \sup_{1\leq j\leq N}
    \left|\frac{\lambda_j(N)}{N}-f\left(\frac{j}{N}\right)\right|<C\ \text{for
    all $N\geq1$}.
  \end{equation*}
\end{definition}

Since the renormalized logarithm of the $\Gamma_i$ factors have a simple
limit, the existence and the value of the free energy is determined by the
existence of the logarithm of the renormalized Schur function.

For any positive integer $j\in \NN$, let $\ol{j}=j\mod n$.

\begin{proposition}[Existence of the normalized free energy in the periodic
  case]
  \mbox{}\\
  \label{lmo} 
  Suppose that the following two conditions hold: 
  \begin{itemize}
    \item $\{\lambda(N)\}_{N\in \NN}$ is a regular sequence of signatures,
    \item as $N\rightarrow\infty$, $m(\lambda(N))$ converges weakly to a
      probability measure $\bm$ on $\RR$.
  \end{itemize}
  Then: 
  \begin{enumerate}
    \item For each $N$, $s_{\lambda(N)}(1,\ldots,1)\geq 1$.
    \item For any $\beta=(\beta_1,\dotsc,\beta_n)\in\mathbb{R}^n$, and any
      sequence $\beta^{(N)}=\left((\beta_1^{(N)},\ldots,\beta_n^{(N)})\right)_N$
      converging to $\beta$,
      the limit
      \begin{equation}
        \lim_{N\rightarrow\infty}\frac{1}{N^2}\log\frac{s_{\lambda(N)}\left(\beta_{\ol{1}}^{(N)},\ldots,\beta_{\ol{N}}^{(N)}\right)}{s_{\lambda(N)}(1,\ldots,1)}
	\label{n2l}
      \end{equation}
      exists, and depends only on the limit $\beta$. In
      particular,
      \begin{equation}
        \lim_{N\rightarrow\infty}\frac{1}{N^2}\log\frac{s_{\lambda(N)}\left(\beta_{\ol{1}}^{(N)},\ldots,\beta_{\ol{N}}^{(N)}\right)}{s_{\lambda(N)}(1,\ldots,1)}=\lim_{N\rightarrow\infty}\frac{1}{N^2}\log\frac{s_{\lambda(N)}(\beta_{\ol{1}},\ldots,\beta_{\ol{N}})}{s_{\lambda(N)}(1,\ldots,1)}
        \label{bn}.
      \end{equation}
  \end{enumerate}
\end{proposition}
\begin{proof}
  By Corollary~\ref{ch}, $s_{\lambda(N)}(1,\ldots,1)$ is the total
  number of dimer configurations on a contracting hexagonal lattice, in which
  the boundary configuration is given by $\lambda(N)$. Since there exists at
  least one dimer configuration on each such lattice, we obtain the first part.

  The existence of the limit is a consequence of the successive application of
  two lemmas stated below: first
  Lemma~\ref{hciz} expressing the Schur function as a matrix integral, then
  Lemma~\ref{lgz} about limit of normalized logarithms of these integrals.

    The convergence condition (2) implies that the empirical measures
    \begin{equation*}
      \frac{1}{N}\sum_{i=1}^{N}\delta_{\beta_{\bar{i}}^{(N)}}\quad
    \text{and}
      \frac{1}{N}\sum_{i=1}^{N}\delta_{\beta_{\bar{i}}}
    \end{equation*}
    converge to the same measure. Thus by Lemma~\ref{lgz}, the renormalized
    logarithms of matrix integrals, and hence Schur functions, converge and have
    the same limit.
\end{proof}

Here are the two lemmas needed to conclude the proof of the previous
proposition.
The first one represents the Schur function as a matrix integral over the
unitary group, the so-called Harish--Chandra--Itzykson--Zuber integral:
\begin{lemma}[\cite{hc,iz}]
  \label{hciz}
  Let $\lambda\in \GT_N^{+}$ be a non-negative signature, and let $B$ be an
  $N\times N$ diagonal matrix given by
\begin{equation*}
B=\operatorname{diag}[\lambda_1+N-1,\ldots,\lambda_j+N-j,\ldots,\lambda_N+N-N]
\end{equation*}
Let $(a_1,a_2,\ldots,a_N)\in \CC^N$, and let $A$ be an $N\times N$ diagonal matrix given by 
\begin{equation*}
A=\operatorname{diag}[a_1,\ldots,a_N].
\end{equation*}
Then,
\begin{equation}
\frac{s_{\lambda}(e^{a_1},\ldots,e^{a_N})}{s_{\lambda}(1,\ldots,1)}=\prod_{1\leq i<j\leq N}\frac{a_i-a_j}{e^{a_i}-e^{a_j}}\int_{U(N)}e^{\mathrm{Tr}(U^*AUB)}dU,
\label{shc}
\end{equation}
where $dU$ is the Haar probability measure on the unitary group $U(N)$.
\end{lemma}

Note that Lemma~\ref{hciz} was originally proved when $A$ is a Hermitian matrix
with eigenvalues $a_1,a_2,\ldots,a_N$, hence $(a_1,a_2,\ldots,a_N)\in\RR^N$.
Since the right hand side of \eqref{shc} depends only on the
eigenvalues of $A$ when $A$ is a Hermitian matrix, the identity \eqref{shc} is
then true for $A=\mathrm{diag}[a_1,\ldots,a_N]$ with complex entries as well
since both the left hand side and the right hand side in \eqref{shc} are entire
functions in $a_1,\ldots,a_N$.

The second lemma is about the convergence of the normalized logarithms of these
integrals:

\begin{lemma}[\cite{GZ02}, Theorem~1.1]
  \label{lgz}
  For an $N\times N$ Hermitian matrix $A$ with eigenvalues $(a_1,\ldots,a_N)$, we denote by 
\begin{eqnarray*}
\mathbf{m}_A^N=\frac{1}{N}\sum_{i=1}^{N}\delta_{a_i}.
\end{eqnarray*}
the spectral measure for $A$. Let $\{D_N\}_{N\in \NN}$, $\{E_N\}_{N\in \NN}$ be
two sequences of diagonal, real-entry matrices, such that the three following
conditions are satisfied:
\begin{itemize}
\item there exists a compact subset $F\subset \RR$ such that $\operatorname{supp} \mathbf{m}_{D_N}^N\subseteq F$ for all $N\in \NN$;
\item $\int x^2 d\mathbf{m}_{E_N}^N$ is uniformly bounded with a bound independent of $N$;
\item $\mathbf{m}_{E_N}^N$ and $\mathbf{m}_{D_N}^N$ converge weakly towards $\mu_E$ and $\mu_D$, respectively.
\end{itemize}
Then as $N\rightarrow\infty$, 
\begin{equation*}
\frac{1}{N^2}\log\int e^{N\mathrm{tr}(UD_NU^*E_N)}dU,
\end{equation*}
has a limit depending only on $\mu_E$ and $\mu_D$.
\end{lemma}

By Weyl's formula,
\begin{equation*}
s_{\lambda(N)}(1,\ldots,1)=\prod_{1\leq i<j\leq N}\frac{(\lambda_i(N)-i)-(\lambda_j(N)-j)}{j-i}.
\end{equation*}
From Proposition~\ref{lmo} we obtain that the existence of the free energy
\begin{equation*}
\lim_{N\rightarrow\infty}\frac{1}{N^2}\log s_{\lambda(N)}\left(\beta_{\ol{1}}^{(N)},\ldots,\beta_{\ol{N}}^{(N)}\right)
\end{equation*}
 under the assumption of Proposition~\ref{lmo} depends on the existence of 
 \begin{equation*}
 \lim_{N\rightarrow\infty}\frac{1}{N^2}\log \prod_{1\leq i<j\leq
 N}\frac{(\lambda_i(N)-i)-(\lambda_j(N)-j)}{j-i},
 \end{equation*}
 which exists when the sequence of signatures is regular, and is given by the
 following integral
 \begin{equation*}
   \iint_{0<x<y<1} \log\left(1-\frac{f(y)-f(x)}{y-x}\right).
 \end{equation*}

\section{Existence of Limit Shape}\label{sec:ls}

In this section, we study the convergence of the counting measure for the signature
corresponding to the random dimer configuration on each row of the contracting
square-hexagon lattice.
We prove that the (random) moments of the counting measure converge to deterministic
quantities, for which we give an explicit formula.
This implies that the rescaled
height function associated to the random perfect matching satisfies certain
\emph{law of large numbers}, and converges to a deterministic shape in the limit.
This limit shape is also the solution of a variational problem, i.e., the unique
deterministic function that maximizes the entropy; see~\cite{ckp00}.

We will need a more refined convergence, at order $\frac{1}{N}$ instead of
$\frac{1}{N^2}$ for the free energy, and where a finite number of arguments of
the Schur function~\eqref{n2l} are allowed to vary.

If $(\lambda(N))$ is a regular sequence of signatures, then the sequence of
counting measures $m(\lambda(N))$ converges weakly to a measure $\bm$ with
compact support.
When the $\beta_i$s are equal to 1, we have by
Theorem~3.6 of~\cite{bk} that there exists an explicit function
$H_{\bm}$, analytic in a neighborhood of 1,  depending on the weak
limit $\bm$ such that
\begin{equation}
  \lim_{N\rightarrow\infty}
  \frac{1}{N} \log\left(%
  \frac{s_{\lambda(N)}(u_1,\ldots,u_k,1,\ldots,1)}{s_{\lambda(N)}(1,\ldots,1)}
  \right) = H_{\bm}(u_1)+\cdots+H_{\bm}(u_k),
  \label{nlc}
\end{equation}
and the convergence is uniform when $(u_1,\dotsc,u_k)$ is in a neighborhood
of $(1,\dots,1)$. 
Precisely, $H_{\bm}$ is constructed as follows: let
$S_{\bm}(z)=z+\sum_{k=1}^\infty M_k(\bm) z^{k+1}$ be the moment generating
function of the measure $\bm$, where $M_k(\bm)=\int x^k d\bm(x)$, and
$S_{\bm}^{(-1)}$ be its inverse for the composition. Let $R_{\bm}(z)$ be the
\emph{Voiculescu R-transform} of $\bm$ defined as
\begin{equation*}
  R_{\bm}(z) = \frac{1}{S_\bm^{(-1)}(z)} - \frac{1}{z}.
\end{equation*}
Then
\begin{equation}
  \label{hmz}
  H_{\bm}(u) = \int_{0}^{\ln u} R_\bm(t)dt+ \ln\left( \frac{\ln u}{u-1} \right).
\end{equation}
In particular, $H_{\bm}(1)=0$, and
\begin{equation*}
  H'_\bm(u) = \frac{1}{u S_\bm^{(-1)}(\ln u)} - \frac{1}{u-1}.
\end{equation*}

\begin{proposition}
  \label{p25}
  Assume that $(\lambda(N)\in \GT_N$, $N=1,2,\ldots)$ is a regular sequence of signatures, such that
\begin{equation*}
\lim_{N\rightarrow\infty}m(\lambda(N))=\bm.
\end{equation*}

Let
$\mathbf{B}^{(N)}=\left(\beta_1^{(N)},\beta_2^{(N)},\ldots,\beta_{n}^{(N)}\right)\in
\RR^{n}$, such one of the two following conditions holds:
\begin{enumerate}
\item there exists a positive constant $\alpha>0$ satisfying
\begin{equation}
\lim_{N\rightarrow\infty}\max_{1\leq i\leq N}\{\left|\beta_i^{(N)}-1\right|\}e^{N\alpha}=0,\label{bes}
\end{equation}
\item there exists a constant $C>0$ such that 
\begin{equation}
  \label{bes2}|\lambda_1(N)|\leq C;\ \text{and}\ \lim_{N\rightarrow\infty}\max_{1\leq i\leq N}\{\left|\beta_i^{(N)}-1\right|\}N=0
\end{equation}
for each $1\leq i\leq n$.
\end{enumerate}

Then for each fixed $k\in\NN$, there is a small open complex neighborhood of
$(1,\ldots,1)\in\mathbb{C}^k$, such that for $N$ large enough,
$s_{\lambda(N)}\left(u_1,\ldots,u_k,\beta_{\ol{k+1}}^{(N)},\ldots,\beta_{\ol{N}}^{(N)}\right)$
is non-zero for $(u_1,\dots,u_k)$ in this neighborhood, and the following
convergence occurs uniformly in this neighborhood:
\begin{equation*}
\lim_{N\rightarrow\infty}
\frac{1}{N}
\log\left(%
\frac{%
  s_{\lambda(N)}\left(u_1,\ldots,u_k,\beta_{\ol{k+1}}^{(N)},\ldots,\beta_{\ol{N}}^{(N)}\right)
}{%
  s_{\lambda(N)}\left(\beta_{\ol{1}}^{(N)},\ldots,\beta_{\ol{N}}^{(N)} \right)
}
\right)
=H_{\bm}(u_1)+\cdots+H_{\bm}(u_k).
\end{equation*}
\end{proposition}

\begin{proof}
  To simplify notation, we will simply write
  $s_\lambda(\mathbf{u},\mathbf{1})$ or $s_\lambda(\mathbf{u},\boldsymbol\beta)$
  for the following quantities
  $s_{\lambda(N)}(u_1,\ldots,u_k,1,\ldots,1)$ or
  $s_{\lambda(N)}(u_1,\ldots,u_k,\beta_{k+1}^{(N)},\ldots,\beta_N^{(N)})$.

  Note that
\begin{equation*}
\log\left(%
\frac{%
  s_{\lambda}(\mathbf{u},\boldsymbol\beta)
}{%
  s_{\lambda}(\boldsymbol\beta,\boldsymbol\beta)
}\right)=
D_{N,1}+D_{N,2}+D_{N,3},
\end{equation*}
where
\begin{equation}
  D_{N,1}=
  \log\left(%
  \frac{%
    s_{\lambda}(\mathbf{u},\boldsymbol\beta)
  }{%
    s_{\lambda}(\mathbf{u},\mathbf{1})
  }
  \right)\label{DN1},
  \quad
  D_{N,2}=
  \log\left(%
  \frac{s_{\lambda}(\mathbf{u},\mathbf{1})}{s_{\lambda}(\mathbf{1})}
  \right),\quad
  D_{N,3}=
  \log\left(%
  \frac{%
    s_{\lambda}(\mathbf{1},\mathbf{1})
  }{%
    s_{\lambda}(\boldsymbol\beta,\boldsymbol\beta)
  }\right).
\end{equation}
By Theorem~3.6 of~\cite{bk}, we have
\begin{eqnarray*}
  \lim_{N\rightarrow\infty} \frac{1}{N} D_{N,2}=H_{\bm}(u_1)+\cdots+H_{\bm}(u_k),
\end{eqnarray*}
and the convergence is uniform in an open complex neighborhood of $(1,\ldots,1)$.
Let us first consider the term $D_{N,3}$, where there is no dependency in
$\mathbf{u}$.
Writing that the Schur function $s_\lambda$ is a sum of
$s_\lambda(\mathbf{1})$ homogeneous monomials of
degree $|\lambda|$, we obtain that
\begin{equation*}
  \left|s_\lambda(\boldsymbol\beta)-s_\lambda(\mathbf{1})\right| \leq C
  |\lambda| \left(\sup|\beta_j^{(N)}|\right)^{|\lambda|} \max|\beta_j^{(N)}-1|.
\end{equation*}
When Hypothesis~\eqref{bes} is satisfied, one has $|\lambda|=O(N^2)$ and
$|\beta_i^{(N)}-1|=O(e^{-N\alpha})=o(N^{-2})$
Therefore, the ratio
$s_{\lambda}(\boldsymbol\beta)/s_{\lambda}(\mathbf{1})$
goes to 1 if $\beta_j^{(N)}=1+o(N^{-2})$ (uniformly
in $j$) and $|\lambda|=O(N^{-2})$, which is the case under
Hypothesis~\eqref{bes}. Therefore its logarithm converges to 0.
The conclusion can be checked in a similar fashion for the second hypothesis~\eqref{bes2}.

The convergence for $D_{N,1}$ requires a better control of the
$\beta_j^{(N)}$. Let us show that uniformly for $\mathbf{u}=(u_1,\ldots,u_k)$
in a small enough open neighborhood of 1, the ratio
\begin{equation*}
  \frac{%
    s_{\lambda}(\mathbf{u},\boldsymbol\beta)
    -
    s_{\lambda}(\mathbf{u},\mathbf{1})
  }{%
    s_{\lambda}(\mathbf{1},\mathbf{1})
  }
\end{equation*}
converges to 0 as $N$ goes to infinity under the hypothesis~\eqref{bes}.
The case~\eqref{bes2} is similar.

We use the following formula for Schur functions:
\begin{equation*}
  s_\lambda(\mathbf{u},\boldsymbol\beta)=\sum_{\mu\preceq \lambda}
  s_{\mu}(\mathbf{u})
  s_{\lambda\setminus\mu}(\boldsymbol\beta)
\end{equation*}
to write
\begin{equation*}
  s_{\lambda}(\mathbf{u},\boldsymbol\beta)
-
s_{\lambda}(\mathbf{u},\mathbf{1})
=
\sum_{\mu}
s_{\mu}(\mathbf{u})
\left(%
s_{\lambda\setminus\mu}(\boldsymbol\beta)
-
s_{\lambda\setminus\mu}(\mathbf{1})\right)
\end{equation*}

Since $s_\mu$ has only $k$ parameters, the only signatures $\mu$ contributing to
the sum have at most $k$ parts, which are at most equal to $\lambda(N)_1 \leq
cN$ for some constant $c>0$.

Fix such a signature $\mu$. The skew Schur function $s_{\lambda\setminus\mu}$ is
the sum of monomials indexed by skew semi-standard Young tableaux of shape
$\lambda\setminus\mu$. There are precisely $s_{\lambda\setminus\mu}(\mathbf{1})$ of
them. And each such monomial has degree $|\lambda\setminus\mu|$.
Since all the
$\beta^{(N)}_j$ are at distance at most $O(e^{-\alpha N})$ of 1, then the
difference of each such monomial between its evaluation at $\mathbf{1}$ and at
$(\boldsymbol\beta$ can be bounded in absolute value, for some
positive constant $C$ uniform in $N$, by
\begin{equation*}
  C|\lambda\setminus\mu| \exp{\left(C |\lambda\setminus\mu| e^{-\alpha N}\right)}
  e^{-\alpha N} = O(N^{2}e^{-\alpha N})
\end{equation*}
uniformly in $N$ and $\mu$.
Moreover, 
  $|s_{\mu}(\mathbf{u})| \leq s_{\mu}(|\mathbf{u}|)$.
Putting these pieces together, we have
\begin{equation*}
\left|
\frac{%
  s_{\lambda}(\mathbf{u},\boldsymbol\beta)
-
s_{\lambda}(\mathbf{u},\mathbf{1})
}{%
  s_{\lambda}(|\mathbf{u}|,\mathbf{1})
}
\right|
\leq \sum_{\mu} r^{cN} s_{\mu}(\mathbf{1})
\frac{s_{\lambda\setminus\mu}(\mathbf{1})}{s_{\lambda}(|\mathbf{u}|,\mathbf{1})}
C N^{2}e^{-\alpha N} \leq CN^{2}e^{-\alpha N}.
\end{equation*}
By \cite[Theorem~3.6]{bk}, the module of the ratio
$s_\lambda(\mathbf{u},1)/s_\lambda(\mathbf{|u|},\mathbf{1})$ is bounded from
below uniformly in $N$ by $e^{-AN}$ for some $A>0$, which tends to 0 as the
radius $r$ of the neighborhood around 1 goes to 0.
Therefore, this radius $r$ can be chosen close enough to 1 so that
\begin{equation*}
  \frac{%
    s_{\lambda}(\mathbf{u},\boldsymbol\beta)
    -
    s_{\lambda}(\mathbf{u},\mathbf{1})
  }{%
    s_{\lambda}(\mathbf{u},\mathbf{1})
  }
\end{equation*}
tends to his quantity goes to 0. Adding 1 and taking the logarithm gives exactly
the quantity $D_{N,1}$ which thus also tends to 0 as $N$ goes to $\infty$.
\end{proof}

\subsection{The Schur generating function and moments of counting measures}

Let 
\begin{equation*}
V(u_1,\ldots, u_N)=\prod_{1\leq i<j\leq N}(u_i-u_j)
\end{equation*}
be the Vandermonde determinant with respect to variables $u_1,\ldots,u_N$.
Introduce the family $(\mathcal{D}_k)$ of differential operators acting on
symmetric functions $f$ with variables $u_1,\ldots, u_N$ as follows:
\begin{equation}
\mathcal{D}_k f=\frac{1}{V}\left(\sum_{i=1}^{N}\left(u_i\frac{\partial}{\partial
u_i}\right)^k\right)(V\cdot f)\label{dk}.
\end{equation}

For every $\lambda\in\GT_N$,
the Schur function $s_\lambda(u_1,\dots,u_N)$ is
an eigenfunction of $\mathcal{D}_k$, associated with the
eigenvalue $\sum_i (\lambda_i+N-i)^k$, see~\cite[Proposition~4.3]{bg}.

Therefore, one can use, in analogy with monomials for integer-valued random
variables, a generating function which would give, by application of these
differential operators, information about the moments of our random distribution
of dimers on a given row of the graph.
We thus adapt slightly the definition of \emph{Schur generating functions},
introduced by Bufetov and Gorin~\cite{bg} to fit our needs:

\begin{definition}
  \label{df21}
  Let $\mathbf{B}=(\beta_1,\ldots,\beta_N)\in\CC^{N}$. Let $\rho$ be a
  probability measure on $\GT_N$. The \emph{Schur generating function}
  $\mathcal{S}_{\rho,\mathbf{B}}(u_1,\ldots,u_N)$ with respect to parameters
  $\mathbf{B}$ is the symmetric Laurent series in $(u_1,\ldots,u_N)$ given by
  \begin{equation*}
    \mathcal{S}_{\rho,\mathbf{B}}(u_1,\ldots,u_N)=
    \sum_{\lambda\in\GT_N} \rho(\lambda)
    \frac{s_{\lambda}(u_1,\ldots,u_N)}{s_{\lambda}(\beta_1,\ldots,\beta_N)}.
  \end{equation*}
\end{definition}

When $\mathbf{B}=(1,\ldots,1)$, this definition is the one found in 
Definition~4.4 of~\cite{bg} and Definition~3.1 of~\cite{bk}.

The following result states that asymptotic behavior of the Schur generating
function for random signatures implies the convergence for the associated random
counting measures, in the same line as~\cite{bk}.

\begin{lemma}[\cite{bg}, Theorem~5.1]
  \label{l28}
  Suppose that
  $\mathbf{B}^{(N)}=\left(\beta_1^{(N)},\ldots,\beta_{N}^{(N)}\right)$ satisfies
  the condition described in \eqref{bes}.
  Let $(\rho_N)_{N\geq 1}$ be a sequence of measures such that for each
  $N$, $\rho_N$ is a probability measure on $\GT_N$, and 
 for every $j$, the following convergence holds uniformly in a complex
 neighborhood of $(1,\ldots,1)\in\mathbb{C}^j$
  \begin{equation}
    \lim_{N\rightarrow\infty} \frac{1}{N}
    \log\mathcal{S}_{\rho_N,\mathbf{B}^{(N)}} \left(u_1,\ldots,u_j,\beta_{\ol{j+1}}^{(N)},\ldots,\beta_{\ol{N}}^{(N)}\right)
    = Q(u_1)+\cdots +Q(u_j),
    \label{lm}
\end{equation}
 with $Q$ an analytic function in a neighborhood of $1$.
 Then the sequence of random measures $(m(\rho_N))_{N\geq 1}$ converges as
 $N\rightarrow\infty$ in
probability in the sense of moments to a deterministic measure $\mathbf{m}$ on
$\RR$, whose moments are given by
\begin{equation*}
\int_{\RR}x^j\mathbf{m}(dx)=\sum_{l=0}^{j}\frac{j!}{l!(l+1)!(j-l)!}\left.\frac{\partial^l}{\partial
u^l}\left(u^jQ'(u)^{j-l}\right)\right|_{u=1}.
\end{equation*}
\end{lemma}

\begin{proof}
  The proof is a direct adaption of the proof of Theorem~5.1 of~\cite{bg}.
  The fact that Schur functions are eigenfunctions of the differential operators
  $\mathcal{D}_k$ allows one to rewrite the moments of the (random) moments of
  the counting measure associated to a random signature with distribution $\rho$
  on $\GT_N$ with Schur generating function
  $\mathcal{S}_{\rho,\mathbf{B}}$ as follows:
  \begin{equation*}
    \EE\left(\int_{\RR}x^km(\rho)(dx)\right)^m=
    \frac{1}{N^{m(k+1)}}
    (\mathcal{D}_k)^m\mathcal{S}_{\rho,\mathbf{B}}(u_1,\ldots,u_N)|_{(u_1,\ldots,u_N)=(\beta_1,\ldots,\beta_N)}.
  \end{equation*}

  The key is then to notice that since the convergence is uniform and $Q$ is
  analytic, then necessarily, $Q(1)=0$, as one can readily check from~\eqref{lm}
  for $j=1$.
\end{proof}

The Boltzmann probability measure from Proposition~\ref{p16} on the set of
perfect matchings of a contracting
square-hexagon lattice $\mathcal{R}(\Omega,\check{a})$ induces a measure on the
set of all possible configurations of $N-\lfloor\frac{k-1}{2}\rfloor$ V-edges in
the $k$-th row, for $k=1,\ldots 2N$, counting from the bottom. We can also think
of it as a measure $\rho^k$ on the signatures $\lambda\in
\GT_{N-\lfloor\frac{k-1}{2}\rfloor}$.

\begin{lemma}\label{l212}
  We have the following expressions for the measures $\rho^k$, depending
    on the parity of $k$:
  
  \begin{enumerate}
    \item Assume that $k=2t+1$, for $t=0,1,\ldots,N-1$, then for $\lambda\in \GT_{N-t}$
      \begin{multline*}
        \rho^k(\lambda)=
        \sum_{%
          \substack{%
            \nu^{(a)}\in \GT_a, (N-t+1\leq a\leq N);\\
            \mu^{(b)}\in\GT_b, (N-t+1\leq b\leq N-1)
          }
        }
        \prod_{i\in I_2\cap\{1,2,\ldots,t\}}\mathrm{st}_{B_i}(\mu^{(N-i+1)}\rightarrow\nu^{(N-i+1)})\\
        \times\prod_{j=1}^{t}\left[\mathrm{pr}_{C_j}(\nu^{(N-j+1)}\rightarrow\mu^{(N-j)})\right]
      \end{multline*}
      where $\mu^{(N)}=\omega$, and $\mu^{(N-t)}=\lambda$.
    \item Assume that $k=2t+2$, for $t=0,1,\ldots,N-1$, then for $\lambda\in \GT_{N-t}$
      \begin{enumerate}
        \item If $t+1\in I_2$, then
          \begin{equation*}
            \rho^k(\lambda)=
            \sum_{\mu^{(N-t)}\in \GT_{N-t}}
            \rho^{k-1}(\mu^{(N-t)}) \mathrm{st}_{B_{t+1}}(\mu^{(N-t)}\rightarrow\lambda).
          \end{equation*}
        \item If $t+1\notin I_2$, then
          $\rho^k(\lambda)=\rho^{k-1}(\lambda)$.
      \end{enumerate}
  \end{enumerate}
\end{lemma}

\begin{proof}
  In Part~(1), the expression for $\rho^k$ is obtained by taking the
  formula for the probability of a configuration~\eqref{pb} and summing over all the
  intermediate signatures corresponding to rows below or above the $k$th one.
  The Markovian structure of the probability measure implies that when fixing
  the $k$th signature, the parts below and above are independent.
  But, when we fix the $k$th row to correspond to a signature $\lambda$, all the
  signatures above correspond to a contracting square hexagon lattice with a
  boundary row given by $\lambda$. Thus the sum over the signatures above level
  $k$ is equal to 1.
  Part~(2) follows from the definition of $\mathrm{st}$ and $\rho^k$.
\end{proof}

In order to study the limit shape, we make the following assumption of
periodicity for the graph:
\begin{assumption}
  \label{ap6p}
  The square-hexagonal lattice $\mathrm{SH}(\check{a})$ is periodic with period
  $2n$. More precisely, for any integer $i,j\in \NN$, the $j$th row and the
  $(j+2ni)$th row in $\mathrm{SH}(\check{a})$ coincides and have the same edge
  weights.
\end{assumption}

Under the periodic assumption~\ref{ap6p},
the sets $I_1$ and $I_2$ defined in Definition~\ref{defI1I2} are periodic, so
$m\in I_2$ if and only if $\bar{m}(:= m \mod n)\in I_2$.
Moreover, all the independent edge weights are
the $x_i$'s for $i=1,\dots,n$, for the NE-SW edges joining the $2i$th and
$(2i+1)$th row (mod $2n$) and the $y_m$, for $m\in I_2\cap\{1,2,\ldots,n\}$, for
the NE-SW edges, joining the $(2i-1)$th and $(2i)$ rows. See
Figure~\ref{fig:SH}.
Obviously $m\in I_2$ if and only if $\bar{m}\in I_2$.

\begin{lemma}
  \label{lm212}
  For any $k$ between 0 and $2N-1$, define $t=\lfloor k/2\rfloor$, and
  let 
  \begin{equation*}
    X^{(N-t)}=(x_{\ol{t+1}},\ldots,x_{\bar{N}}),
    \quad
    \text{and}
    \quad
    Y^{(t)}=(x_{\bar{1}},\ldots,x_{\bar{t}}).
  \end{equation*}
  Then the generating Schur function $\mathcal{S}_{\rho^k,X^{(N-t)}}$ is given by:
\begin{equation*}
\mathcal{S}_{\rho^k,X^{(N-t)}}(u_1,\ldots,u_{N-t})=
\frac{s_{\omega}\left(u_1,\ldots,u_{N-t},Y^{(t)}\right)}{s_{\omega}(X^{(N)})}
\prod_{i\in\{1,\ldots,t\}\cap I_2}
\prod_{j=1}^{N-t}\left(\frac{1+y_{\bar{i}}u_j}{1+y_{\bar{i}}x_{\ol{t+j}}}\right).
\end{equation*}
\end{lemma}

\begin{proof}
  We prove the case when $k=2t+1$ is odd here; the case when $k$ is even can be
  proved similarly.

  Notice that for $1\leq i\leq N$,
$\mathrm{pr}_{C_i}(\cdot)=\mathrm{pr}_{B_i}(\cdot)$,
by the expression \eqref{d3} of $\mathrm{pr}$, and the expression of
\eqref{dbi}, \eqref{dci} for $B_i$, $C_i$, respectively.

By Definition~\ref{df21}, we have
\begin{equation}
\label{srx}
\mathcal{S}_{\rho^k,X^{(N-t)}}(u_1,\ldots,u_{N-t})=\sum_{\lambda\in\GT_{N-t}}\rho^{k}(\lambda)\frac{s_{\lambda}(u_1,\ldots,u_{N-t})}{s_{\lambda}(x_{\ol{t+1}},\ldots,x_{\bar{N}})}.
\end{equation}
Plugging the expression of $\rho^k(\lambda)$ in Lemma~\ref{l212} (1) into
\eqref{srx}, and applying \eqref{d1}, \eqref{d2}, and the previous remark about
$\mathrm{pr}_{C_i}=\mathrm{pr}_{B_i}$ sequentially, we obtain the lemma.
\end{proof}

\begin{proposition}
  \label{p213}
  Assume that the sequence of signatures $(\omega(N))_N$ corresponding the first
  row is regular, and
  $\lim_{N\rightarrow\infty}m[\omega(N)]=\mathbf{m}_{\omega}$. Assume that the
  edge weights $y_i$ for $i\in\{1,\ldots,n\}\cap I_2$ are independent of $N$,
  while the edge weights $x_1^{(N)},x_2^{(N)},\ldots, x_{n}^{(N)}$ satisfy
  \eqref{bes} or \eqref{bes2}.
  Let $(k_N)$ be a sequence of nonnegative integers such that
  $\lim_{N\to\infty}\frac{k_N}{2N} =\kappa\in[0,1]$.
  Then, the sequence of random measures $m(\rho^{k_N}(N))$ converges as
  $N\rightarrow\infty$ in probability, in the sense of moments to a
  deterministic measure $\mathbf{m}^{\kappa}$ in $\RR$, whose moments are given
  by 
  \begin{equation*}
    \int_{\RR}x^j \mathbf{m}^{\kappa}(dx)
    =\frac{1}{2(j+1)\pi\mathbf{i} }\oint_{1}\frac{dz}{z}\left(zQ'(z)+\frac{z}{z-1}\right)^{j+1},
  \end{equation*}
  where
  \begin{equation}
    Q(u)=\frac{1}{1-\kappa}H_{\mathbf{m}_{\omega}}(u)+
    \frac{\kappa}{(1-\kappa)n}
    \sum_{l=\{1,2,\ldots,n\}\cap I_2}\log\left(\frac{1+y_{l} u}{1+y_{l}}\right)
    \label{eq:defQ}
\end{equation}
and the integration goes over a small positively oriented contour around 1.
\end{proposition}

\begin{proof}From Proposition~\ref{p25} it follows that
  \begin{equation*}
    \lim_{N\rightarrow\infty}
    \frac{1}{N}\log\left(\frac{%
      s_{\omega(N)}\left(u_1,\ldots,u_j,x_{\ol{1}}^{(N)},\ldots,x_{\ol{N-j}}^{(N)}\right)
    }{%
      s_{\omega(N)}\left(x_{\ol{1}}^{(N)},\ldots,x_{\ol{N}}^{(N)}\right)}
    \right)
    = H_{\mathbf{m}_{\omega}}(u_1)+\ldots +H_{\mathbf{m}_{\omega}}(u_j).
  \end{equation*}

  Let us look at what is happening in the $k$th row, for
  $k=k(N)=2N(\kappa+o(1))$.
  Let $\rho_N^k$ be the probability measure of perfect matchings along
  the $k$th row, given that the first row has a configuration $\omega(N)$. The
  $k$th row has $(N-t)=\left(N-\lfloor\frac{k-1}{2}\rfloor\right)$ V-squares. By
  Lemma~\ref{lm212}, we have
  \begin{multline}
    \lim_{N\rightarrow\infty}\frac{1}{(1-\kappa) N} 
    \log\left(\mathcal{S}_{\rho_N^k}\left(u_1,\ldots,u_j,x_{\ol{t+1}}^{(N)},\ldots,x_{\ol{N-j}}^{(N)}\right)\right)
    = \\
\lim_{N\rightarrow\infty}\frac{1}{(1-\kappa) N}
\log\left(\frac{%
  s_{\omega(N)}\left(u_1,\ldots,u_j,x_{\ol{1}}^{(N)},\ldots,x_{\ol{N-j}}^{(N)}\right)
}{%
  s_{\omega(N)}(x_{\ol{1}}^{(N)},\ldots,x_{\ol{N}}^{(N)})}
  \!\!\!\!
\prod_{l\in\{1,2,\ldots,t/t+1\}\cap I_2}
\prod_{i=1}^{j}
\frac{1+y_{\ol{l}}u_i}{1+y_{\ol{l}}x_{\ol{N-j+i}}^{(N)}}
\right)\\
=\sum_{i=1}^{j}\left(\frac{1}{1-\kappa}H_{\mathbf{m}_{\omega}}(u_i)+\frac{\kappa}{(1-\kappa)n}\sum_{l=\{1,2,\ldots,n\}\cap
I_2}\log\left(\frac{1+y_{l} u_i}{1+y_{l}}\right)\right)
= \sum_{i=1}^j Q(u_i).
\label{cls}
\end{multline}
where in the last identity we use the assumption \eqref{bes} or \eqref{bes2},
and $n$ is a fixed finite positive integer relating to the size of the
fundamental domain.

Then the proposition follows from Lemma~\ref{l28} and (3.14) of~\cite{bk}.
The link between the expression of moments in Lemma~\ref{l28} and this
proposition is obtained by an explicit evaluation of the integral by residues.
\end{proof}

\subsection{Height function}
\label{subsec:height}

Let $\mathcal{R}(\Omega(N),\check{a})$ be a contracting square-hexagon lattice.

As for any bipartite planar graph, dimer configurations can be encoded by a
height function with values on the faces of the graph. A convenient way to construct a height function (which will help also for
the matter of discussing the scaling limit for these graphs), is to see a
contracting
square-hexagon lattice, as in fact a subgraph of the square lattice, by drawing
the \emph{vertical} edges of the hexagonal rows as diagonal NW-SE (the missing
diagonal NE-SW diagonal on these rows correspond to the fusion of two unit
squares to make hexagonal faces). The vertices of the contracting square-hexagon
graph are thus on the sublattice
$\frac{1}{2}\mathbb{Z}\times\frac{1}{2}\mathbb{Z}$ with coordinates $(i,j)$
satisfying $i+j\in\mathbb{Z}$.

The height function we will consider is then defined on the sublattice where
$i+j\in\frac{1}{2}+\mathbb{Z}$, applying the same rule as for domino
tilings~\cite{Thurston}: the height increases (resp.\@ decreases) by 1 when going
counterclockwise around a white vertex (resp.\@ black vertex), i.e., for $j$
half-integer (resp.\@ integer) as long as a dimer
is not crossed.
For definiteness, we fix the height to be 0 at the square face
$(\frac{1}{2},0)$.
When keeping $i=\frac{1}{2}$ and increasing $j$ by one, we jump over a black
vertex with two edges, with only one of them being a dimer. Then the height
increases by 2. See Figure~\ref{fig:sq_height}. Note that the height function under this definition
is a constant multiple of the height function under the classical definition in \cite{KOS06}.

\begin{figure}[htb]
  \centering
  \includegraphics[width=12cm]{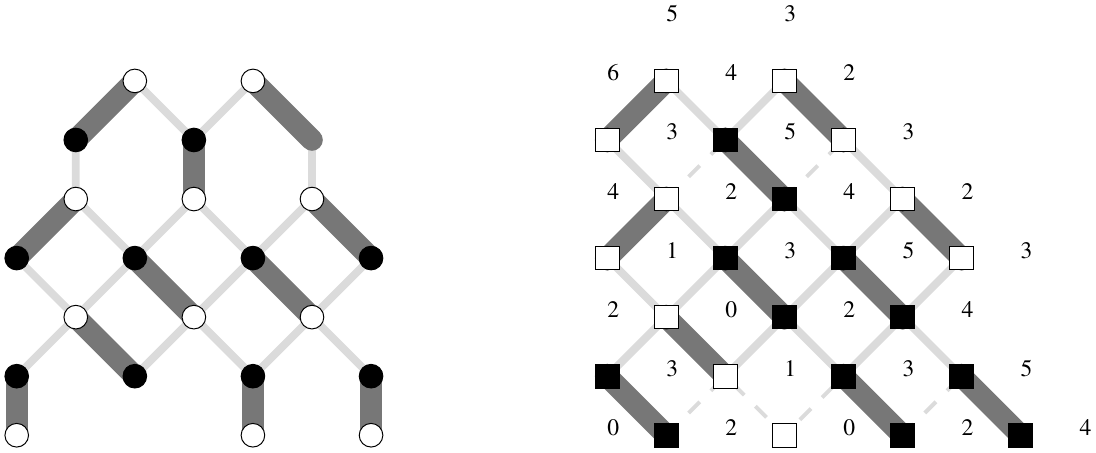}
  \caption{Two representations of a dimer configuration on a contracting square
    hexagon graph. Left: usual embedding, with the bipartite coloring of
    vertices. Right: as a subgraph of the square
  lattice, with the associated height function, and particles associated to Maya
diagrams encoding each row.}
  \label{fig:sq_height}
\end{figure}

The value of the height function can be related directly to the counting
measures of the signatures associated to rows of the graph: the height at
position $(i,j)$ is given by the following formula:
\begin{equation*}
  h(i,j)=2j-1 + 2 \operatorname{Card}\{\blacksquare\ \text{at the left of $i$}\}
- 2 \operatorname{Card}\{\square\ \text{at the left of $i$}\}.
\end{equation*}

If $\lambda$ is the signature of the $2j$th row (with ordinate $j$), the numbers
$\lambda_k-k+N-2j+1$ are the positions of the $\blacksquare$, which is (up to a
small shift), $N-2\lfloor j\rfloor$ times the atoms of the counting measure for
this signature.

Therefore, if we define $N_j=N-\lfloor j\rfloor$, for $j$-half integer, the
height along the $2j$th row, encoded by the signature $\mu^{(N_j)}$, is given by
the following formula:
\begin{equation*}
  h(i,j) = 2j-1+2 N_j
  \int_{[0,\frac{i}{2N_j})}(2\mathrm{d}m(\mu^{(N_j)})-\mathrm{dx}).
\end{equation*}

The convergence of the counting measures from Proposition~\ref{p213} implies
directly the following theorem about the convergence in probability of the
rescaled height:

\begin{theorem}[Law of large numbers for the height function]
  \label{thm:limitshape}
  Consider $N\rightarrow\infty$ asymptotics such that all the dimensions of a
  contracting square-hexagon lattice $\mathcal{R}(\Omega(N),\check{a})$ linearly
  grow with $N$.  Assume that 
 \begin{itemize}
   \item the sequence of signatures $(\omega(N))$ corresponding to the first row is regular, and
     \begin{equation*}
       \lim_{N\rightarrow\infty}m[\omega(N)]=\mathbf{m}_{\omega}
     \end{equation*}
     in the weak sense; and
   \item the edge weights are assigned as in Assumption~\ref{apew} (see
     Figure~\ref{fig:SH} for an example) and satisfy the Assumption~\ref{ap6p}
     of periodicity, such that for each $1\leq i\leq n$ and $i\in I_2$, $y_{i}>0$
     are fixed and independent of $N$, while for $1\leq i\leq n$ $x_{i}^{(N)}>0$
     satisfies \eqref{bes} or \eqref{bes2}.
 \end{itemize}

 Let $\rho_N^k$ be the measure on the configurations of the $k$th row, and let
 $\kappa\in (0,1)$, such that $j=[\kappa N]+\frac{1}{2}$, Then $m(\rho_N^k)$ converges to
 $\mathbf{m}^{\kappa}$ in probability as $N\rightarrow\infty$, and the moments
 of $\mathbf{m}^{\kappa}$ is given by Proposition~\ref{p213}.

 Define
 \begin{equation*}
   \mathbf{h}(\chi,\kappa):=2\kappa-2\chi+4(1-\kappa)\int_0^{\frac{\chi}{1-\kappa}} d\mathbf{m}^{\kappa}.
 \end{equation*}
 Then the random height function $h_M$ associated to a random perfect matching
 $M$, satisfies the following law of large  numbers: as $N$ goes to infinity,
 its scaling by a factor $N^{-1}$
 \begin{equation*}
   (\chi,\kappa)\mapsto\frac{h_{M}([\chi N],[\kappa N]+\frac{1}{2})}{N}
\end{equation*}
 converges uniformly, in probability, to the deterministic function
 $\mathbf{h}(\chi,\kappa)$,
 where $\chi, \kappa$ are new continuous parameters of the domain.
\end{theorem}

\begin{remark}
  The continuous variables $(\chi,\kappa)$ lie in a
  region $R$ of the plane obtained from $\mathcal{R}(\Omega(N),\check{a})$ by
  translating each row from the second to the left such that the leftmost vertex
  of each row are on the same vertical line; then rescaled by $\frac{1}{N}$, and
  taking the limit as $N\rightarrow\infty$. When
  $\mathcal{R}(\Omega(N),\check{a})$ is a square grid, $R$ is a rectangle;
  otherwise $R$ is a trapezoid with two right angles on the left.

  The convergence also occurs on the other sublattice, when $j$ is integer, and
  $i$ half integer, due to the fact that the discrete height function is
  Lipschitz.
\end{remark}

\section{Density of the Limit Measure}\label{s3}

We work in this section with the same hypotheses as in
Theorem~\ref{thm:limitshape}.
We obtain an explicit formula to compute the density of the
limit of the random counting measure corresponding to the random signatures for
dimer configurations on a row of the contracting square-hexagon lattice. This
formula for the density of limit measure will be used in the next section to
obtain the \emph{frozen boundary}, i.e., the frontier between \emph{frozen}
regions where the density of the limit measure is 0 or 1, and a \emph{temperate}
region where the density lies strictly between 0 and 1.

Under the assumptions above, it is not hard to check that
all the measures $\bm^\kappa$, $\kappa\in(0,1)$
have compact support, and that they is absolutely continuous with respect
to the Lebesgue measure on $\RR$ with a density taking values in $[0,1]$.

Recall that the \emph{Stieltjes transform} of a compactly supported measure $\bm$ is defined by
\begin{equation}
\mathrm{St}_{\bm}(t)=\int_{\RR}\frac{\bm(ds)}{t-s},\label{stmt}
\end{equation}
for $t\in \mathbb{C}\setminus \mathrm{Support}(\bm)$, which has an expansion as
a series in $t$ in a neighborhood of infinity whose coefficients are the moments
of the measure $\bm$: if the $M_j(\bm)$ is $j$th moment of $\bm$,
then for $t$ large enough:
\begin{equation*}
  \mathrm{St}_{\bm}(t) = \sum_{j=0}^\infty M_j(\bm) t^{-j-1}.
\end{equation*}

By Proposition~\ref{p213}, we know in principle for any $\kappa\in(0,1)$, all the moments
of the limiting measure $\bm^\kappa$, expressed in terms of the function $Q$ from
Equation~\eqref{eq:defQ}, so we can have an expression of the Stieltjes
transform of $\bm^\kappa$.
However, $Q$ depends on $H_{\mathbf{m}_\omega}$ which is itself expressed
in terms of the Stieltjes transform $\bm_\omega$. Indeed, for any measure $\bm$,
the function $H_\bm$ is is related to the Stieltjes transform of $\bm$ by the
following relation:
\begin{equation*}
  H_{\bm}'(z)
  =\frac{1}{z \mathrm{St}_{\bm}^{(-1)}(\log(z))}-\frac{1}{z-1}.
\end{equation*}
See also Equation~\eqref{hmz}.

Introducing an additional variable $t$ such that $\mathrm{St}_{\bm}(t)=\log(z)$,
one can then write for $\kappa\in(0,1)$:
\begin{equation}
F_{\kappa}(z,t):=
zQ'(z)+\frac{z}{z-1} =
\frac{z}{1-\kappa}\left(\frac{t}{z}-\frac{1}{z-1}+\frac{\kappa}{n}\sum_{i\in I_2\cap \{1,2,\ldots,n\}}\frac{y_{i}}{1+y_{i}z}\right)+\frac{z}{z-1}.
\label{fk}
\end{equation}

As a consequence, injecting the expression of the moments of the limiting
measure into the definition of the Stieltjes transform, one gets an implicit
equation to be solved:
for any $x\in\mathbb{C}$, finding $(z,t)\in
(\mathbb{C}\setminus\mathbb{R}_-)\times
(\mathbb{C}\setminus\operatorname{Support}(\bm_\omega))$ such that
\begin{equation}
  \begin{cases}
    F_{\kappa}(z,t)=x\\
    \mathrm{St}_{\mathbf{m}_{\omega}}(t)=\log(z)
  \end{cases},
  \label{es}
\end{equation}
allows one to express
$\operatorname{St}_{\bm^\kappa}$:
let $x\mapsto z^{\kappa}(x)$ be the composite inverse of 
\begin{equation*}
  u:z\mapsto
F_{\kappa}\left(z,\mathrm{St}^{(-1)}_{\mathbf{m}_{\omega}}(\log z)\right).
\end{equation*}
Note
that $z^{\kappa}(x)$ is a uniformly convergent Laurent series in $x$ when $x$ is
in a neighborhood of infinity, and
\begin{equation*}
z^{\kappa}\left(F_{\kappa}\left(z,\mathrm{St}^{(-1)}_{\mathbf{m}_{\omega}}(\log
z)\right)\right)=z.
\end{equation*}
See Section~4.1 of~\cite{bk}.

The following identity holds when $x$ is in a neighborhood of infinity
\begin{equation}
\mathrm{St}_{\mathbf{m}^{\kappa}}(x)=\log(z^{\kappa}(x)).
\label{eq:solveSt}
\end{equation}
Indeed, by Proposition~\ref{p213},
the $j$-th moment of
$\mathbf{m}^{\kappa}$ is given by
\begin{equation*}
M_j(\mathbf{m}^{\kappa})=\frac{1}{2(j+1)\pi\mathbf{i}}\oint_1\frac{dz}{z}[F_{\kappa}(z,\mathrm{St}^{(-1)}_{\mathbf{m}_{\omega}}(\log z))]^{j+1},
\end{equation*}
where the integral is along a small counterclockwise contour winding once around $1$.
Then the identity~\eqref{eq:solveSt} follows from the same computation as in the
proof of Lemma 4.1 in~\cite{bk}, by performing an integration by parts and a
change of variable from $z$ to $u$.

The first equation of the system~\eqref{es} with $F_\kappa$ from \eqref{fk} is
linear in $t$ for given $x$ and $z$, which gives
with $c_i=\frac{1}{y_i}$
the value of $t$, as a function of $z$, $\kappa$, and $x$:
\begin{multline}
t=t(z,\kappa,x)=
x(1-\kappa)+\frac{\kappa z}{z-1}-\frac{\kappa z}{n}\sum_{i\in
  I_2\cap\{1,2,\ldots,n\}}\frac{1}{c_i+z}\\
=x(1-\kappa)+\kappa\frac{r}{n}+\kappa\left(
\frac{1}{z-1}+\frac{1}{n}\sum_{i\in I_2\cap\{1,\dots,n\}}\frac{c_i}{z+c_i}
\right),
\label{tzkx}
\end{multline}
where 
\begin{equation}
  r=n-\operatorname{Card}(I_2)=\operatorname{Card}\{1,\dotsc,n\}\cap I_1.
  \label{eq:defr}
\end{equation}
For a given value $y\in\mathbb{R}$, and fixed $x$ (and $\kappa$), we investigate
properties of the complex numbers $z$ such that $t(z,\kappa,x)=y$. In
particular, we have the following:
\begin{lemma}
  \label{l32}
  Let $c_i>0$, for $i\in I_2\cap\{1,2,\ldots,n\}$. Let $\kappa\in (0,1)$, and
  $x,y\in\RR$.
  Then the following equation in $z$
\begin{equation}
  t(z,\kappa,x)=y
  \label{fe}
\end{equation}
has $m+1$ roots on the Riemann sphere $\mathbb{C}\cup\{\infty\}$, where $m$ is
the number of distinct values of $c_i$,
and all these roots are real and simple.
\end{lemma}

\begin{proof}
  Let $0<\gamma_1<\cdots <\gamma_m$ be all the possible distinct values for the $c_i$, and
  $n_1,\dotsc,n_m$ be their respective multiplicities among the $c_i$'s.
  Define
  \begin{equation}
    H(z;x,y) = t(z,\kappa,x)-y =
    K+\kappa\left(%
    \frac{1}{z-1} + \frac{1}{n}\sum_{j=1}^m \frac{n_j \gamma_j}{z+\gamma_j}
    \right)
    \label{eq:defHzxy}
  \end{equation}
  where $K=x(1-\kappa)-y+\kappa\frac{r}{n}$, with $r$ as in
  Equation~\eqref{eq:defr}.
  When writing $H(z;x,y)$ as a single rational fraction by bringing all the terms
  onto the same polynomial denominator (of degree $m+1$), the polynomial on the numerator has
  degree at most $m+1$. So there are at most $m+1$ roots in $\mathbb{C}$ (and
  exactly $m+1$ if we add roots at infinity). Notice that the denominator does
  not depend on $x$ and $y$, but just on the $\gamma_j$'s.

  Moreover, each factor of the form $\frac{1}{z-b}$ with $b=-\gamma_j$ or $1$ is
  a decreasing function of $z$ on any interval where it is defined. As a
  consequence, on each of the intervals $(-\gamma_{j+1},-\gamma_j)$,
  $j=1,\dots,m-1$ and $(-\gamma_1,1)$, $H$ realizes a bijection with
  $\mathbb{R}$. In particular, the equation $H(z;x,y)$ has a unique solution in
  every such interval. It is also decreasing on $(\infty,-\gamma_m)$  and
  $(1,+\infty)$. Since the limits of $H(z;x,y)$ when $z$ goes to $\pm\infty$
  coincide, and
  \begin{equation*}
    \lim_{z\to 1+} H(z;x,y) = +\infty,
    \qquad
    \lim_{z\to -\gamma_m^-} H(z;x,y) = -\infty,
  \end{equation*}
  the equation $H(z;x,y)=0$ has a unique solution in
  $(-\infty,-\gamma_m)\cup(1,+\infty)\cup\{\infty\}$, which is in fact infinite if and
  only if the limits of $H$ at infinity is zero, that is, when $K=0$.
  This gives thus $m+1$ real roots (with possibly one at infinity).
\end{proof}

\begin{remark}
  \label{rem:interlaceH}
  Increasing the value of $y$ translates downward the graph of the function
  $z\in\mathbb{R}\mapsto H(z;x,y)$. Since $H(z;x,y)$ is decreasing in any
  interval of definition, the roots present in the bounded intervals decrease.
  The one in $(-\infty,-\gamma_m)\cup(1,+\infty)\cup\{\infty\}$ moves also to
  the left, and if it started in $\mathbb{R}_-$, when it reaches $-\infty$, it
  jumps to the right part of $(1,+\infty)$ and then continues to decrease.
  In particular, it means that if $y<y'$, the respective roots
  $z_1<\cdots<z_{m+1}$ and $z'_1 < \cdots < z'_{m+1}$ are interlaced:
  \begin{itemize}
    \item if $y<y'<x(1-\kappa)+\frac{\kappa r}{n}$,
      \begin{equation*}
	z'_1<z_1 < -\gamma_m < z'_2< z_2 <-\gamma_{m-1} <\cdots < -\gamma_1
	<z'_{m+1} < z_{m+1} < 1,
      \end{equation*}
    \item if $y<x(1-\kappa)+\frac{\kappa r}{n}<y'$,
      \begin{equation*}
	z_1 < -\gamma_m < z'_1< z_2 <-\gamma_{m-1} <\cdots < -\gamma_1
	<z'_m < z_{m+1} < 1 < z'_{m+1},
      \end{equation*}
    \item if $<x(1-\kappa)+\frac{\kappa r}{n}<y<y'$,
      \begin{equation*}
	-\gamma_m < z'_1<z_1 < -\gamma_{m-1} < z'_2< z_2 <-\gamma_{m-1} <\cdots
	< 1 <z'_{m+1}< z_{m+1},
      \end{equation*}
  \end{itemize}
  The limiting case when $y$ or $y'$ is equal to $x(1-\kappa)+\frac{\kappa
  r}{n}$ is obtained by sending the corresponding root in
  $(-\infty,-\gamma_m)\cup(1,+\infty)$ to $\infty$.
  \end{remark}

Rational fractions where zeros of the numerator and denominator interlace have
interesting monotonicity properties, already used for example in~\cite{OR07},
which are straightforwardly checked by induction using the 
decomposition of $R(z)$ into the sum of simple fractions:
\begin{lemma}
  \label{l33}
  Let
  \begin{equation*}
    R(z)=\frac{(z-u_1)(z-u_2)\cdots(z-u_s)}{(z-v_1)(z-v_2)\cdots(z-v_s)},
  \end{equation*}
  where $\{u_i\}$ and $\{v_i\}$ are two sets of real numbers.
  \begin{itemize}
    \item If $\{u_i\}$ and $\{v_i\}$ satisfy
      \begin{equation*}
	v_1<u_1<v_2<u_2<\cdots<v_s<u_s.
      \end{equation*}
      Then $R(z)$ is monotone increasing in each one of the following intervals
      \begin{equation*}
	(-\infty,v_1), (v_1,v_2),\ldots, (v_{s-1},v_s),(v_s,\infty).
      \end{equation*}
    \item If $\{u_i\}$ and $\{v_i\}$ satisfy
      \begin{equation*}
	u_1<v_1<u_2<v_2<\cdots<u_s<v_s.
      \end{equation*}
      Then $R(z)$ is monotone decreasing in each one of the following intervals
      \begin{equation*}
	(-\infty,v_1), (v_1,v_2),\ldots, (v_{s-1},v_s),(v_s,\infty).
      \end{equation*}
  \end{itemize}
  The monotonicity on each interval of definition is still true if $R$ has the
  form:
  \begin{equation*}
    R(z)=\frac{(z-u_1)\cdots(z-u_{s-1})}{(z-v_1)\cdots(z-v_s)}
    \text{with}\ v_1<u_1<\cdots<u_{s-1}<v_s
  \end{equation*}
  or
  \begin{equation*}
    R(z)=\frac{(z-u_1)\cdots(z-u_{s+1})}{(z-v_1)\cdots(z-v_s)}
    \text{with}\ u_1<v_1<\cdots<v_{s}<u_{s+1}.
  \end{equation*}
\end{lemma}

This is helpful to determine the number of solutions of Equation~\eqref{es}, as
shown in the following lemma:

\begin{lemma}
  \label{l34}
  Suppose that $\mathbf{m}_{\omega}$ is a measure with a density with respect to the
  Lebesgue measure equal to the indicator of a union of intervals
  $\bigcup_{i=1}^s[a_i,b_i]$, with
\begin{equation*}
a_1<b_1<a_2<b_2<\cdots<a_s<b_s\quad \text{and}\quad \sum_{i=1}^{s}(b_i-a_i)=1.
\end{equation*}
Then the system of equations \eqref{es} has at most one pair of complex
(non real) conjugate solutions. Moreover,
\begin{itemize}
  \item if $b_i \neq x(1-\kappa)+\frac{\kappa r}{n}$, for all $1\leq i\leq s$,
    then for each fixed $x\in \RR$, \eqref{es} has at least $(m+1)s-1$ distinct
    real roots;
  \item if $b_i = x(1-\kappa)+\frac{\kappa r}{n}$ for some $i$ in
    $\{1,2,\ldots,s\}$, then for each fixed $x\in \RR$, \eqref{es} has at least
    $(m+1)s-2$ distinct real roots.
\end{itemize}
where $m$ is the number of distinct $c_i$s.
\end{lemma}

\begin{proof}
  The Stieltjes transform can be computed explicitly from the definition:
  \begin{equation*}
    \mathrm{St}_{\bm_{\omega}}(t)=\log\prod_{i=1}^{s}\frac{t-a_i}{t-b_j}.
  \end{equation*}
  We use the second expression from \eqref{tzkx} to substitute $t(z,\kappa,x)$ in
  the second equation of \eqref{es}, to get (after exponentiation)
  \begin{equation}
    z=\prod_{i=1}^s \frac{H(z;x,a_i)}{H(z;x,b_i)}.
    \label{gzx}
  \end{equation}
Let us suppose that none of the $a_i$'s or $b_i$'s is equal to
$x(1-\kappa)+\frac{\kappa r}{n}$.
The rational fractions $\prod H(z;x,a_i)$ and $\prod H(z;x,b_i)$ have the same
poles $m+1$ poles (of same order $s$) and according to Lemma~\ref{l32} have
$s(m+1)$ distinct real roots, which by Remark~\ref{rem:interlaceH}, interlace.
Therefore, the ratio:
\begin{equation*}
  \prod_{i=1}^s \frac{H(z;x,a_i)}{H(z;x,b_i)}
\end{equation*}
is a rational fraction of the form described in the hypotheses of
Lemma~\ref{l33}. Therefore, on each bounded interval between two consecutive
poles, by monotonicity, the graph of the rational fraction will cross the first
diagonal exactly once and there are $(m+1)s-1$ such intervals.

If (no $b_i$, and exactly) one $a_i$ is equal to $x(1-\kappa)+\frac{\kappa
r}{n}$, the same argument is
applicable. The only difference is that the rational fraction on the right hand
side of Equation~\ref{gzx} has only $(s-1)(m+1)+m=s(m+1)-1$ zeros, but still
$s(m+1)$ poles. Therefore we still get the same number $s(m+1)-1$ of
intersection with the first diagonal, one on each finite interval between two
consecutive poles.

If (no $a_i$ and exactly) one $b_i$ is equal to $x(1-\kappa)+\frac{\kappa
r}{n}$, then this time the rational fraction has $s(m+1)-1$ finite real poles. Therefore,
there is only $s(m+1)-2$ roots found by this approach between two successive
poles.
\end{proof}

\begin{remark}
  Note that when rewriting Equation~\ref{gzx} as a polynomial equation in $z$,
  it has degree
  \begin{equation*}
    \begin{cases}
      s(m+1)+1 & \text{when no $b_i$ equals $x(1-\kappa)+\frac{\kappa
      r}{n}$},\\
      s(m+1) & \text{when a $b_i$ is equal to $x(1-\kappa)+\frac{\kappa r}{n}$}.
    \end{cases}
  \end{equation*}
  Indeed, in the last case, the leading coefficients of the numerator and
  denominator of the rational fraction are distinct, thus there is no
  cancellation of the monomials of higher degree when multiplying both sides by
  the denominator. In both case, it is exactly the number of real roots we found
  plus 2. Which means that Equation~\ref{gzx}, and thus~Equation~\ref{es} has at
  most a pair of complex conjugated roots.
\end{remark}

When there are complex conjugated roots, the density of the counting measure is
given by their the normalized argument as stated in the Theorem~\ref{tm38}
below. In order to prove it, we need two additional lemmas, which are minor
adaptations of the ones of~\cite{bk} adapted to our context:

\begin{lemma}
  \label{l35}
  Let $x_0\in \RR$ be such that Equation~\eqref{es} has a pair of complex roots. Let
  $z_j(x_0)$ be the $j$th smallest real root of \eqref{es}. 
Let $\mathcal{U}$ be a small complex neighborhood of $x_0$. Let $z_j(x)$ be the
root of \eqref{es} approximating $z_j(x_0)$, when $x\rightarrow x_0$ (which is
well defined if $\mathcal{U}$ is small enough). 
Then the derivative of $z_j(x)$ with respect to $x$ at $x_0$ is non-negative.
Moreover, it is equal to 0 if and only if $z_j(x_0)=1$.
\end{lemma}
\begin{proof}Write
\begin{eqnarray*}
  G(z,x)=\prod_{i=1}^s \frac{H(z;x,a_i)}{H(z;x,b_i)}=\frac{\prod_{i=1}^s(x(1-\kappa)+\frac{\kappa z}{z-1}-\frac{\kappa z}{n}\sum_{i=1}^{m}\frac{1}{z+c_i}-a_i)}{\prod_{i=1}^s(x(1-\kappa)+\frac{\kappa z}{z-1}-\frac{\kappa z}{n}\sum_{i=1}^{m}\frac{1}{z+c_i}-b_i)},
\end{eqnarray*}
then follow the same argument as in the proof of Lemma 4.5 of~\cite{bk} to show $G'_x(z(x_0),x_0)\leq 0$ and $G'_z(z(x_0),x_0)> 1$. Then the lemma follows.
\end{proof}

\begin{lemma}
  \label{l37}
  Let $\mathbf{m}_{\omega}$ be a measure as described in Lemma~\ref{l34}.
Let
\begin{equation*}
\mathbf{Z}^{\kappa}(x)=\mathrm{exp}(\mathrm{St}_{\mathbf{m}^{\kappa}}(x)),
\end{equation*}
where $x\in \mathbb{C}\setminus \mathrm{Support}(\mathbf{m}^{\kappa})$.
Then $\mathbf{Z}^{\kappa}(x)$ is a solution of \eqref{es}.
Let $x_0$ be such that \eqref{es} has exactly one pair of non-real solutions.
Then $\lim_{\epsilon\rightarrow 0+}\mathbf{Z}^{\kappa}(x_0+i\epsilon)$ coincides
with the unique non-real root with negative imaginary part of \eqref{es} when
$x=x_0$.
\end{lemma}
\begin{proof}
  The fact that $\mathbf{Z}^{\kappa}(x)$ solves \eqref{es} follows
  Equation~\eqref{eq:solveSt}. The limit $\lim_{\epsilon\rightarrow
  0+}\mathbf{Z}^{\kappa}(x_0+i\epsilon)$ has strictly negative imaginary part
  follows from the analysis on Page 24 of~\cite{bk}, by applying Lemma~\ref{l35}.
\end{proof}

Here is the main theorem to be proved in this section. 

\begin{theorem}\label{tm38}The density of $\mathbf{m}^k$ is given by
\begin{eqnarray*}
d\mathbf{m}^{\kappa}(x)=\frac{1}{\pi}\mathrm{Arg}(\mathbf{z}_+^{\kappa}(x)),
\end{eqnarray*}
where $\mathbf{z}_+^{\kappa}(x)$ is the unique complex root of the system of equations \eqref{es}
which lies in the upper half plane. If such a complex root does not exist, the density is equal to 0 or 1.
\end{theorem}
\begin{proof}
  First we assume that $\mathbf{m}_{\omega}$ is a uniform measure with density
  one on a sequence of intervals, as described in Lemma~\ref{l34}. In this case,
  the theorem follows from Lemma~\ref{l37}, and from the classical fact about
  Stieltjes transform that if a measure $\mu$ has a continuous density $f$ with
  respect to the Lebesgue measure then one can reconstruct $f$ by the following
  identity (See e.g.~\cite[Lemma~4.2]{bk}):
  \begin{equation}
    \label{l36}
    f(x)=-\lim_{\epsilon\rightarrow 0+}\frac{1}{\pi}\Im(\mathrm{St}_{\mu}(x+i\epsilon)),
  \end{equation}
  where $\Im$ denote the imaginary part of a complex number.

In the general case of a measure $\mathbf{m}_{\omega}$, there exists a sequence
of measures $\{\mu_i\}$ converging weakly to $\mathbf{m}_{\omega}$, where each
$\mu_i$ is a measure with the form as described in Lemma~\ref{l34}. Passing to
the limit we obtain the theorem.
\end{proof}

\section{Frozen Boundary}\label{sec:fb}

In this section, we study the \emph{frozen boundary}, which is the curve
separating the ``liquid region'' and the
``frozen region'' in the scaling limit of dimer models on a
contracting square-hexagon lattice. We prove an explicit formula of the frozen
boundary under the assumption that each segment of  the boundary row of the
square-hexagon lattice grows linearly with respect the dimension of the graph.
We then prove that the frozen boundary is a curve of a certain type, more
precisely, a cloud curve whose class depends on the size of the fundamental
domain and the number of segments of the boundary row. Similar results for dimer
configurations on the square grid or the hexagonal lattice with uniform measure
or a $q$-deformation of the uniform measure was obtained in~\cite{KO07,bk}.

We consider special sequences of contracting square-hexagon lattices 
$\mathcal{R}(\Omega,\check{a})$ with
\begin{multline}
\label{cc1}
\Omega=
(A_1,A_1+1,\ldots,B_1-1,B_1,A_2,A_2+1,\ldots,
B_2,\ldots,A_s,A_s+1,\ldots,
B_s),
\end{multline}
where
\begin{eqnarray*}
\sum_{i=1}^{s}(B_i-A_i+1)=N.
\end{eqnarray*}
In other words, $\Omega$ is an $N$-tuple of integers whose entries take values
of all the integers in $\cup_{i=1}^s[A_i,B_i]$.

We shall consider $\Omega(N)$
changing with $N$, and discuss the asymptotics of the frozen boundary as
$N\rightarrow\infty$. 

Suppose that for each $N$, $\Omega(N)$ has corresponding $A_i(N)$, $B_i(N)$, for
a fixed $s$.
Assume also that $A_i(N),B_i(N),\Omega(N)_N-N$ have the following asymptotic
growth:
\begin{equation}
\label{cc2}
A_i(N)=a_i N+o(N),\ \ B_i(N)=b_iN+o(N),\  \Omega(N)_N-N=\mu N+o(N)
\end{equation}
where $a_1<b_1<\ldots<a_s<b_s$ are new parameters such that $\sum_{i=1}^{s}(b_i-a_i)=1$. 

\begin{definition}
  \label{df41}
  Let $\mathcal{L}$ be the set of $(\chi,\kappa)$ inside $\mathcal{R}$ such that
  the density $d\mathbf{m}^{\kappa}\left(\frac{\chi}{1-\kappa}\right)$ is not
  equal to 0 or 1. Then $\mathcal{L}$ is called the \emph{liquid region}. Its boundary
  $\partial \mathcal{L}$ is called the \emph{frozen boundary}.
\end{definition}

\begin{theorem}
  \label{fb}
  The frozen boundary $\partial\mathcal{L}$ of the limit of a contracting square-hexagon lattice
  satisfying \eqref{cc1}, \eqref{cc2} is a rational algebraic curve $C$ with an
  explicit parametrization $(\chi(t),\kappa(t))$ defined as follows:
  \begin{equation*}
    \chi(t)=t-\frac{J(t)}{J'(t)},\quad
    \kappa(t)=\frac{1}{J'(t)},
  \end{equation*}
  where
  \begin{multline}
    J(t)=\Phi_s(t)\left[\frac{1}{\Phi_s(t)-1}-\frac{1}{n}\sum_{i\in
    I_2\cap\{1,2,\ldots,n\}}\frac{1}{\Phi_s(t)+c_i}\right]=\\
    \frac{r}{n}+\frac{1}{\Phi_s(t)-1} + \frac{1}{n}\sum_{j=1}^m \frac{n_j
    \gamma_j}{\Phi_s(t)+\gamma_j}
    \label{djt}
  \end{multline}
  and
  \begin{equation*}
    \Phi_s(t)=\frac{(t-a_1)(t-a_2)\cdots(t-a_s)}{(t-b_1)(t-b_2)\cdots(t-b_s)}
  \end{equation*}
\end{theorem}

\begin{proof}
  According to the discussions in Section~\ref{s3}, the frozen boundary is given
  by the condition that the following equation in the unknown $z$ has a double root:
  \begin{equation}
    G\left(z,\frac{\chi}{1-\kappa}\right)=z;\label{gzck}
  \end{equation}
  where 
  \begin{equation}
    G\left(z,\frac{\chi}{1-\kappa}\right)
    =\prod_{i=1}^s
    \frac{H(z;\frac{\chi}{1-\kappa},a_i)}{H(z;\frac{\chi}{1-\kappa},b_i)},
  \end{equation}
  and $H(z;x,y)$ is defined by Equation~\eqref{eq:defHzxy}.

  We can also rewrite the system of equations~\eqref{es} as follows:
  \begin{equation*}
      \begin{cases}
	\Phi_s(t)&=z;\\
	(1-\kappa)F_\kappa(z) &=
	t-\kappa\Bigl(
	\underbrace{%
	  \frac{r}{n}+\frac{1}{z-1}+\sum_{j=1}^m
      \frac{n_j\gamma_j}{z+\gamma_j}}_{J(t)}\Bigr) 
		 =\chi.
      \end{cases}
    \end{equation*}
    We plug the expression of $z$ from the first equation into the second
    equation, and note that the condition that the resulting equation has a
    double root is equivalent to the following system of equations
    \begin{equation*}
      \begin{cases}
	\chi=t-\kappa J(t),\\
	1=\kappa J'(t).
      \end{cases}
    \end{equation*}
    where $J(t)$ is defined by \eqref{djt}. Then the parametrization of the
    frozen boundary follows.
  \end{proof}

  The algebraic curve we obtain for the frozen boundary has special properties,
  that can be read from its dual curve, 
  as described in the definition and the proposition below:

\begin{definition}[\cite{KO07}]
  \label{df43}
  A degree $d$ real algebraic curve $C\subset \RR P^2$ is \emph{winding} if the
  following two conditions hold:
  \begin{enumerate}
    \item it intersects every line $L\subset \RR P^2$ in at least $d-2$ points
      counting multiplicity,
    \item there exists a point $p_0\in \RR P^2$ called center, such that every
      line through $p_0$ intersects $C$ in $d$ points.
  \end{enumerate}
  The dual curve of a winding curve is called a \emph{cloud curve}.
\end{definition}

\begin{proposition}
  \label{prop:cloud1}
  The frozen boundary $C$ is a cloud curve of class $(m+1)s$, where $s$ is the
  number of segments, and $m$
  is the number of distinct values of $c_i=\frac{1}{y_i}$ in one period.
  Moreover, the curve $C$ is tangent to the following lines in the
  $(\chi,\kappa)$ coordinates:
\begin{equation*}
  \mathcal{L}=\{\chi=a_i|i=1,\ldots,s\}\cup\left\{\chi+ r\kappa-b_i=0|1\leq i\leq s\right\}\cup\{\kappa=0\}\cup\{\kappa=1\}.
\end{equation*}
\end{proposition}

So the proposition states that the dual curve is winding of degree $(m+1)s$, and
passes through the points $(-\frac{1}{a_i},0)$ and
$(-\frac{1}{b_i},-\frac{r}{nb_i})$.

The result about the frozen boundary being a cloud curve extends the result of
Kenyon and Okounkov~\cite{KO07} for the uniform measure of rhombus tilings of
polygonal domains, and Bufetov and Knizel~\cite{bk} for Aztec rectangles.

\begin{proof}
  We recall that the class of a curve is the degree of its dual curve. So we
  need to show that the dual curve $C^{\vee}$ has degree $(m+1)s$ and is
  winding.

  We apply the classical formula to obtain from a parametrization $(x(t),
  y(t))$ of the curve $C$ defining the frozen boundary, another
  one for its dual $C^\vee$, 
    $(x^\vee(t),y^\vee(t))$:
    \begin{equation*}
      x^{\vee}=\frac{y'}{yx'-xy'},\quad
      y^{\vee}=-\frac{x'}{yx'-xy'}.
    \end{equation*}
    and obtain that 
    the dual curve $C^{\vee}$ is
  given in the following parametric form:
  \begin{equation}
    C^{\vee}=\left\{\left(-\frac{1}{t},-\frac{J(t)}{t}\right)\ ;\
    t\in\mathbb{C}\cup\{\infty\}\right\}.
    \label{dual}
  \end{equation}
  from which we can read that its degree is $(m+1)s$. To show that $C^{\vee}$ is
  winding, we need to look at real intersections with straight lines.

  First, from Equation~\eqref{dual}, one sees that the first coordinate $x^\vee$ of
  the dual curve $C^\vee$ and the parameter $t$ are linked by the simple
  relation $x^\vee t=-1$.
  
  Using this relation to eliminate $t$ from the expression of the second
  coordinate, we obtain that the points $(x^\vee,t)$ on the dual curve satisfy the
  following implicit equation:
  \begin{equation*}
    y^{\vee}=x^{\vee} J(\frac{-1}{x^\vee}).
  \end{equation*}

  The points of intersection $(x^\vee(t),y^\vee(t))$ of the dual curve with a straight line
  of the form $y^\vee=cx^\vee+d$ have a parameter $t$ satisfying:
  \begin{equation}
    (c-dt)=J(t)
    \label{eq:intersect_dual}
  \end{equation}
  but since $J$ is the composition of two rational fractions $\Phi_s$ and an
  affine transformation of $H$, with interlacing poles and zeros, with degrees
  $s$ and $m+1$ respectively, the exact same argument as in Lemma~\ref{l35}
  (but with the role of $s$ and $(m+1)$ exchanged)
  shows that Equation~\ref{eq:intersect_dual} has at least $(m+1)s-2$
  distinct real solutions, yielding $(m+1)s-2$ points of intersections for the
  dual curve and the line $y^\vee=cx^\vee+d$.
  Moreover, if $t_0$ doesn't lie in a compact interval containing all the zeros of
  $J$, then any non vertical straight line passing through $(t_0,0)$ will have
  $(m+1)s-1$ intersections with the graph of $J$. See Figure~\ref{fig:plotJ} for
  the graph of $J(t)$.  
  This means that for $x_0^\vee$ in some
  closed interval, there are at least
  $(m+1)s-1$ real intersections of the dual curve with a line $y^\vee=cx^\vee+d$ passing
  through $(x_0^\vee,y_0^\vee)$, thus exactly $(m+1)s$ real intersections, since there
  cannot be a single complex one.
  Such points $(x_0^\vee,y_0^\vee)$ are candidates to be the center of the dual curve.

  \begin{figure}[h]
    \centering
    \includegraphics[width=8cm]{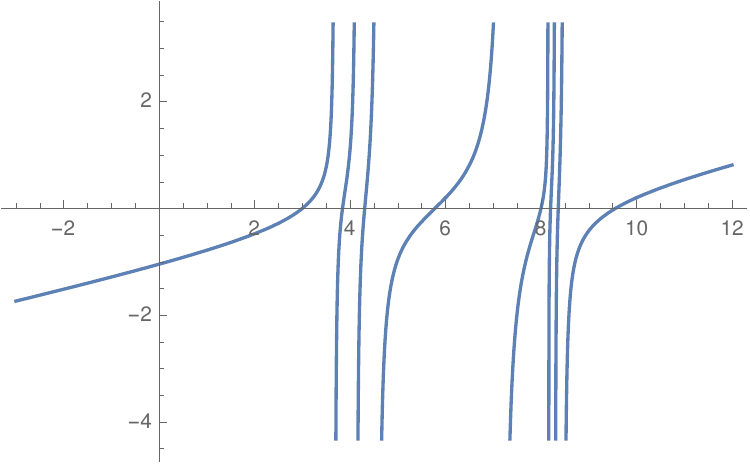}
    \caption{A plot of the graph of the function $J$, for the parameters $r=1$,
      $n=5$, $s=2$, $m=3$, $(n_1,n_2,n_3)=(2,1,1)$,
      $(\gamma_1,\gamma_2,\gamma_3)=(0.7,0.4,0.2)$ and $(a_1,a_2)=(3,8)$,
    $(b_1,b_2)=(6,10)$.}
    \label{fig:plotJ}
  \end{figure}

  To consider the vertical lines $x^\vee=d$, we rewrite the equations in homogeneous
  coordinates $[x^\vee:y^\vee:z^\vee]$ and get that the line $x^\vee=dz^\vee$ intersects the curve at the
  point $[0:1:0]$ with multiplicity $(m+1)s-1$, so again, by the same argument
  as above, $(m+1)s$ real intersections. The case of the line $z^\vee=0$ is similar.

Recall that each point on the dual curve $C^{\vee}$ corresponds to a tangent
line of $C$. The points $(x,y)=\left(-\frac{1}{a_i},0\right)$ belong to
$C^\vee$. Indeed, they correspond to $t=a_i$ which is a zero of $\Phi_s$, and
thus of $J$, by \eqref{djt}.
Similarly, $(x,y)=\left(-\frac{1}{b_i},-\frac{r}{n b_i}\right)$
are also points of $C^{\vee}$, since they correspond to $t=b_i$ which is a pole
of $\Phi_s$, and thus $J(b_i)=\frac{r}{n}$, again by \eqref{djt}.

These families of points correspond to the families of tangent lines
$\left(\{\chi=a_i\}\right)_{1\leq i\leq s}$ and
$\left(\{\chi+\frac{r}{n}\kappa-b_i=0\}\right)_{1\leq i\leq s}$.
The former are vertical lines passing through $(a_i,0)$, and the latter are lines
passing through $(b_i,0)$ with slope $-\frac{n}{r}$.

The point
$(x^\vee,y^\vee)=(0,-1)\in C^{\vee}$ corresponds to the tangent line $\kappa=1$ of $C$.
By
the parametrization of $C$ given in Theorem~\ref{fb}, the roots of 
\begin{equation*}
\left(\Phi_s\left(-\frac{1}{x^\vee}\right)-1\right)\prod_{i=1}^{m}\left(\Phi_s\left(-\frac{1}{x^\vee}\right)+c_i\right)
\end{equation*}
correspond to $(m+1)s-1$ points of tangency of $C$ with the line $\kappa=0$.
\end{proof}

\section{Positions of V-edges in a row and eigenvalues of GUE random matrix}
\label{sec:gue}

In this section, we prove that near a turning point of the frozen boundary, the
present $V$-edges in a random perfect matching of the contracting square-hexagon
lattice are distributed like the eigenvalues of a
random Hermitian matrix from the Gaussian Unitary Ensemble~(GUE).

A matrix of the GUE of size $k$ is diagonalizable with real eigenvalues $\epsilon_1\geq
\epsilon_2\geq\ldots\geq\epsilon_k$, whose distribution $\PP_{\GUE_k}$ on $\RR^k$
has a density with respect to the Lebesgue measure on $\RR^k$ proportional to:
\begin{equation*}
\prod_{1\leq i<j\leq
k}(\epsilon_i-\epsilon_j)^2\exp\left(-\sum_{i=1}^{k}\epsilon_i^2\right),
\end{equation*}
See~\cite[Theorem~3.3.1]{MM04}.
Here is the main theorem we prove in this section.

\begin{theorem}
  \label{tgue}
  Let $\mathcal{R}(\Omega(N),\check{a})$ be a contracting square-hexagon lattice, such that
  \begin{itemize}
    \item $\Omega_N=(\Omega_1(N),\ldots,\Omega_N(N))$ is an $N$-tuple of
      integers denoting the location of vertices on the first row,
    \item the edge weights are assigned as in Assumptions~\ref{apew}
      and~\ref{ap6p},
    \item for every $i$, $x_i=1+o(N^{-2})$,
  \end{itemize}
  Let $\lambda^{k}(N)$ be the signature corresponding to the dimer configuration
  incident to the $(N-k+1)$th row of white vertices in
  $\mathcal{R}(\Omega(N),\check{a})$, and for $1\leq l \leq k$,
  \begin{equation*}
    b_{kl}^{N}=\lambda^{k}_l(N)+N-l.
  \end{equation*}
  Let 
  \begin{equation*}
    \psi_1=\int_{\RR} xd\mathbf{m}^1,\quad
    \psi_2=\int_{\RR} x^2d\mathbf{m}^1,
  \end{equation*}
  where $\mathbf{m}^1$ is the limit counting measure of signatures on the top of
  $\mathcal{R}(\Omega(N),\check{a})$.
  Let
  \begin{equation*}
    \tilde{b}_{kl}^{(N)}=
    \frac{%
      \frac{b_{kl}^{(N)}}{\sqrt{N}}
      -\sqrt{N}\left(%
      \psi_1-\frac{1}{2}+\frac{1}{n}\sum_{i\in I_2\cap\{1,\ldots,n\}}\frac{y_i}{1+y_i}
    \right)
  }{%
    \psi_2-\psi_1^2-\frac{1}{12}+
    \frac{1}{n}\sum_{i\in I_2\cap\{1,2,\ldots,n\}}\frac{y_{i}}{(1+y_{i})^2}},
    \ \ 1\leq l\leq k.
  \end{equation*}
  Then, for any fixed $k$, the distribution of
  $\left(\tilde{b}_{kl}^{(N)}\right)_{l=1}^k$ converges weakly to
  $\PP_{\GUE_k}$ as $N\rightarrow\infty$.
\end{theorem}

The convergence of the
distribution of certain present edges near the boundary to the GUE minor process was proved in
\cite{JN06} in the case of uniform perfect matching on the Aztec diamond,
and in \cite{OR06} for plane partitions.
In the case
of the uniform perfect matching on a hexagon lattice, the result was proved in
\cite{VP15,JN14}. In the case of $q$-distributed perfect matching on a hexagon
lattice with $q=e^{\frac{-\gamma}{N}}$, the result was proved in \cite{MP17}.

\begin{proof}[Proof of Theorem~\ref{tgue}]
In order to prove Theorem~\ref{tgue}, we shall apply the following
characterization of $\PP_{\GUE_k}$, which follows directly from the definition.
See e.g.\@\cite{VP15,JN14}.

\begin{lemma}
  \label{lgue}
  Let $(q_1,\dotsc,q_k)\in\RR^k$ be a random vector with distribution $\PP$ and
  $Q=\operatorname{diag}[q_1,\ldots,q_k]$ the $k\times k$ diagonal matrix
  obtained by putting the $q_i$ on the diagonal.

  Then $\PP$ is $\PP_{\GUE_k}$ if and only if for any matrix $P$,
\begin{equation*}
\EE \int_{U(k)}
\exp[\mathrm{Tr}(PUQU^*)]dU=
\exp\left(\frac{1}{2}\mathrm{Tr}P^2\right).
\end{equation*}
It is in fact enough to check the case when $P$ is diagonal, with real
coefficients.
\end{lemma}

Let
\begin{equation*}
  P_{k}=\mathrm{diag}\left[v_1,\ldots,v_k\right],\quad
  Q_k^{(N)}=\mathrm{diag}\left[b_{k1}^{(N)},\ldots,b_{kk}^{(N)}\right].
\end{equation*}

By Lemma~\ref{hciz}, we have for any $k$:
\begin{equation}
  \label{lg0}
  \int_{U(k)}\exp\left[\mathrm{Tr}\left(P_k U Q_k^{(N)}U^*\right)\right]dU=
  \left[\prod_{1\leq i<j\leq k}\frac{e^{v_i}-e^{v_j}}{v_i-v_j}\right]
  \frac{s_{\lambda^{k}(N)}(e^{v_1},\ldots,e^{v_k})}{s_{\lambda^k(N)}(1,\ldots,1)}
\end{equation}
Let $\rho^k_N$ be the distribution of dimer configurations restricted on the
$(N+k-1)$th row of white vertices on $\mathcal{R}(\Omega(N),\check{a})$, and let 
\begin{equation*}
X_k^{(N)}=\left(x_{\ol{N-k+1}}^{(N)},\ldots,x_{\ol{N}}^{(N)}\right).
\end{equation*}
Then
\begin{multline}
\label{lg1}
\sum_{\lambda_k(N)\in \GT_k^+}\rho^k_N(\lambda_k(N))\int_{U(k)}\exp\left[\frac{1}{\sqrt{N}}\mathrm{Tr}\left(P_kUQ_k^{(N)}U^*\right)\right]dU\\
=\prod_{1\leq i<j\leq k}
\left[\sqrt{N}\frac{e^{v_i/\sqrt{N}}-e^{v_j/\sqrt{N}}}{v_i-v_j}\right] \times
\sum_{\lambda_k(N)} \rho_N^k(\lambda_k(N))
\frac{s_{\lambda_k(N)}(e^{v_1/\sqrt{N}},\dotsc e^{v_k
  /\sqrt{N}})}{s_{\lambda_k(N)}(1,\dotsc,1)}\\
  =
  \prod_{1\leq i<j\leq k}
\left[\sqrt{N}\frac{e^{v_i/\sqrt{N}}-e^{v_j/\sqrt{N}}}{v_i-v_j}\right] \times
  \mathcal{S}_{\rho^k_N,X_k^{(N)}}(e^{v_1/\sqrt{N}},\dotsc,e^{v_k/\sqrt{N}})
  \\
  \times\frac{%
\sum_{\lambda_k(N)} \rho_N^k(\lambda_k(N))
\frac{s_{\lambda_k(N)}(e^{v_1/\sqrt{N}},\dotsc e^{v_k
  /\sqrt{N}})}{s_{\lambda_k(N)}(1,\dotsc,1)}
}{%
\sum_{\lambda_k(N)} \rho_N^k(\lambda_k(N))
\frac{s_{\lambda_k(N)}(e^{v_1/\sqrt{N}},\dotsc e^{v_k
  /\sqrt{N}})}{s_{\lambda_k(N)}(x_1,\dotsc,x_k)}
},
\end{multline}
the denominator of the last fraction being exactly 
$\mathcal{S}_{\rho^k_N,X_k^{(N)}}(e^{v_1/\sqrt{N}},\dotsc,e^{v_k/\sqrt{N}})$
by Definition~\ref{df21}.
First, this fraction is converging to 1 as $N$ goes to infinity. Indeed, all
partitions $\lambda_k(N)\in \GT_k$ that contribute to the sum must have
parts bounded by a constant times $N$ by hypothesis. Since
$|x_i-1|=o(N^{-2})$, one has:
\begin{equation*}
  \left| \frac{s_{\lambda_k(N)}(x_1,\dotsc,x_k)-s_{\lambda_k(N)}(1,\dotsc,1)}{%
    s_{\lambda_k(N)}(1,\dotsc,1)}
    \right| \leq C
    o(N^{-2}) |\lambda_k(N)| C^{o(N^{-2})|\lambda_k(N)|}
   = o(1)
\end{equation*}
uniformly in $\lambda_k(N)$, which implies that the difference between the
fraction and 1 is negligible as $N$ goes to $\infty$.

We then use Lemma~\ref{lm212} to re-express the Schur generating function:
\begin{equation}
\label{lg2}
  \mathcal{S}_{\rho^k_N,X_k^{(N)}}(\zeta_1,\dotsc,\zeta_k)
=
\frac{%
s_{\omega}\left(x_{\ol{1}}^{(N)},\ldots,x_{\ol{N-k}}^{(N)},\zeta_1,\ldots,\zeta_k\right)
}{%
s_{\omega}\left(x_{\ol{1}}^{(N)},\ldots,x_{\ol{N}}^{(N)}\right)}
\!\!\!\!\!\!\!\!\!\!
\prod_{\substack{{1\leq i\leq k} \\{j\in I_2\cap \{1,2,\ldots,N-k\}}}}
\!\!\!\!\!\!\!
\left(\frac{1+y_{\ol{j}}\zeta_i}{1+y_{\ol{j}}x_{\ol{N-k+i}}}\right),
\end{equation}
where $\omega=(\omega_1\geq \omega_2\geq\cdots\geq \omega_N)\in \GT_N^+$ is the
signature corresponding to the first row. 

The same estimate for Schur functions we used before to compare
$s_{\lambda_k(N)}$ evaluated at $(1,\dotsc,1)$ and $(x_1^{(N)},\dotsc,
x_k^{(N)})$ can be used this time for $s_\omega$, using the fact that
$|\omega|=O(N^2)$ by hypothesis.
We have then that
\begin{equation*}
  \left|\frac{%
    s_\omega(x_1^{(N)},\dotsc x_{N}^{(N)})-s_\omega(1,\dotsc,1)
  }{%
s_\omega(1,\dotsc,1)
} \right| \leq \epsilon(N)N^{-2}|\omega|C^{|\omega|\epsilon(N)N^{-2}}=o(1),
\end{equation*}
with $\lim_{N\to\infty} \epsilon(N)=0$,
so we can, up to a negligible correction, replace in the denominator
$s_\omega(x_1^{(N)},\dotsc x_{N}^{(N)})$ by 
$s_\omega(1,\dotsc,1)$ in Equation~\eqref{lg2}.

This ratio of Schur functions
appears also when using again Lemma~\ref{hciz} to rewrite the matrix
integral over $U(N)$ this time.
More precisely, let
\begin{align*}
R_N=\mathrm{diag}\left[v_1,\dotsc,v_k,t_{\ol{1}}^{(N)},\dotsc,t_{\ol{N-k}}^{(N)}\right],
\quad
Q_N=\mathrm{diag}\left[\omega_1+N-1,\omega_2+N-2\ldots,\omega_N\right].
\end{align*}
where $t_i^{(N)}=\sqrt{N}\log\left[x_i^{(N)}\right]$ fo $1\leq i\leq n$.
For $1\leq i\leq k$, let $\zeta_i=e^{\frac{v_i}{\sqrt{N}}}$.
Then,
\begin{multline}
  \label{lg2b}
\frac{%
s_{\omega}\left(x_{\ol{1}}^{(N)},\ldots,x_{\ol{N-k}}^{(N)},\zeta_1,\ldots,\zeta_k\right)
}{%
s_{\omega}\left(1,\dotsc,1\right)}
=
\int_{U(N)}\exp\left[\frac{1}{\sqrt{N}} \mathrm{Tr}\left(R_N U Q_N U^*\right)\right]dU \times\\
\left[\prod_{1\leq i<j\leq k}\!\!\!
\sqrt{N}\frac{e^{\frac{v_i}{\sqrt{N}}}-e^{\frac{v_j}{\sqrt{N}}}}{v_i-v_j}\right]^{-1}\!\!\!
\left[\prod_{\substack{{1\leq i\leq k}\\{1\leq j\leq N-k}}}\!\!\!
  \sqrt{N}\frac{%
    e^{\frac{v_i}{\sqrt{N}}}- x_{\ol{j}}^{(N)}
}{%
  v_i-t_{\ol{j}}^{(N)}
}\right]^{-1}\!\!\!
\left[\prod_{1\leq i<j\leq N-k}\!\!\!
  \sqrt{N}\frac{x_{\ol{i}}^{(N)}- x_{\ol{j}}^{(N)}
}{%
  t_{\ol{i}}^{(N)}-t_{\ol{j}}^{(N)}
}\right]^{-1}.
\end{multline}

Plugging the modified version of Equation~\eqref{lg2} and \eqref{lg2b}
into~\eqref{lg1}, one gets:
\begin{align*}
\log\sum_{\lambda^k(N)\in \GT_k^+}
\rho^k_N(\lambda^k(N))\int_{U(k)} 
\exp\left[\frac{1}{\sqrt{N}}\mathrm{Tr}\left(P_kUQ_k^{(N)}U^*\right)\right]dU\\
=\log\int_{U(N)}\exp\left[\frac{1}{\sqrt{N}}\mathrm{Tr}\left(R_NUQ_NU^*\right)\right]dU
-\sum_{\substack{{1\leq i\leq k}\\{1\leq j\leq N-k}}}
\log\left[
  \frac{%
    e^{v_i/\sqrt{N}}-e^{t_j/\sqrt{N}}
  }{%
    \frac{v_i}{\sqrt{N}}-\frac{t_j}{\sqrt{N}}
  }
\right]
\\
-\sum_{1\leq i<j\leq N-k}
\log\left[
  \frac{%
    e^{t_i\sqrt{N}}-e^{t_j\sqrt{N}}
  }{%
    \frac{t_i}{\sqrt{N}}-\frac{t_j}{\sqrt{N}}
  }
\right]
+
\sum_{\substack{{1\leq i \leq k}\\ {j\in I_2\cap\{1,\ldots,N-k\}}}}
\log\frac{1+y_j e^{v_i\sqrt{N}}}{1+y_j e^{t_{N-k+i}/\sqrt{N}}}
+O(N^{2-\alpha}).
\end{align*}

Expand the sums using the  two expansions
$\log\frac{e^u-e^v}{u-v}=\frac{u+v}{2}+\frac{(u-v)^2}{24}+O(|u|^3+|v|^3)$, and
$\log(\frac{1+ye^v}{1+y})=\frac{y}{1+y}v+\frac{y}{2(1+y)^2}v^2+O(|v|^3)$ as $u$ and $v$
tend to 0, to get that:
\begin{align*}
\sum_{\substack{{1\leq i\leq k}\\{1\leq j\leq N-k}}}
\log\left[
  \frac{%
    e^{v_i/\sqrt{N}}-e^{t_j/\sqrt{N}}
  }{%
    \frac{v_i}{\sqrt{N}}-\frac{t_j}{\sqrt{N}}
  }
\right]&=o(1),\\
\sum_{1\leq i<j\leq N-k}
\log\left[
  \frac{%
    e^{t_i\sqrt{N}}-e^{t_j\sqrt{N}}
  }{%
    \frac{t_i}{\sqrt{N}}-\frac{t_j}{\sqrt{N}}
  }
\right]&=(N-k)\sum_{i=1}^k \left(\frac{v_i}{2\sqrt{N}}+\frac{v_i^2}{24
N}+O(N^{-3/2})\right),\\
\sum_{\substack{{1\leq i \leq k}\\ {j\in I_2\cap\{1,\ldots,N-k\}}}}\!\!\!
\log\frac{1+y_j e^{v_i\sqrt{N}}}{1+y_j e^{\frac{t_{N-k+i}}{\sqrt{N}}}} 
&=\!\!\!\sum_{j\in I_2\cap\{1,\ldots N-k\}}\sum_{i=1}^k \frac{y_j}{1+y_j}
\frac{v_i}{\sqrt{N}}+\frac{y_j}{(1+y_j)^2}\frac{v_i^2}{2N}+O(N^{-\frac{3}{2}}).
\end{align*}

Then we use Lemma~\ref{lem:asympexpiz} below to get the asymptotic behavior of the
integral over $U(N)$, to obtain that
\begin{equation*}
  \log\int_{U(N)}\exp\left[\frac{1}{\sqrt{N}}\mathrm{Tr}\left(R_NUQ_NU^*\right)\right]dU=\psi_1
  \left(\sum_{i=1}^k v_i\right)+\frac{\psi_2-\psi_1^2}{2}\left(\sum_{i=1}^k
  v_i^2\right)+o(1),
\end{equation*}
where $\psi_1$ and $\psi_2$ are respectively the first and second moments of the
limiting measure $\mathrm{m}_\omega$, as $N$ goes to infinity.
Bringing all the pieces together, and defining
\begin{equation*}
  A = \psi_1-\frac{1}{2}+\frac{1}{n}\sum_{j\in
    I_2\cap\{1,\ldots,n\}}\frac{y_j}{1+y_j},\quad
    B = \psi_2-\psi_1^2-\frac{1}{12}+\frac{1}{n}\sum_{j\in
    I_2\cap\{1,\ldots,n\}}\frac{y_j}{1+y_j},
\end{equation*}
one finally obtain that
\begin{equation*}
  \int_{U(k)}\exp\left[\frac{1}{\sqrt{N}}\mathrm{Tr}\left(P_k U Q_k^{(N)}U^*\right)\right]dU=
  \exp\left(\sqrt{N}A\sum_i v_i+
    \frac{1}{2}B \sum_i v_i^2+o(1)\right).  
\end{equation*}
Therefore, according to Lemma~\ref{lgue}, the distribution of the diagonal
coefficients  of
\begin{equation*}
\frac{1}{\sqrt{N}B}(Q_k^{(N)}-N A \operatorname{Id}_k)
\end{equation*}
which are exactly the
$\tilde{b}_{kl}^{(N)}$, converges to $\PP_{\GUE_k}$.
\end{proof}

\begin{lemma}
  \label{lem:asympexpiz}
  Let $k\in \NN$ be fixed.
  For any $N$, let $\omega(N)\in \GT_N^+$ and
  \begin{align*}
    q_i^{(N)}&=\omega_i(N)+N-i,\quad
    Q_N=\operatorname{diag}\left[q_1^{(N)},\dotsc, q_N^{N}\right],\\
    P_N&=\mathrm{diag}\left[
    v_1,\ldots,v_k,\log\left(x_{\ol{1}}^{(N)}\right),\ldots,\log\left(x_{\ol{N-k}}^{(N)}\right)\right].
\end{align*}
Assume that 
\begin{itemize}
\item $\left(\omega(N)\right)_{N\in\NN}$ is a regular sequence of signatures,
\item as $N\rightarrow\infty$, the counting measures $\bm_{\omega(N)}$ converge
  weakly to $\bm_{\omega}$,
\item there exists a fixed positive integer $n$, such that for $1\leq i\leq N-k$, $x_{\ol{i}}^{(N)}$=$x_{i\mod n}^{(N)}$, and
\begin{equation*}
  \forall\ 1\leq j \leq n\ 
  \lim_{N\rightarrow\infty}\left|x_j^{(N)}-1\right|N^{2}=0.
\end{equation*}
\end{itemize}
 Then as $N\rightarrow\infty$, we have the following asymptotic expansion
\begin{equation}
\log \int_{U(N)} \exp\left[\frac{1}{\sqrt{N}}\mathrm{Tr}\left(P_NU Q_N U^*\right)\right]dU
= K_1p_1(v_1,\ldots,v_k) N^{1/2} + \frac{K_2}{2!}p_2(v_1,\ldots,v_k)+o(1)
\label{eq:asymptexpiz}
\end{equation}
where $K_d$ is $(d-1)!$ multiplying the $d$th free cumulant of the measure
$\bm_{\omega}$ (in particular we have $K_1=\psi_1, K_2=\psi_2-\psi_1^2$)
and the $p_d$ is the $d$th power sum:
\begin{equation*}
p_d(v_1,\ldots,v_k)=\sum_{i=1}^{k}v_i^d.
\end{equation*}
\end{lemma}

\begin{proof}
  We write $P_N$ as $P^0_N+P'_N$ where
  $P^0_N=\mathrm{diag}(v_1,\cdots,v_k,0,\cdots,0)$ and $P'_N$ is diagonal, with
  the first $k$ coefficients equal to 0 and the others are the $x_i$s.

  From \cite[Corollary~6]{JN14} (see also
  \cite{ggn}), there is an asymptotic expansion (at any order) of
  the left hand side of~\eqref{eq:asymptexpiz}
  \begin{equation*}
    \log \int_{U(N)} \exp\left[\frac{1}{\sqrt{N}}\mathrm{Tr}\left(P^0_N U Q_N \right)U^*\right]dU
\sim \sum_{d=1}^{\infty} \frac{K_d}{d!}p_d(v_1,\ldots,v_k)N^{1-\frac{d}{2}},
  \end{equation*}
  the first two orders giving the right hand side of \eqref{eq:asymptexpiz}.

  We now show that the additional perturbation coming from $P'_N$ does not
  change the first two orders as long as the $x_i$s are close enough to 1. For
  this, we rewrite the trace of $P'_N U Q U^*$ as:
  \begin{equation*}
    \mathrm{Tr}(P'_N U Q_N U^*)=
    \sum_{i=k+1}^N\sum_{j=1}^N \log(x^{(N)}_{i-k}) q_j^{(N)} |U_{i,j}|^2.
  \end{equation*}
  Then $|\log(x^{(N)}_{i})|=o(N^{-2})$, $q_j^{(N)}=O(N)$ uniformly in $i$
  and $j$, and the sum $\sum_j |U_{i,j}|^2$ is equal to 1, as $U$ is unitary.
  Therefore, this trace is $o(1)$, and
  \begin{equation*}
    \exp\left(\frac{1}{\sqrt{N}} \mathrm{Tr}(P'_N U Q_NU^*)\right)
    =1+o(N^{-\frac{1}{2}}),
  \end{equation*}
  uniformly in $U$. Thus the correction induced by $P'_N$ to the integral is of
  lower order than the two first terms of the asymptotic expansion above, thus
  yielding the result.
\end{proof}

\section{Gaussian Free Field}\label{sec:gff}

In this section, we show that under a homeomorphism from the liquid region to
the upper half plane, the (non-rescaled) height function of the dimer
configurations on a contracting square-hexagon lattice converges to the Gaussian
free field (GFF). The relationship between the dimer height function and GFF has
been studied extensively in the past few decades. Here is an incomplete list of
contributions to this question:
the convergence of height functions to GFF in distribution for uniform perfect
matchings on a simply-connected domain with Temperley boundary condition was
proved in~\cite{RK00,RK01}; for random perfect matchings on the isoradial double
graph with an analogous Temperley boundary condition was proved in~\cite{ZL17};
for uniform perfect matchings on a  hexagonal lattice with sawtooth boundary was
proved in~\cite{BF14,LP12}; for uniform perfect matchings on a square grid with
sawtooth boundary (rectangular Aztec diamond) was proved in~\cite{bk}; for interacting particle systems with
non-flat boundary was proved in ~\cite{MD1,MD2}.

\subsection{Mapping from the liquid region to the upper half plane}

By Theorem~\ref{tm38} and Definition~\ref{df41}, $(\chi,\kappa)$ is in the liquid region if and only if the following system of equations
\begin{equation*}
  \begin{cases}
    F_{\kappa}(z,t)=\frac{\chi}{1-\kappa}\\
    \mathrm{St}_{\mathbf{m}_{\omega}}(t)=\log(z)
  \end{cases}
\end{equation*}
 has non-real roots, where
\begin{equation*}
F_{\kappa}(z,t)=\frac{z}{1-\kappa}\left(\frac{t}{z}-\frac{1}{z-1}+\frac{\kappa}{n}\sum_{i\in I_2\cap\{1,2,\ldots,n\}}\frac{1}{z+c_i}\right)+\frac{z}{z-1}.
\end{equation*}

By explicit computation, one obtains the following:
\begin{lemma}Let $t\in \HH$, where $\HH$ is the upper half plane. Then $(z(t),t)$ is the solution of~\eqref{es} if and only if the following equation holds
\begin{multline}
p^{\chi,\kappa}(t):=\label{pck}\\
\chi-\left(t+\kappa\left(\frac{1}{\exp(-\mathrm{St}_{\mathbf{m}_{\omega}}(t))-1}+\frac{1}{n}\sum_{i\in I_2\cap \{1,2,\ldots,n\}}\frac{1}{c_i\exp(-\mathrm{St}_{\mathbf{m}_{\omega}}(t))+1}\right)\right)=0
\end{multline}
 \end{lemma}
 
 Note that there exists a unique solution $t\in\HH$ of~\eqref{pck}.
 Let $S(t):=\mathrm{St}_{\mathbf{m}_{\omega}}(t)$.
 
 \begin{proposition}
   Let $\mathcal{L}$ denote the liquid region, and $T_\mathcal{L}$ the mapping
 \begin{equation*}
 T_{\mathcal{L}}: \mathcal{L}\rightarrow \HH
 \end{equation*}
 sending $(\chi,\kappa)\in \mathcal{L}$ to the corresponding unique root of~\eqref{pck} in $\HH$.
 Then, $T_{\mathcal{L}}$ is a homeomorphism with inverse
 $t\mapsto(\chi_{\mathcal{L}}(t),\kappa_{\mathcal{L}}(t))$ for all $t\in \HH$, given by
 \begin{align}
 \chi_{\mathcal{L}}(t)&=t-\frac{(t-\ol{t})\zeta(t)}{\zeta(t)-\zeta(\ol{t})}\label{chi}\\
 \kappa_{\mathcal{L}}(t)&=-\frac{t-\ol{t}}{\zeta(t)-\zeta(\ol{t})},\label{ka}
 \end{align}
 where
 \begin{equation*}
 \zeta(t)=-\frac{\exp(S(t))}{\exp(S(t))-1}+\frac{\exp(S(t))}{n}\sum_{i\in
 I_2\cap\{1,\ldots,n\}}\frac{1}{\exp(S(t))+c_i}.
 \end{equation*}
 \end{proposition}

 \begin{proof}
   The proof is an adaptation of Proposition~6.2 of~\cite{bk}; see also Theorem~2.1 of~\cite{DM15}.
 
 It suffices to show all the following statements:
 \begin{enumerate}
 \item $\mathcal{L}$ is nonempty.
 \item $\mathcal{L}$ is open.
 \item $T_{\mathcal{L}}: \mathcal{L}\rightarrow \HH$ is continuous.
 \item $T_{\mathcal{L}}: \mathcal{L}\rightarrow \HH$ is injective.
 \item $T_{\mathcal{L}}: \mathcal{L}\rightarrow T_{\mathcal{L}}(\mathcal{L})$ has continuous inverse for all $t\in T_{\mathcal{L}}(\mathcal{L})$.
 \item $T_{\mathcal{L}}(\mathcal{L})=\HH$.
 \end{enumerate}
 
 We first prove (1). Explicit computations show that
 $(\chi_{\mathcal{L}}(t),\kappa_{\mathcal{L}}(t))$ satisfies \eqref{pck}. Since
 $\mathbf{m}_{\omega}$ is a measure on $\RR$ with compact support, assume that
 $\mathrm{Support}(\mathbf{m}_{\omega})\subset[a,b]$ where $a,b\in\RR$. The
 Stieltjes transform satisfies
 \begin{equation*}
 S(t)=\frac{1}{t}+\frac{\alpha}{t^2}+\frac{\beta}{t^3}+O(|t|^{-4}),
\end{equation*}
where $\alpha$ and $\beta$ are the first two moments of the measure
$\bm_\omega$:
\begin{equation*}
\alpha=\int_a^b x\mathbf{m}_{\omega}(dx),\quad
\beta=\int_a^b x^2\mathbf{m}_{\omega}(dx).
\end{equation*}
After computations we have
\begin{align*}
\chi&=\alpha-\frac{1}{2}+\frac{1}{n}\sum_{i=I_2\cap\{1,2,\ldots,n\}}\frac{1}{1+c_i}+O(|t|^{-1}),\\
\kappa&=1+\left(-\beta-\frac{1}{n}\sum_{i\in I_2\cap\{1,2,\ldots,n\}}\frac{1}{(1+c_i)^2}+\frac{1}{12}+\alpha^2\right)\frac{1}{|t|^2}+O(|t|^{-3}).
\end{align*}

Let $\lambda$ be the Lebesgue measure on $\RR$. Recall that $\mathbf{m}_{\omega}$ is a probability measure supported in the interval $[a,b]$ and $\mathbf{m}_{\omega}\leq\lambda$. Hence we have $b-a\geq 1$. Then we infer
\begin{equation*}
a+\frac{1}{2}=\int_{a}^{a+1}xdx\leq\alpha\leq\int_{b-1}^{b}xdx=b-\frac{1}{2}.
\end{equation*}
Similarly,
\begin{equation*}
\beta-\alpha^2\geq \frac{1}{2}\int_{0}^1\int_{0}^1(x-y)^2dxdy=\frac{1}{12}.
\end{equation*}
As a result, $(\chi,\kappa)\in\left(a+\frac{1}{2},b-\frac{1}{2}\right)\times(0,1)$ whenever $|t|$ is sufficiently large. Then (1) follows.

The facts (2) and (3) follow from Rouch\'e's theorem by the same arguments as in
the proof of Proposition 6.2 in~\cite{bk}. The facts (4)-(6) can also be
obtained by the same arguments as in the proof of Proposition 6.2 in~\cite{bk}.
 \end{proof}

 \subsection{Gaussian Free Field} 
 
 Let $C_0^{\infty}$ be the space of smooth real-valued functions with compact
 support in the upper half plane $\HH$. The \emph{Gaussian free field} (GFF)
 $\Xi$ on $\HH$ with the zero boundary condition is a collection of Gaussian
 random variables $\{\Xi_{f}\}_{f\in C_0^{\infty}}$ indexed by functions in
 $C_0^{\infty}$, such that the covariance of two Gaussian random variables
 $\Xi_{f_1}$, $\Xi_{f_2}$ is given by
 \begin{equation}
 \mathrm{Cov}(\Xi_{f_1},\Xi_{f_2})=\int_{\HH}\int_{\HH}f_1(z)f_2(w)G_{\HH}(z,w)dzd\ol{z}dwd\ol{w},\label{cvf}
 \end{equation}
 where
 \begin{equation*}
 G_{\HH}(z,w):=-\frac{1}{2\pi}\ln\left|\frac{z-w}{z-\ol{w}}\right|,\qquad z,w\in \HH
 \end{equation*}
 is the Green's function of the Laplacian operator on $\HH$ with the Dirichlet
 boundary condition. The Gaussian free field $\Xi$ can also be considered as a
 random distribution on $C_0^{\infty}$,
 such that for any $f\in C_0^{\infty}$, we have
 \begin{equation*}
 \Xi(f)=\int_{\HH}f(z)\Xi(z)dz:=\xi_f.
 \end{equation*}
Here $\Xi(f)$ is the Gaussian random variable with respect to $f$, which has
mean 0 and variance given by \eqref{cvf}) with $f_1$ and $f_2$ replaced by $f$.
See~\cite{SS07} for more about the GFF.
 
 Consider a contracting square-hexagon lattice $\mathcal{R}(\Omega,\check{a})$.
 Let $\omega$ be a signature corresponding to the boundary row.
 
 By \eqref{pb}, we defined a probability measure on the set of signatures
 \begin{equation*}
 \mathcal{S}^N=(\mu^{(N)},\nu^{(N)},\ldots,\mu^{(1)},\nu^{(1)}).
 \end{equation*}
 By Proposition~\ref{p16}, this measure is exactly the
 probability measure on dimer configurations such that the probability of a
 perfect matching is proportional to the product of weights of present edges.
 
 Define a function $\Delta^N$ on $\RR_{\geq 0}\times \RR_{\geq 0}\times \mathcal{S}\rightarrow \NN$ as follows
 \begin{multline*}
\Delta^{N}: (x,y,(\mu^{(N)},\nu^{(N)},\ldots,\mu^{(1)},\nu^{(1)}))\mapsto \\
 \sqrt{\pi}|\{1\leq s\leq N-\lfloor y\rfloor:\mu_s^{N-\lfloor y\rfloor}+(N-\lfloor y\rfloor)-s\geq x\}|
 \end{multline*}
 This is, up to a deterministic shift, and an affine transformation, the height
 function defined in Section~\ref{subsec:height}, extended to a piecewise constant
 function.

Let $\Delta_{\mathcal{M}}^N(x,y)$ be the push-forward of the measure
$P_{\omega}^N$ on $\mathcal{S}^N$ with respect to $\Delta^N$. For $z\in \HH$,
define
\begin{equation*}
\mathbf{\Delta}_{\mathcal{M}}^N(z):=\Delta_{M}^N(N\chi_{\mathcal{L}}(z),N\kappa_{\mathcal{L}}(z)),
\end{equation*}
where $\chi_{\mathcal{L}}(z)$, $\kappa_{\mathcal{L}}(z)$ are defined by
\eqref{chi}, \eqref{ka}, respectively.

Here is the main theorem we shall prove in the section.

\begin{theorem}
  \label{tt73}
  Let $\mathbf{\Delta}_{\mathcal{M}}^{N}(z)$ be a random function corresponding
  to the random perfect matching of the contracting square-hexagon lattice, as
  explained above, with the following hypothesis on the weights
  $\beta_i^{(N)}$:
\begin{equation*}
\lim_{N\rightarrow\infty}\sup_{i}|\beta_i^{(N)}-1|N=0.
\end{equation*}
  
  Then
$\mathbf{\Delta}_{\mathcal{M}}^{N}(z)-\EE\mathbf{\Delta}_{\mathcal{M}}^{N}(z)$,
seen as a random field, 
converges as $N$ goes to $\infty$ to the Gaussian free field $\Xi$ in $\HH$ with
Dirichlet boundary conditions, in the following sense:
for $0<\kappa\leq 1$, $j\in \NN$, define:
\begin{equation*}
M_j^{\kappa}=\int_{-\infty}^{+\infty}\chi^j(\mathbf{\Delta}_{\mathcal{M}}^{N}(N\chi,N\kappa)-\mathbb{E}\mathbf{\Delta}_{\mathcal{M}}^{N}(N\chi,N\kappa))d\chi,
\end{equation*}
and
\begin{equation*}
\mathbf{M}_j^{\kappa}=\int_{z\in\HH;\kappa_{\mathcal{L}}(z)=\kappa}\left[\chi_{\mathcal{L}}(z)\right]^j\Xi(z)d\chi_{\mathcal{L}}(z).
\end{equation*}
Then:
\begin{equation*}
M_j^{\kappa}\rightarrow \mathbf{M}_j^{\kappa},\ \mathrm{as}\ N\rightarrow\infty.
\end{equation*}
\end{theorem}

\subsection{Covariance}

To derive this statement, we use the technology developed by Bufetov and Gorin,
linking the asymptotic behavior of Schur generating functions to the
fluctuations for the related tiling/interlacing signature model~\cite{bg,bg2}.

\begin{definition}
  \label{dfa}
  A sequence of symmetric functions $\{H_N(u_1,\ldots,u_N)\}_{N\geq 1}$ is
  \emph{appropriate} (or \emph{CLT-appropriate}) if there exist two collections
  of real numbers $\{\gamma_k\}_{k\geq 1}$, $\{\lambda_{k,l}\}_{k,l\geq 1}$,
  such that the following properties are satisfied:
  \begin{itemize}
    \item for any $N$, the function $\mathbf{u}\mapsto\log H_N(\mathbf{u})$
      is holomorphic in an open complex neighborhood of
      $B^{(N)}=(\beta_{1}^{{N}},\ldots,\beta_{N}^{(N)})$, where
      $\mathbf{u}=(u_1,\ldots,u_N)$, and
      \begin{equation*}
	\lim_{N\rightarrow\infty}\sup_{1\leq i\leq N}|\beta_i^{(N)}-1|N=0.
      \end{equation*}
    \item for any index $i$ and $k\in \NN$, we have
      \begin{equation*}
	\lim_{N\rightarrow\infty}
	\left.\frac{\partial_i^{k}\log H_N(\mathbf{u})}{N}\right|_{\mathbf{u}=B^{(N)}}=\gamma_k.
      \end{equation*}
    \item For any distinct indices $i,j$ and any $k,l\in\NN$, we have
      \begin{equation*}
	\lim_{N\rightarrow\infty}
	\left.\partial_i^k\partial_j^l\log H_N(\mathbf{u})\right|_{\mathbf{u}=B^{(N)}}=\lambda_{k,l}
      \end{equation*}
    \item For any $s\in\NN$ and any indices $i_1,\ldots,i_s$, such that there
      are at least three distinct numbers in $i_1,\ldots,i_s$, we have
      \begin{equation*}
	\lim_{N\rightarrow\infty}
	\left.\partial_{i_1}\partial_{i_2}\ldots\partial_{i_s}
	\log H_N(\mathbf{u})\right|_{\mathbf{u}=B^{(N)}}=0.
      \end{equation*}
    \item The power series
      \begin{equation*}
	\sum_{k=1}^{\infty}\frac{\gamma_k}{(k-1)!}(z-1)^{k-1},\qquad \qquad 
	\sum_{k=1,l=1}^{\infty}\frac{\lambda_{k,l}}{(k-1)!(l-1)!}(z-1)^{k-1}(w-1)^{l-1}
      \end{equation*}
      converge in an open neighborhood of $z=1$ and $(z,w)=(1,1)$, respectively.
  \end{itemize}
\end{definition}

The definition of CLT-apprioriateness is extended to measures using their Schur
gererating functions as follows:
\begin{definition}
  \label{asm}
  Let as before
\begin{equation*}
B^{(N)}=(\beta_{1}^{{N}},\ldots,\beta_{N}^{(N)}).
\end{equation*}
satisfying the hypothesis
\begin{equation*}
\lim_{N\rightarrow\infty}\sup_{1\leq i\leq N}|\beta_i^{(N)}-1|N=0.
\end{equation*}
A sequence of measures $\rho=\{\rho_N\}_{N\geq 1}$ is \emph{appropriate} (or
\emph{CLT-appropriate}) if the sequence of Schur generating functions
$\{\mathcal{S}_{\rho_N,B^{(N)}}(u_1,\ldots,u_N)\}_{N\geq 1}$, defined as in
Definition~\ref{df21}, is appropriate in the sense of Definition~\ref{dfa}.

For such a sequence we define functions
\begin{align*}
A_{\rho}(z)&=\sum_{k=1}^{\infty}\frac{\gamma_k}{(k-1)!}(z-1)^{k-1}\\
B_{\rho}(z,w)&=\sum_{k=1,l=1}^{\infty}\frac{\lambda_{k,l}}{(k-1)!(l-1)!}(z-1)^{k-1}(w-1)^{l-1}\\
C_{\rho}(z,w)&=B_{\rho}(1+z,1+w)+\frac{1}{(z-w)^2}.
\end{align*}
Here $\{\gamma_k\}_{k\geq 1}$ and $\{\lambda_{k,l}\}_{k,l\geq 1}$ are obtained from the definition of CLT-appropriate symmetric functions with $H_N=\mathcal{S}_{\rho_N,B^{(N)}}$ as in Definition~\ref{dfa}.
\end{definition}

\begin{lemma}\label{l57}Assume that the Schur generating function of a sequence of probability measures $\{\rho_N\}_{N\geq 1}$ on $\GT_N$ satisfies the conditions
\begin{eqnarray*}
&&\lim_{N\rightarrow\infty}\frac{\partial_1\log\mathcal{S}_{\rho_N,B^{{N}}}(u_1,\ldots,u_k,\beta_{k+1}^{{(N)}},\ldots,\beta_N^{(N)})}{N}=Q_1(u_1),\qquad \forall\ k\geq 1,\\
&&\lim_{N\rightarrow\infty}\partial_1\partial_2\log\mathcal{S}_{\rho_N,B^{{N}}}(u_1,\ldots,u_k,\beta_{k+1}^{{(N)}},\ldots,\beta_N^{(N)})=Q_2(u_1,u_2),\qquad \forall\ k\geq 1,\\
&&\lim_{N\rightarrow\infty}\sup_{1\leq i\leq N}|\beta_i^{(N)}-1|N=0
\end{eqnarray*}
where $Q_1(z)$, $Q_2(z,w)$ are holomorphic functions, and the convergence is uniform in a complex neighborhood of unity. Then $\{\rho_N\}$ is a (CLT-)appropriate sequence of measures with 
\begin{eqnarray*}
A_{\rho}(z)=Q_1(z),\qquad B_{\rho}(z,w)=Q_2(z,w)
\end{eqnarray*}

\end{lemma}

Let $t$ be a positive integer, and $a_t\in (0,1]$ be a real number. Let $\rho_N$ be a probability measure on $\GT_N$. Define the random variables
\begin{eqnarray*}
p_{k,t}^{(N)}:=\sum_{i=1}^{[a_tN]}(\lambda_i^{(t)}+[a_tN]-i)^k,
\end{eqnarray*}
where $k=1,2,\ldots$, and $\lambda=(\lambda_1,\ldots,\lambda_N)$ is $\rho_N$-distributed.

\begin{theorem}\label{gff1}Let $\rho=\{\rho_N\}_{N\geq 1}$ be an appropriate sequence of measures on signatures with limiting functions $A_{\rho}(z)$ and $C_{\rho}(z,w)$, as defined in Definition~\ref{asm}. Then the collection of random variables 
\begin{eqnarray*}
\{N^{-k}(p_{k,t}^{(N)}-\mathbb{E}p_{k,t}^{(N)})\}_{k\in \NN;t=1,\ldots,m}
\end{eqnarray*}
converges, as $N\rightarrow\infty$, in the sense of moments, to the Gaussian vector with zero mean and covariance
\begin{eqnarray*}
\lim_{N\rightarrow\infty}\frac{\mathrm{cov}(p_{k_1,t_1}^{(N)},p_{k_2,t_2}^{(N)})}{N^{k_1+k_2}}=\frac{a_{t_1}^{k_1}a_{t_2}^{k_2}}{(2\pi\mathbf{i})^2}&&\oint_{|z|=\epsilon}\oint_{|w|=2\epsilon}\left(\frac{1}{z}+1+(1+z)A_{t_1}(1+z)\right)^{k_1}\\
&&\times\left(\frac{1}{w}+1+(1+w)A_{t_2}(1+w)\right)^{k_2}C_{t_2}(z,w)dzdw,
\end{eqnarray*}
where 
\begin{itemize}
  \item for $i=1,2$, $A_{t_i}(z)$ and $C_{t_i}(z)$ are analytic functions in a
    complex neighborhood of $1$ obtained by Definition~\ref{asm} and Lemma~\ref{l57} from
    the induced measure on signatures of the $2[a_t N]$-th row of the
    contracting square-hexagon lattice.
\item the $z$- and $w$-contours of integration are counter-clockwise and $\epsilon\ll 1$.
\item $1\leq t_1\leq t_2\leq m$.
\end{itemize}
\end{theorem}

\begin{definition}
  \label{df54}
Let $B=(\beta_1,\beta_2,\ldots,\beta_n,\ldots)$.
Let $m$ be a positive integer and $\epsilon>0$. Let $\Lambda_{\epsilon}^m$ be the space of analytic symmetric functions in the region
\begin{equation*}
\{(z_1,\ldots,z_m)\in\CC^m: z_i\in \cup_{i=1}^{\infty}O_{\epsilon}(\beta_i)\},
\end{equation*}
where
\begin{equation*}
O_{\epsilon}(\beta_i)=\{z:|z-\beta_i|<\epsilon\}.
\end{equation*}
The space $\Lambda_{\epsilon}^m$ can be considered as a topological space with
topology of uniform convergence in the region above.
Let
\begin{equation*}
\Lambda^m:=\cup_{\epsilon>0}\Lambda_{\epsilon}^m
\end{equation*}
the topological space endowed with the topology of direct limit.

Let $\mathbf{p}_{m,n,B}: \Lambda^m\rightarrow\Lambda^n$ be a map with the properties below
\begin{enumerate}
\item $\mathbf{p}_{m,n,B}$ is a continuous linear map.
\item For every $\lambda\in \GT_m$, 
\begin{eqnarray*}
\mathbf{p}_{m,n,B}\left(\frac{s_{\lambda}(u_1,\ldots,u_m)}{s_{\lambda}(\beta_1,\ldots,\beta_m)}\right)=\sum_{\mu\in\GT_n}c_{\lambda,\mu}^{\mathbf{p}_{m,n,B}}\frac{s_{\mu}(u_1,\ldots,u_n)}{s_{\mu}(\beta_1,\ldots,\beta_n)},\qquad c_{\lambda,\mu}^{\mathbf{p}_{m,n,B}}\geq 0.
\end{eqnarray*}
\item For any $f\in \Lambda^m$ we have
\begin{equation*}
f(\beta_1,\ldots,\beta_m)=\mathbf{p}_{m,n,B}(f)(\beta_1,\ldots,\beta_n).
\end{equation*}
\end{enumerate}
It follows from (2) and (3) that
\begin{equation*}
  \sum_{\mu\in \GT_n}c_{\lambda,\mu}^{\mathbf{p}_{m,n,B}}=1.
\end{equation*}
\end{definition}

Recall that we have a probability measure on dimer configurations of the
square-hexagon lattice defined by (\ref{pb}), where the square-hexagon lattice
has $(2N+1)$ rows and on the bottom row, the configuration has signature
$\omega$. Define a mapping $\phi: \{1,\ldots,2N+1\}\rightarrow
\{\mu^{(i)},\nu^{(j)}:i,j\in \{1,2,\ldots,N\}\}$ as follows
\begin{equation*}
\phi(n)=
  \begin{cases}
  \mu^{\left(\frac{n-1}{2}\right)}&\text{if $n$ is odd}\\
  \nu^{\left(\frac{n}{2}\right)}&\text{if $n$ is even}
\end{cases}
\end{equation*}

Let $1\leq n_1\leq n_2\leq\ldots \leq n_s=2N+1$ be positive row numbers of the square-hexagon lattice, counting from the top.
For $1\leq i\leq s$, let $\rho_{n_i}$ be the induced probability measure on dimer configurations of the $n_i$th row. Then the induced  probability measure on the state space
\begin{equation*}
\GT_{\lfloor\frac{n_1}{2}\rfloor}\times\GT_{\lfloor\frac {n_2}{2}\rfloor}\times \ldots\times \GT_{\lfloor \frac{n_s}{2}\rfloor}
\end{equation*}
by the measure \eqref{pb} can be expressed as follows:
\begin{equation}
\mathrm{Prob}\left(\phi(n_s),\ldots,\phi(n_1)\right)=\rho_{n_s}\left((\phi(n_s)\right)\prod_{i=2}^{k}c_{\phi(n_i),\phi(n_{i-1})}^{\mathbf{p}_{n_i,n_{i-1},B}},
\label{prb}
\end{equation}
where $c_{\phi(n_i),\phi(n_{i-1})}^{\mathbf{p}_{n_i,n_{i-1},B}}$ is the probability of $\phi(n_{i-1})$ conditional on $\phi(n_i)$. In particular, we have
\begin{equation*}
c_{\nu^{(t)},\mu^{(t-1)}}^{\mathbf{p}_{2t,2t-1,B}}=\mathrm{pr}_{B}(\nu^{(t)}\rightarrow\mu^{(t-1)}),\qquad
c_{\mu^{(t)},\nu^{(t)}}^{\mathbf{p}_{2t+1,2t,B}}=\mathrm{st}_{B}(\nu^{(t)}\rightarrow\mu^{(t-1)}),
\end{equation*}
where $B$ depends on the edge weights. Moreover, we have
\begin{eqnarray*}
\mathbf{p}_{2t,2t-1,B}\left(\frac{s_{\nu^{(t)}}(u_1,\ldots,u_{t-1},\beta_{t})}{s_{\nu^{(t)}}(\beta_1,\ldots,\beta_{t})}\right)&=&\frac{s_{\nu^{(t)}}(u_1,\ldots,u_{t-1},\beta_{t})}{s_{\nu^{(t)}}(\beta_1,\ldots,\beta_{t})};\\
\mathbf{p}_{2t+1,2t,B}\left(\frac{s_{\mu^{(t)}}(u_1,\ldots,u_{t-1},u_{t})}{s_{\mu^{(t)}}(\beta_1,\ldots,\beta_{t})}\right)&=&\left(\prod_{i=1}^{t}\frac{1+u_i}{1+\beta_i}\right)\cdot\left(\frac{s_{\mu^{(t)}}(u_1,\ldots,u_{t-1},u_{t})}{s_{\mu^{(t)}}(\beta_1,\ldots,\beta_{t})}\right);
\end{eqnarray*}
and for $2\leq i\leq s$
\begin{equation*}
\mathbf{p}_{n_{i},n_{i-1},B}=\mathbf{p}_{n_{i},n_{i}-1,B}\circ\mathbf{p}_{n_{i}-1,n_{i}-2,B}\circ\ldots\mathbf{p}_{n_{i-1}+1,n_{i-1},B}.
\end{equation*}

For simplicity, we will denote the Schur generating function
$\mathcal{S}_{\rho_N,B}$ from Definition~\ref{df21} simply by:
\begin{equation*}
S_{N,B^{(N)}}:=\mathcal{S}_{\rho_N,B^{(N)}}.
\end{equation*}
For $1\leq i\leq s$, let $S_{N,B^{(N)},n_i}$ be the Schur generating function for the induced measure $\rho_{n_i}$ on the $n_i$th row of the $\mathcal{R}(N,\Omega,m)$.

We use also the following notation $\mathbf{u}_a=(u_1,\ldots,u_a)$ for a
positive integer $a$.


\begin{lemma}
  \label{l59}
  Let $m_1,\ldots,m_k$ be positive integers.
  Let $n_1,\ldots,n_k,\mathbf{p}_{n_2,n_1,B},\ldots,\mathbf{p}_{n_k,n_{k-1},B}$
  be defined as in Definition~\ref{df54}. Assume that
  $(\phi(n_s),\ldots,\phi(n_1))$ has distribution \eqref{prb}.
Let $\mathcal{D}_{k}^{(m)}$ be the $k$-th order differential operator defined by
\eqref{dk} on functions in $\Lambda^m$. 
 Then
\begin{multline*}
\mathcal{D}_{m_1}^{(n_1)}\mathbf{p}_{n_2,n_1,B}\mathcal{D}_{m_2}^{(n_2)}\mathbf{p}_{n_3,n_2,B}\ldots\mathbf{p}_{n_k,n_{k-1},B}\mathcal{D}_{m_k}^{(n_k)}S_{N,B^{(N)}}(u_1,\ldots,u_{n_k})|_{(u_1,\ldots,u_{n_k})=(\beta_1,\ldots,\beta_{n_k})}\\
=\mathbb{E}\left(%
  \prod_{j=1}^k\left(%
    \sum_{i_j=1}^{\lfloor\frac{n_j}{2}\rfloor}
    (\phi(n_j)_{i_j} +\lfloor\frac{n_j}{2}\rfloor-i_j)^{m_j}
     \right)
   \right).
\end{multline*}
\end{lemma}
\begin{proof}The theorem follows from the fact that Schur functions are
  eigenfunctions for the operators $\mathcal{D}_m^{(n)}$, and explicit
  computations as in the proof of Proposition 4.3 of~\cite{bg2}.
\end{proof}

Let $f$ be a function with $r$ complex variables. Define its symmetrization:
\begin{equation*}
\mathrm{Sym}_{u_1,\ldots,u_r}f(u_1,\ldots,u_r):=\frac{1}{r!}\sum_{\sigma\in S_r}f(u_{\sigma(1)},u_{\sigma(2)},\ldots,u_{\sigma(r)}),
\end{equation*}
where $S_r$ is the symmetric group of $r$ symbols.

Let $m$ be a positive integer and let $0<a_1\leq a_2\leq\ldots\leq a_m=1$ be real numbers.
For positive integers $l$ and $q$ satisfying $1\leq q\leq m$, we introduce the
notation
\begin{align*}
F_{l,a_q^{1}}(\mathbf{u}_{[a_qN]}):&=\frac{1}{S_{N,B^{(N)},2[a_q
N]-1}(\mathbf{u}_{[a_qN]})V_{[a_qN]}(\mathbf{u}_{[a_qN]})}\\
&\times\sum_{i=1}^{[a_q N]}(u_i\partial_i)^l \left[S_{N,B^{(N)},2[a_q
N]-1}(\mathbf{u}_{[a_qN]})V_{[a_qN]}(\mathbf{u}_{[a_qN]})\right]\\
&=\frac{1}{S_{N,B^{N},2[a_qN]-1}(\mathbf{x})}\mathcal{D}_{l}S_{N,B^{(N)},2[a_q
N]-1}(\mathbf{x}),\\
F_{l,a_q^{2}}(\mathbf{u}_{[a_qN]}):&=\frac{1}{S_{N,B^{(N)},2[a_q N]}(\mathbf{u}_{[a_qN]})V_{[a_qN]}(\mathbf{u}_{[a_qN]})}\\
&\times\sum_{i=1}^{[a_q N]}(u_i\partial_i)^l \left[S_{N,B^{(N)},2[a_q N]}(\mathbf{u}_{[a_qN]})V_{[a_qN]}(\mathbf{u}_{[a_qN]})\right]\\
&=\frac{1}{S_{N,B^{N},2[a_qN]}(\mathbf{x})}\mathcal{D}_{l}S_{N,B^{(N)},2[a_q N]}(\mathbf{x})
\end{align*}
where $\mathcal{D}_l$ is defined by \eqref{dk}, and $V_N$ is the Vandermonde
determinant on $N$ variables $x_1,\ldots,x_N$.

For a positive integer $s$, let $[s]=\{1,2,\ldots,s\}$.

\begin{lemma}
  \label{lsmn}
  Assume that for each $r=1,2,\ldots$ $\xi_r(\mathbf{u})$ is an analytic
  function of $\mathbf{u}$ in an open neighborhood of $1^r$. Then for any indices $b_1,\ldots,
  b_{q+1}$ the function
\begin{equation}
\mathrm{Sym}_{b_1,b_2,\ldots,b_{q+1}}\left(\frac{\xi_r(\mathbf{u})}{(u_{b_1}-u_{b_2})\cdots(u_{b_1}-u_{b_{q+1}})}\right)
\label{smn}
\end{equation}
is analytic in a (possibly smaller) open neighborhood of $1^r$. If the degree of
$r$ in $\xi_r(\mathbf{u})$ is at most $D$ (less than $D$), then the sequence
\eqref{smn} has $r$-degree at most $D$ (less than $D$).
\end{lemma}

\begin{lemma}\label{l510}Let $\{\rho_N\}_{N\geq 1}$ be an appropriate sequence of measures on $\GT_N$ with the corresponding Schur generating function $\mathcal{S}_{N,B^{(N)}}$. Let $\tau\in\{1,2\}$, then
\begin{multline*}
\partial_i F_{l,a_q^{\tau}}(\mathbf{u})=\partial_i\left[\sum_{r=0}^{l}\left(\begin{array}{c}l\\r\end{array}\right)(r+1)!\right.\\
\left.\times\sum_{\{b_1,\ldots,b_{r+1}\}\subset[N]}\mathrm{Sym}_{b_1,\ldots,b_{r+1}}\left(\frac{u_{b_1}^{l}(\partial_{b_1}[\log S_{N,B^{(N)},2[a_qN]+\tau-2}])^{l-r}}{(u_{b_1}-u_{b_2})\ldots(u_{b_1}-u_{b_{r+1}})}\right)\right]+\hat{T}_l(\mathbf{u})
\end{multline*}
where the degree of $N$ in $\partial_i F_{l,a_q^{\tau}}(\mathbf{u})$ is at most $l$, and the degree of $N$ in $\hat{T}(\mathbf{u})$ is less than $l$.
\end{lemma}
\begin{proof}This lemma follows from the same arguments as in the proof of Lemma 5.5 in~\cite{bg2}.
\end{proof}

For positive integers $l_1,l_2$, $q,s$ satisfying $1\leq q\leq s\leq m$, and $\tau_1,\tau_2\in\{1,2\}$, we define
\begin{multline}
  G_{l_1,l_2;a_q^{\tau_1},a_s^{\tau_2}}(\mathbf{u})=l_1\sum_{r=0}^{l_1-1}\binom{l_1-1}{r}\sum_{\{b_1,\ldots,b_{r+1}\}\subset[N]}(r+1)!\\
\times\mathrm{Sym}_{b_1,\ldots,b_{r+1}}\frac{x_{b_1}^{l_1}\partial_{b_1}[F_{l_2,a_s^{\tau_2}}](\partial_{b_1}[\log S_{N,B^{(N)},2[a_qN]]+\tau_1-2})^{l_1-1-r}}{(u_{b_1}-u_{b_2})\ldots(u_{b_1}-u_{b_{r+1}})}.\label{g12}
\end{multline}

For a subset $\{j_1,\ldots,j_p\}\subseteq[s]$, let
$\mathcal{P}_{j_1,\ldots,j_p}^s$ be the set of all pairings of the set
$[s]\setminus\{j_1,\ldots,j_p\}$. Note that
$\mathcal{P}_{j_1,\ldots,j_p}^s=\emptyset$ if $s-p$ is odd. For a pairing $P$, let $\prod_{(a,b)\in P}$ be the product over all pairs $(a,b)$ from this pairing.
Finally, for each positive integer $l$ let
\begin{equation*}
E_{l,B^{(N)},a_q^{\tau}}=F_{l,a_q^{\tau}}\left(\beta_1^{(N)},\ldots,\beta_N^{(N)}\right)
\end{equation*}

\begin{proposition}\label{p510}Let $\rho_N$ be an appropriate sequence of measures on $\GT_N$, $N=1,2,\ldots$ Then for any positive integer $m$, any positive integers $l_1,\ldots,l_m$, and $\tau_i\in\{1,2\}$, for $1\leq i\leq m$, we have
\begin{eqnarray*}
&&\lim_{N\rightarrow\infty}\frac{1}{N^{l_1+\ldots+l_m}}\frac{1}{V_{[a_mN]}}\left(\sum_{i_1=1}^{[a_1 N]}(u_{i_1}\partial_{i_1})^{l_1}-E_{l_1,B^{(N)},a_1^{\tau_1}}\right)\frac{S_{N,B^{(N)},2[a_1N]+\tau_1-2}}{S_{N,B^{(N)},2[a_2N]+\tau_2-2}}\\
&&\times \left(\sum_{i_1=1}^{[a_2 N]}(u_{i_2}\partial_{i_2})^{l_2}-E_{l_2,B^{(N)},a_{2}^{\tau_2}}\right)\frac{S_{N,B^{(N)},2[a_2N]+\tau_2-2}}{S_{N,B^{(N)},2[a_3N]+\tau_3-2}}\\
&&\times\cdots\times\left.\left(\sum_{i_m=1}^{[a_mN]}(u_{i_m}\partial_{i_m})^{l_m}-E_{l_m,B^{(N)},a_{m}^{\tau_m}}\right)V_{[a_m N]}S_{N,B^{(N)},2[a_mN]+\tau_m-2}\right|_{\mathbf{u}=B_N^{(N)}}\\
&=&\lim_{N\rightarrow\infty}\frac{1}{N^{l_1+\ldots+l_m}}\sum_{P\in \mathcal{P}_{\emptyset}^{s}}\left.\prod_{(s,t)\in P}G_{l_s,l_t;a_s^{\tau_s},a_t^{\tau_t}}(\mathbf{u})\right|_{\mathbf{u}=B_N^{(N)}}
\end{eqnarray*}
\end{proposition}
\begin{proof}
  The proposition follows from the same technique as the proof of Proposition
  5.12 in~\cite{bg2}, although our definition of Schur generating functions
  $S_{N,B^{(N)}}$ is slightly different.
\end{proof}
\bigskip
\noindent\textbf{Proof of Theorem~\ref{gff1}}
The case when $\beta_i^{(N)}=1$ for all $N$ and $1\leq i\leq N$ was proved
in~\cite[Theorem~2.8]{bg2}. We will prove it under the assumption that
\begin{equation*}
\lim_{N\rightarrow\infty}\sup_{i}|\beta_i^{(N)}-1|N=0.
\end{equation*}
By Lemma~\ref{l59} and Proposition~\ref{p510}, it suffices to show that
\begin{equation}
\lim_{N\rightarrow\infty}\left.\frac{G_{l_1,l_2;a_1^{\tau_1},a_2^{\tau_2}}(\mathbf{u})}{N^{l_1+l_2}}\right|_{\mathbf{x}=B^{(N)}}=\lim_{N\rightarrow\infty}\left.\frac{G_{l_1,l_2;a_1^{\tau_1},a_2^{\tau_2}}(\mathbf{u})}{N^{l_1+l_2}}\right|_{\mathbf{u}=1^N},\label{cri}
\end{equation}
where by the result in~\cite{bg2}, the right hand side of \eqref{cri} is known to be 
\begin{multline*}
\lim_{N\rightarrow\infty}\left.\frac{G_{l_1,l_2,a_1^{\tau_1},a_2^{\tau_2}}(\mathbf{u})}{N^{l_1+l_2}}\right|_{\mathbf{x}=1^N}=\frac{a_1^{l_1}a_2^{l_2}}{(2\pi\mathbf{i})^2}\oint_{|z|=\epsilon}\oint_{|w|=2\epsilon}\left(\frac{1}{z}+1+(1+z)A_1(1+z)\right)^{k_1}\\
\times\left(\frac{1}{w}+1+(1+w)A_2(1+w)\right)^{k_2}C_2(z,w)dzdw
\end{multline*}
By definition of an appropriate sequence, see Definitions~\ref{dfa},~\ref{asm}, we have
\begin{eqnarray*}
&&\left.\lim_{N\rightarrow\infty}\frac{\partial_i[\log S_{N,B^{(N)},a_t^{\tau_t}}]}{N}\right|_{\mathbf{x}=B^{(N)}}=\left.\lim_{N\rightarrow\infty}\frac{\partial_i[\log S_{N,1^N},a_t^{\tau_t}]}{N}\right|_{\mathbf{x}=1^N}\\
&&\left.\lim_{N\rightarrow\infty}\partial_i\partial_j[\log S_{N,B^{(N)},a_t^{\tau_t}}]\right|_{\mathbf{x}=B^{(N)}}=\left.\lim_{N\rightarrow\infty}\partial_i\partial_j[\log S_{N,1^N,a_t^{\tau_t}}]\right|_{\mathbf{x}=1^N}
\end{eqnarray*}
From the expression of $G_{l_1,l_2}$ in \eqref{g12} and the expression of
$\partial_i F_l$ in Lemma~\ref{l510}, as well as Lemma~\ref{lsmn}, \eqref{cri} follows.
\hfill$\Box$

\begin{lemma}\label{l513}Assume that $\lambda(N)\in\GT_N$, $N=1,2,\ldots$ is a regular sequence of signatures such that
\begin{equation*}
\lim_{N\rightarrow\infty}[m(\lambda(N))]=\mathbf{m}.
\end{equation*}
Then
\begin{multline*}
\lim_{N\rightarrow\infty}\partial_1\partial_2\log\left(\frac{s_{\lambda(N)}(u_1,u_2,\ldots,u_k,1^{N-k})}{s_{\lambda(N)}(1^N)}\right)\\
=\partial_1\partial_2\left(1-(u_1-1)(u_2-1)\frac{u_1H'_{\mathbf{m}}(u_1)-u_2H'_{\mathbf{m}}(u_2)}{u_1-u_2}\right),
\end{multline*}
where the convergence is uniform over an open complex neighborhood of $(u_1,\ldots,u_k)=(1^k)$.
\end{lemma}
\begin{proof}See Theorem 6.7 of~\cite{bk}.
\end{proof}

\begin{proposition}\label{p515}Assume that $\lambda(N)\in\GT_N$, $N=1,2,\ldots$ is a regular sequence of signatures such that
\begin{equation*}
\lim_{N\rightarrow\infty}[m(\lambda(N))]=\mathbf{m}.
\end{equation*}
Let $B^{(N)}=(\beta_1^{(N)},\ldots,\beta_N^{(N)})$
such that
\begin{equation}
\lim_{N\rightarrow\infty}\sup_{1\leq i\leq N}|\beta_i^{(N)}-1|N=0.\label{lsz}
\end{equation}
Then
\begin{multline*}
\lim_{N\rightarrow\infty}\partial_1\partial_2\log\left(\frac{s_{\lambda(N)}(u_1,u_2,\ldots,u_k,\beta_{k+1}^{(N)},\ldots,\beta_N^{(N)})}{s_{\lambda(N)}\left(B^{(N)}\right)}\right)\\
=\partial_1\partial_2\left(1-(u_1-1)(u_2-1)\frac{u_1H'_{\mathbf{m}}(u_1)-u_2H'_{\mathbf{m}}(u_2)}{u_1-u_2}\right),
\end{multline*}
where the convergence is uniform over an open complex neighborhood of $(u_1,\ldots,u_k)=(1^k)$.
\end{proposition}
\begin{proof}Note that
\begin{equation*}
  \partial_1\partial_2\log\Bigl(\frac{s_{\lambda(N)}(u_1,\ldots,u_k,\beta_{k+1}^{(N)},\ldots,\beta_N^{(N)})}{s_{\lambda(N)}\left(B^{(N)}\right)}\Bigr)=\partial_1\partial_2\log\Bigl(\frac{s_{\lambda(N)}(u_1,\ldots,u_k,\beta_{k+1}^{(N)},\ldots,\beta_N^{(N)})}{s_{\lambda(N)}\left(1^N\right)}\Bigr),
\end{equation*}
by Lemma~\ref{l513}, it suffices to show that
\begin{multline}
\lim_{N\rightarrow\infty}\partial_1\partial_2\log\left(\frac{s_{\lambda(N)}(u_1,u_2,\ldots,u_k,\beta_{k+1}^{(N)},\ldots,\beta_N^{(N)})}{s_{\lambda(N)}\left(1^N\right)}\right)\label{ib}\\
=\lim_{N\rightarrow\infty}\partial_1\partial_2\log\left(\frac{s_{\lambda(N)}(u_1,u_2,\ldots,u_k,1^{N-k})}{s_{\lambda(N)}(1^N)}\right).
\end{multline}
Indeed, it suffices to show that the left hand side of \eqref{ib} is independent of $B^{(N)}$.
Recall that by Lemma~\ref{hciz}, we have:
\begin{multline}
\frac{s_{\lambda(N)}\left(u_1,\ldots,u_k,\beta_{\ol{k+1}}^{(N)},\ldots,\beta_{\ol{N}}^{(N)}\right)}{s_{\lambda(N)}(1,\ldots,1)}\\
=\left[\prod_{1\leq i<j\leq k}\frac{\log(u_i)-\log(u_j)}{u_i-u_j}\prod_{1\leq i\leq k,k+1\leq j\leq N}\frac{\log\left(u_i\right)-\log\left(\beta_{\ol{j}}^{(N)}\right)}{u_i-\beta_{\ol{j}}^{(N)}}\right]\\
\left[\prod_{k+1\leq i<j\leq N}\frac{\log\left(\beta_{\ol{i}}^{(N)}\right)-\log\left(\beta_{\ol{j}}^{(N)}\right)}{\beta_{\ol{i}}^{(N)}-\beta_{\ol{j}}^{(N)}}\right]
\int_{U(N)}e^{z_NN\mathrm{tr}(U^*P_{N,B,k}UQ_N)}dU.
\label{in1}
\end{multline}

We can compute
\begin{multline*}
\partial_1\partial_2\log\left(\frac{s_{\lambda(N)}(u_1,u_2,\ldots,u_k,\beta_{k+1}^{(N)},\ldots,\beta_N^{(N)})}{s_{\lambda(N)}\left(1^N\right)}\right)\\
=\partial_1\partial_2\left[\frac{\log(u_1)-\log(u_2)}{u_1-u_2}\right]+\partial_1\partial_2\log I_{N,B,k}(z_N),
\end{multline*}
where

\begin{equation}
\label{inbk}
I_{N,B,k}(z):=\int_{U(N)}e^{z N\mathrm{tr}(U^*P_{N,B,k}UQ_N)}dU
=\int_{U(N)}e^{z N\sum_{i,j=1}^{N}P_{N,B,k}(i,i)Q_N(j,j)|U(i,j)|^2}dU,\\
\end{equation}


Hence it suffices to show that
\begin{equation*}
\lim_{N\rightarrow\infty}\partial_1\partial_2\log I_{N,B,k}(z_N)=\lim_{N\rightarrow\infty}\partial_1\partial_2\log I_{N,k}(z_N),
\end{equation*}
where $I_{N,k}(z)$ is given by
\begin{equation}
  \label{ink}I_{N,k}(z):=\int_{U(N)}e^{z N\mathrm{tr}(U^*P_{N,k}UQ_N)}dU
  =\int_{U(N)}e^{z N\sum_{i,j=1}^{N}P_{N,k}(i,i)Q_N(j,j)|U(i,j)|^2}dU.
\end{equation}

Using the same technique as in the proof of  Theorem 2.1 in~\cite{ggn}, we have
\begin{align}
\label{inbs}
\log
I_{N,B,k}(z_N)&=\sum_{g=0}^{\infty}\frac{1}{N^{2g-2}}\sum_{d=0}^{\infty}\frac{1}{d!}\sum_{|\alpha|=|\psi|=d}(-1)^{l(\alpha)+l(\psi)}\\
&\prod_{i=1}^{l(\alpha)}\left(\frac{\sum_{t=1}^k\left(\log
u_t\right)^{\alpha_i}+\sum_{t=k+1}^N(\log \beta_t)^{\alpha_i}}{N}\right)\notag\\
&\times\prod_{j=1}^{l(\psi)}\left(\frac{1}{N}\sum_{t=1}^N\left(\frac{\lambda_t+N-t}{N}\right)^{\psi_j}\right)H_g(\alpha,\psi)\notag
\\
\log
\label{ins}
I_{N,k}(z_N)&=\sum_{g=0}^{\infty}\frac{1}{N^{2g-2}}\sum_{d=0}^{\infty}\frac{1}{d!}\sum_{|\alpha|=|\psi|=d}(-1)^{l(\alpha)+l(\psi)}\\
&\prod_{i=1}^{l(\alpha)}\left(\frac{\sum_{t=1}^k\left(\log
u_t\right)^{\alpha_i}}{N}\right)\times\prod_{j=1}^{l(\psi)}\left(\frac{1}{N}\sum_{t=1}^N\left(\frac{\lambda_t+N-t}{N}\right)^{\psi_j}\right)H_g(\alpha,\psi)\notag
\end{align}
where the $H_g(\alpha,\beta)$ are the double Hurwitz number, see~\cite{AO00}.
Under the assumption that $\lambda(N)$ is a regular sequence of signatures and
\eqref{lsz}, we have:
\begin{multline*}
\left|\prod_{i=1}^{l(\alpha)}\left(\frac{\sum_{t=1}^k\left(\log u_t\right)^{\alpha_i}+\sum_{t=k+1}^N(\log \beta_t)^{\alpha_i}}{N}\right)
\times\prod_{j=1}^{l(\psi)}\left(\frac{1}{N}\sum_{t=1}^N\left(\frac{\lambda_t+N-t}{N}\right)^{\psi_j}\right)\right|\\
\leq  C_4^d[\max_{1\leq t\leq k}|\log u_t|]^d\left(\frac{k}{N}\right)^{l(\alpha)}.
\end{multline*}
Let 
\begin{multline*}
\Phi_{d,g,B,N}(u_1,\ldots,u_k)=\\
\frac{1}{d!}\sum_{|\alpha|=|\psi|=d}(-1)^{l(\alpha)+l(\psi)}
\prod_{i=1}^{l(\alpha)}\left(\frac{\sum_{t=1}^k\left(\log
u_t\right)^{\alpha_i}+\sum_{t=k+1}^N(\log \beta_t)^{\alpha_i}}{N}\right) \\
\times\prod_{j=1}^{l(\psi)}\left(\frac{1}{N}\sum_{t=1}^N\left(\frac{\lambda_t+N-t}{N}\right)^{\psi_j}\right)H_g(\alpha,\psi),
\end{multline*}
then $\Phi_{d,g,B,N}(u_1,\ldots,u_k)$ is an analytic function in an open complex neighborhood of $1^k$. Let 
\begin{equation*}
\Phi_{d,g,N}(u_1,\ldots,u_k)=\Phi_{d,g,1^N,N}(u_1,\ldots,u_k).
\end{equation*}

We will compute $\partial_1\partial_2\Phi_{d,g,N}(u_1,\ldots,u_k)$.
Note that only those signatures $\alpha$ satisfying $l(\alpha)\geq 2$
contributes to the derivative
$\partial_1\partial_2\Phi_{d,g,N}(u_1,\ldots,u_k)$.
In a sufficiently small complex neighborhood of $1^k$, and when $N$ is
sufficiently large, by Lemma~\ref{chg} below, we have
\begin{equation*}
|\partial_1\partial_2\Phi_{d,g,B,N}(u_1,\ldots,u_k)|\leq \frac{1}{N^2}\left(\frac{1}{2}\right)^d
\end{equation*}
Therefore for any $\epsilon>0$ there exists an integer $E\geq 1$, such that for
any $(u_1,\ldots,u_k)$ in an small open complex neighborhood of $1^k$, and for
any $B^{(N)}$ satisfying \eqref{lsz}, for any $g\geq 0$, we have
\begin{equation}
\sum_{d\geq E+1}N^2|\partial_1\partial_2\Phi_{d,g,B,N}(u_1,\ldots,u_k)|< \frac{\epsilon}{8}.\label{tle}
\end{equation}
By the uniform convergence of the derivative we have
\begin{equation*}
\partial_1\partial_2\log I_{N,B,k}(z_N)=\sum_{g=0}^{\infty}\frac{1}{N^{2g-2}}\sum_{d=0}^{\infty}\partial_1\partial_2\Phi_{d,g,B,N}(u_1,\ldots,u_k).
\end{equation*}
As a result,
\begin{multline}
\label{db1}
|\partial_1\partial_2\log I_{N,B,k}(z_N)-\partial_1\partial_2\log I_{N,k}(z_N)|\\
\leq \sum_{g=0}^{\infty}\frac{1}{N^{2g-2}}\sum_{d=1}^{E}|\partial_1\partial_2\left[\Phi_{d,g,B,N}(u_1,u_2,\ldots,u_k)-\Phi_{d,g,N}(u_1,u_2,\ldots,u_k)\right]|\\
+\sum_{g=0}^{\infty}\frac{1}{N^{2g-2}}\sum_{d\geq E+1}|\partial_1\partial_2\left[\Phi_{d,g,B,N}(u_1,u_2,\ldots,u_k)-\Phi_{d,g,N}(u_1,u_2,\ldots,u_k)\right]|
\end{multline}

By \eqref{tle}, we have
\begin{multline}
\label{db1l}
\sum_{g=0}^{\infty}\frac{1}{N^{2g-2}}\sum_{d\geq E+1}|\partial_1\partial_2\left[\Phi_{d,g,B,N}(u_1,u_2,\ldots,u_k)-\Phi_{d,g,N}(u_1,u_2,\ldots,u_k)\right]|\\
\leq \frac{\epsilon N^2}{4(N^2-1)}<\frac{\epsilon}{2},
\end{multline}
when $N$ is large. Moreover,
\begin{align}
&|\partial_1\partial_2\left[\Phi_{d,g,B,N}(u_1,u_2,\ldots,u_k)-\Phi_{d,g,N}(u_1,u_2,\ldots,u_k)\right]|\notag\\
&=\left|\frac{1}{d!}\sum_{|\alpha|=|\psi|=d}(-1)^{l(\alpha)+l(\psi)}\sum_{i=1}^{l(\alpha)}\sum_{1\leq j\leq l(\alpha),j\neq i}\frac{[\log u_1]^{\alpha_i-1}[\log u_2]^{\alpha_j-1}}{N^2u_1u_2}\right.\notag\\
&\times\left(\prod_{1\leq r\leq l(\alpha),r\neq i,r\neq j}\left(\frac{\sum_{t=1}^k(\log u_t)^{\alpha_r}+\sum_{t=k+1}^{N}(\log \beta_t)^{\alpha_r}}{N}\right)\right.\notag\\
&\left.-\prod_{1\leq r\leq l(\alpha),r\neq i,r\neq j}\left(\frac{\sum_{t=1}^k(\log u_t)^{\alpha_r}}{N}\right)\right)\notag\\
&\left.\times\prod_{j=1}^{l(\psi)}\left(\frac{1}{N}\sum_{t=1}^N\left(\frac{\lambda_t+N-t}{N}\right)^{\psi_j}\right)H_g(\alpha,\psi)\right|\notag\\
&=o\left(\frac{1}{N^3}\right),\label{db1s}
\end{align}
given that $B^{(N)}$ satisfies  \eqref{lsz}. By \eqref{db1}, \eqref{db1l}, \eqref{db1s}, we have
\begin{equation*}
\lim_{N\rightarrow\infty}|\partial_1\partial_2\log I_{N,B,k}(z_N)-\partial_1\partial_2\log I_{N,k}(z_N)|=0,
\end{equation*}
and the proof is complete.
\end{proof}

We now state the technical lemma about the generating function of the double
Hurwitz numbers:
\begin{lemma}
  \label{chg}
  Let $H_g(\alpha,\beta)$ be the double Hurwitz number as
  given in \eqref{inbs} and \eqref{ins}. For each $g\geq 0$, let
\begin{equation*}
\mathbf{H}_g(z)=\sum_{d=1}^{\infty}\frac{z^d}{d!}\sum_{|\alpha|=|\beta|=d}H_g(\alpha,\beta).
\end{equation*}
Then the series $\mathbf{H}_g(z)$ has a radius of convergence of at least $\frac{1}{54}$ and at most $\frac{2}{27}$. 
\end{lemma}
\begin{proof}See Theorem 3.4 of~\cite{ggn}.
\end{proof}

Let $\mathcal{R}(\Omega,\check{a})$ be a contracting square-hexagon lattice. Let
$[2\kappa N]$ be the row number of a row of $\mathcal{R}(\Omega,\check{a})$
counting from the bottom. Then the induced probability measure of dimer
configurations on that row is also a measure on Young diagrams given by
$[(1-\kappa)N]$-tuples. We define the moment function by
\begin{equation}
p_{j}^{\kappa}:=\sum_{i=1}^{(1-\kappa)N}\left(\lambda_i^{([(1-\kappa)N])}+[(1-\kappa)N]-i\right)^j.\label{mpjk}
\end{equation}

\begin{theorem}\label{tt716}The collection of random variables 
\begin{equation*}
\{N^{-j}(p_j^{\kappa}-\EE p_j^{\kappa})\}_{0<\kappa\leq 1;j\in N},
\end{equation*}
defined by~\eqref{mpjk}, is asymptotically Gaussian with limit covariance
\begin{multline*}
\lim_{N\rightarrow\infty}N^{-(j_1+j_2)}\mathrm{cov}(p_{j_1}^{\kappa_1},p_{j_2}^{\kappa_2})=\frac{(1-\kappa_1)^{j_1}(1-\kappa_2)^{j_2}}{(2\pi\mathbf{i})^2}\times\\
\oint_{|z|=\epsilon}\oint_{|w|=2\epsilon}
\Bigl(\frac{1}{z}+1+(1+z)A_{\kappa_1}(1+z)\Bigr)^{j_1}
\Bigl(\frac{1}{w}+1+(1+w)A_{\kappa_2}(1+w)\Bigr)^{j_2}Q(z,w)dzdw,
\end{multline*}
where $\epsilon\ll 1$ and $1\geq \kappa_1\geq\kappa_2>0$,
\begin{equation*}
A_{\kappa}(1+z)=\frac{1}{1-\kappa}H_{\mathbf{m}_{\omega}}'(1+z)+\frac{\kappa}{(1-\kappa) n}\sum_{l=1}^{n}\frac{1}{z+c_l+1}.
\end{equation*}
Here $c_l=\frac{1}{x_{2l-1}}$ is the reciprocal of edge weights. Moreover,
\begin{equation*}
Q(z,w)=\partial_z\partial_w\left(\log\left(1-\frac{zw[(1+z)H_{\mathbf{m}_{\omega}}'(1+z)-(1+w)H_{\mathbf{m}_{\omega}}'(1+w)]}{z-w}\right)\right)+\frac{1}{(z-w)^2}
\end{equation*}
\end{theorem}
\begin{proof}The theorem follows from Proposition~\ref{p515} and Theorem~\ref{gff1}.
\end{proof}

\noindent\textbf{Proof of Theorem~\ref{tt73}} Theorem~\ref{tt73} follows from Theorem~\ref{tt716} in a similar way as the proof of Theorem 6.3 in~\cite{bk}.

\section{Examples}\label{sec:ex}

In this section, we study a few examples of contracting square-hexagon lattices.
By applying the theory developed in previous work and this paper, we explicitly
find the limit shapes and the frozen boundaries of perfect matching on these
lattices. The examples with parameters $x_i=1$ are studied in
Section~\ref{sec:ex1}; while the examples with general periodic parameters
$x_i\neq 1$ are studied in Section~\ref{sec:ex2}.

\subsection{Examples with $x_i=1$}\label{sec:ex1}

\subsubsection{An Aztec rectangle.} 
Figure~\ref{fig:aztec} shows a domino tiling of a large Aztec rectangle, sampled
from the Boltzmann measure with $1\times 4$ periodic weights given by
$x_1=x_2=1$, $y_1=4$, $y_2=\frac{1}{4}$. Here the number of distinct $y_i$'s in
a period is equal to $m=2$ and the number of distinct parts for the boundary
partition $\omega$ is equal to $s=4$.

This typical tiling, as explained in Section~\ref{sec:fb}, exhibits a spatial
phase separation between frozen regions close to the boundary, and a liquid
region in the middle. The \emph{frozen boundary} separating the phases converges
in probability to an algebraic curve which, as one can see from
Figure~\ref{fig:aztec}, has $(m+1)s-1=11$ tangent
points with the bottom boundary. In each interval $[a_j,b_j]$ for $1\leq j\leq
s=4$ of the bottom boundary, the frozen boundary has 2 tangent points, while in
each interval $(b_j,a_{j+1})$ for $1\leq j\leq s-1=3$, the frozen boundary has 1
tangent point. 

\begin{figure}[h!] 
\includegraphics[width=\textwidth]{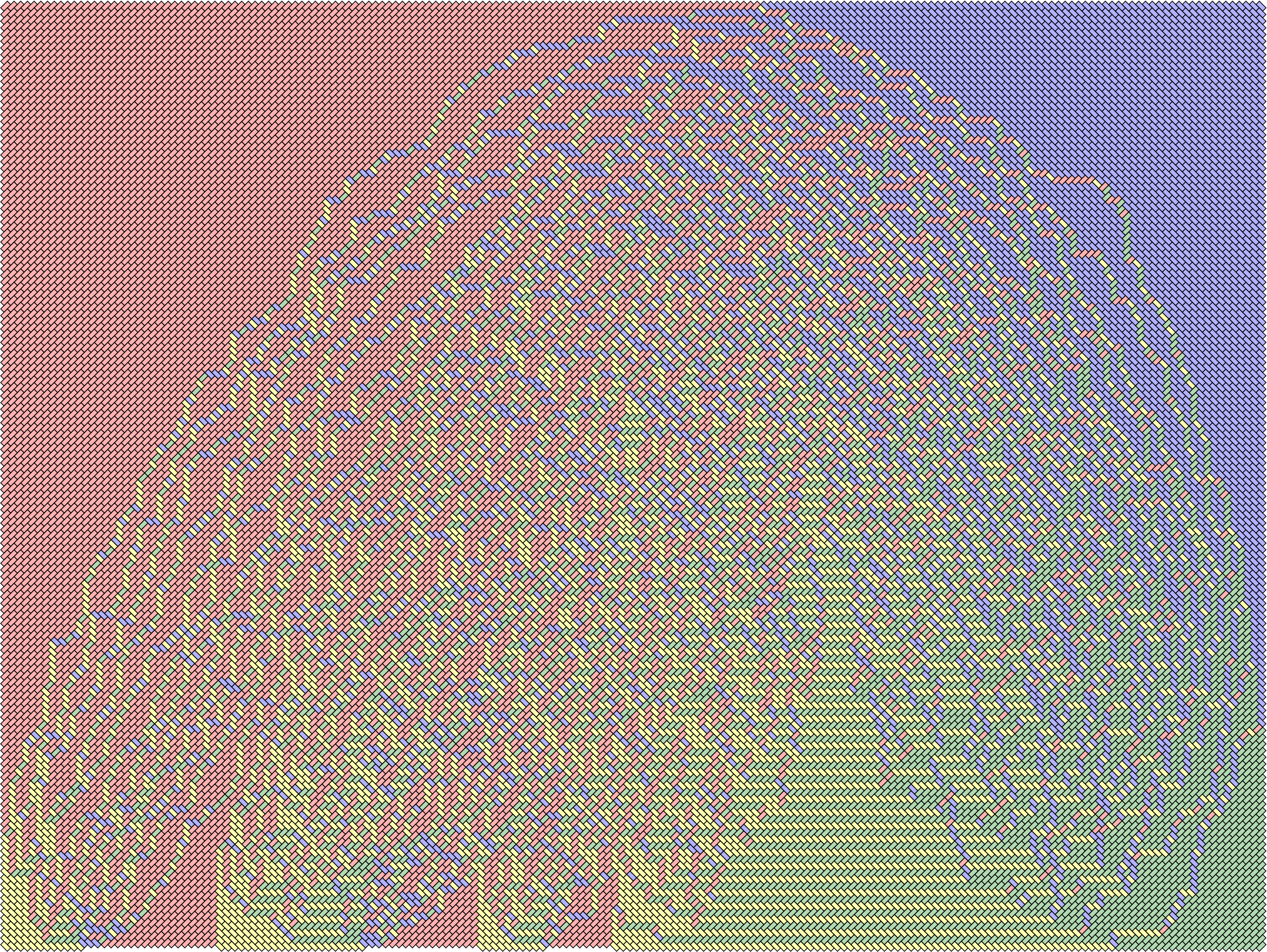}
\caption{A domino tiling of a large Aztec rectangle with $1\times 4$ periodic
weights, and a boundary partition with parts taking 4 distinct values. It
exhibits a frozen boundary with 11 tangent points on the boundary}\label{fig:aztec}
 \end{figure}

 \subsubsection{A simple square-hexagon lattice}
Consider a square-hexagon lattice $\SH(\check{a})$ in which 
 \begin{equation}
 a_i=
 \begin{cases}
   0 & \text{if $i$ is odd}\\
   1&\text{if $i$ is even}
 \end{cases}.
 \label{cdna}
 \end{equation}
See the left graph of Figure~\ref{fig:shtl} for a subgraph of the original
square-hexagon lattice, and the right graph Figure~\ref{fig:shtl} for a subgraph
by translating every row to start from the same vertical line.
 
  \begin{figure}[htbp] 
\includegraphics*[width=0.3\hsize]{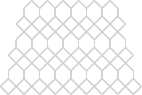}\qquad\qquad\includegraphics*[width=0.3\hsize]{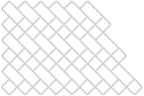}
\caption{A subgraph of a square hexagon lattice: the left graph represents a subgraph of the original lattice; the right graph represents the lattice obtained by letting all the rows of the left graph start from the same vertical line.}\label{fig:shtl}
 \end{figure}
 
 Assume that the contracting square-hexagon lattice
 $\mathcal{R}(\Omega,\check{a})$ with $\check{a}$ satisfying \eqref{cdna} has
 boundary condition specified by
 \begin{equation*}
 \Omega=(A_1,A_1+1,\ldots,B_1-1,B_1,A_2,A_2+1,\ldots,B_2-1,B_2,A_3,A_3+1,\ldots,B_3-1,B_3).
 \end{equation*}
 i.e.\@ the boundary partition $\omega$ has three distinct parts of macroscopic
 size, repeated a macroscopic number of times.
 
 We give two realizations of the dimer model on a large graph
 $\mathcal{R}(\Omega,\check{a})$ on Figures~\ref{fig:shu} and~\ref{fig:shw}.
 Instead of drawing dimers on the graph, we draw only the particles from the
 bijective correspondence with finite Maya diagrams of Section~\ref{sec:pyd}.

 Figure~\ref{fig:shu} shows the uniform dimer model (i.e.\ all
 the edge weights are 1) on a contracting square-hexagon lattice
 $\mathcal{R}(\Omega,\check{a})$, whereas Figure~\ref{fig:shw} shows a random
 dimer configuration of the same graph but with periodic weights with a
 fundamental domain consisting of 8 rows, and edge
 weights  $x_1=x_2=x_3=x_4=1$, $y_1=3,y_3=0.5$. We can see that in the periodic
 case, the frozen boundary has more tangent points with the horizontal line
 $\kappa=0$ compared to the uniform case.
 
 \begin{figure}[htbp] 
\includegraphics*[width=0.9\hsize]{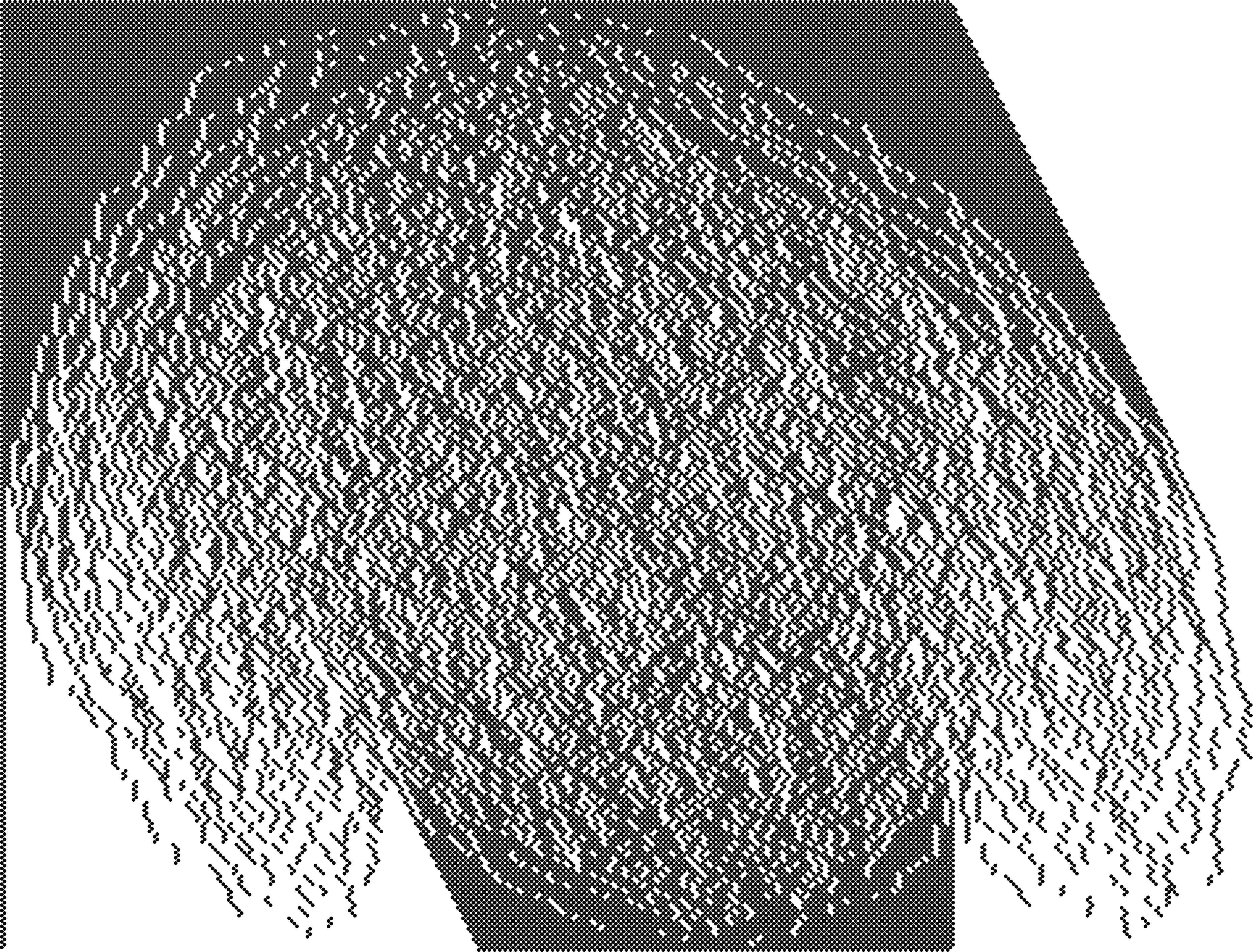}
\caption{Limit shape of perfect matchings on the square-hexagon lattice with
uniform weights $x_1=x_2=x_3=x_4=1=y_1=y_3$.}\label{fig:shu}
 \end{figure}
 
 \begin{figure}[htbp] 
\includegraphics*[width=0.9\hsize]{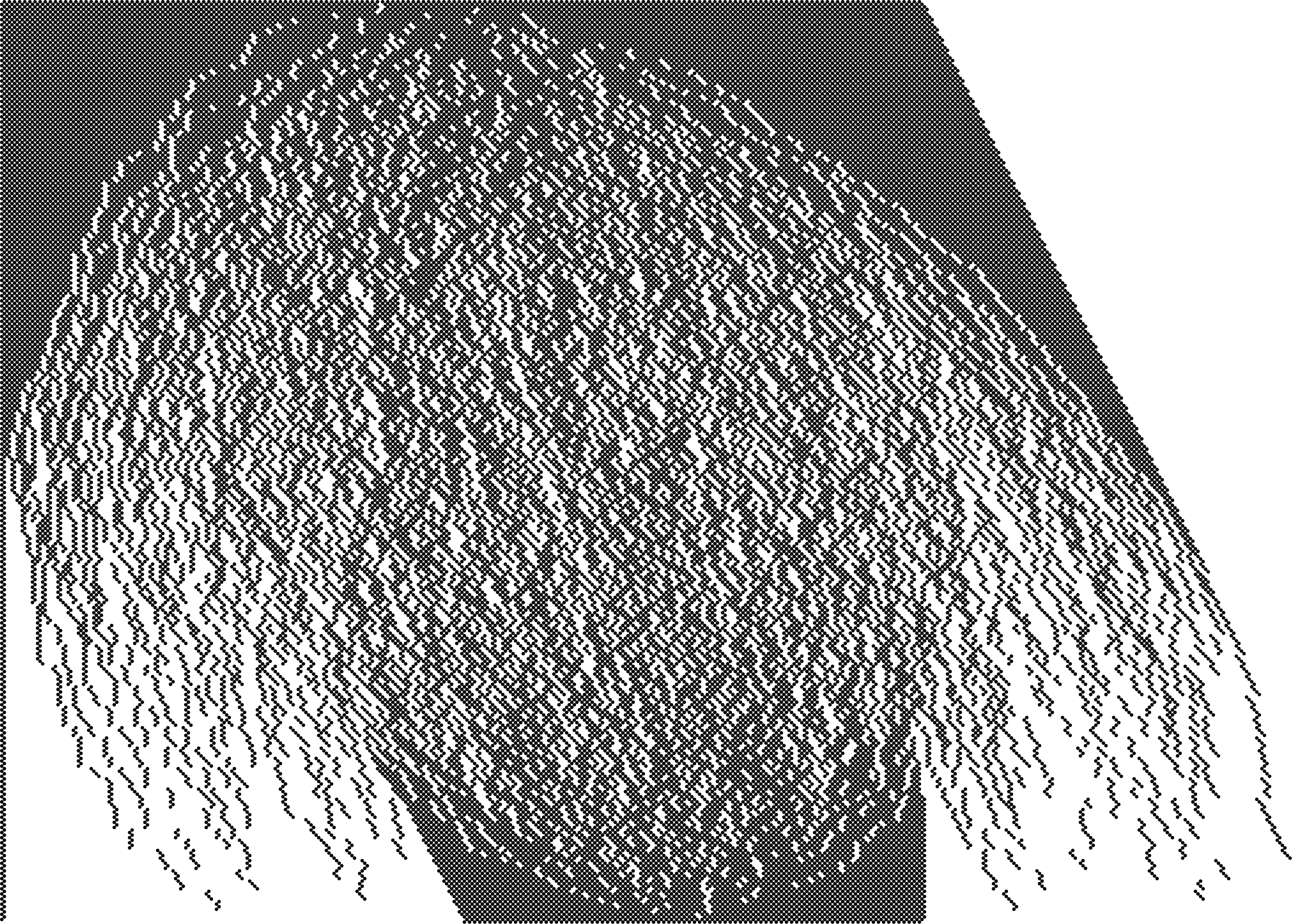}
\caption{Limit shape of perfect matchings on the square-hexagon lattice with
periodic weights $x_1=x_2=x_3=x_4=1,y_1=3,y_3=0.5$.}\label{fig:shw}
 \end{figure}
 
We may also consider contracting square hexagon lattices with $\check{a}$
satisfying \eqref{cdna} and other boundary conditions. For example, we may
consider the counting measures corresponding to the signatures on the bottom
boundary converge to a uniform measure $\mathbf{m}_{\omega}$ on $[0,2]$ as
$N\rightarrow\infty$. Assume that each fundamental domain consists of four rows,
$x_1=x_2=1$, and $c_1=\frac{1}{y_1}=1$. In this case we have
\begin{equation*}
\mathrm{St}_{\mathbf{m}_{\omega}}(t)=-\frac{1}{2}\log\left(1-\frac{2}{t}\right).
\end{equation*}
Then we solve the equation $\mathrm{St}_{\mathbf{m}_{\omega}}(t)=\log z$ for $t$
and substitute it into \eqref{es}, we have
\begin{eqnarray}
\frac{z[(4-\kappa)z-3\kappa]}{2(z+1)(z-1)}=\chi.\label{seex}
\end{eqnarray}
 Hence the frozen boundary is given by the condition that the discriminant of
 \eqref{seex} is 0, which gives 
 \begin{equation*}
 16\chi^2+9\kappa^2+8\chi\kappa-32\chi=0;\ \kappa\geq 0.
 \end{equation*}
See Figure~\ref{fig:fbu} for the frozen boundary in that case, which is not a
full closed real algebraic curve inscribed in the domain, contrary to the ``stepped
case'' above.

 \begin{figure}[htbp] 
\includegraphics*[width=0.7\hsize]{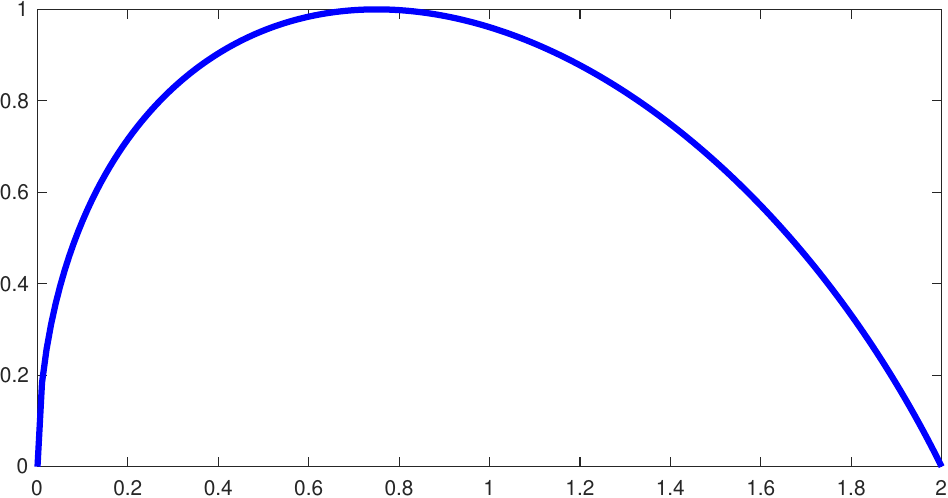}
\caption{Frozen boundary of contracting square-hexagon lattices with $x_1=x_2=y_1=1$, $a_{2i-1}$=0, $a_{2i}=1$, and uniform counting measure on $[0,2]$ for the signatures on the lower boundary}\label{fig:fbu}
 \end{figure}
 
 \subsection{An example with $\lim_{N\rightarrow\infty}x_i^{(N)}\neq 1$}\label{sec:ex2}
 We show on a particular example how computations above can be adapted to take
 into account situations when the weights $x_i^{(N)}$ do not go to 1. The
 boundary condition considered is given by a staircase partition where the steps
 have constant height, not depending on $N$. More precisely, 
 consider a contracting square-hexagon lattice $\mathcal{R}(\Omega,\check{a})$
 with edge weights assigned as in Proposition~\ref{p16}. Assume the
 configuration on the boundary row is given by the following very specific partition:
 \begin{equation}
 \lambda(N)=((M-1)(N-1),(M-1)(N-2),\ldots,(M-1),0),\label{scb}
 \end{equation}
 where $M\geq 1$ is a positive integer. In other words, there are $N$ vertices
 remaining in the boundary row in total; the leftmost vertex and the rightmost
 vertex are remaining vertices in the boundary row; between each pair of nearest
 remaining vertices on the boundary row, there are $(M-1)$ removed vertices.
 This distribution in the limit as $N$ goes to infinity converges to a uniform
 measure on the whole interval $[0,M]$. By~\cite[Example~1.3.7]{IGM15}, we have
 \begin{equation*}
 s_{\lambda(N)}(x_1,\ldots,x_N)=\prod_{1\leq i<j\leq N}\frac{x_i^M-x_j^M}{x_i-x_j}
 \end{equation*}
 when all the $x_i$'s are distinct, and extended by continuity when some $x_i$'s
 are equal.
 
We further assume that the edge weights are assigned periodically with a finite
period $n$. That is, for each $i\in \NN$,  and $\ol{i}=i\mod n$
\begin{equation}
x_i=x_{\ol{i}};\label{xp}
\end{equation}
and if $i\in I_2$,
\begin{equation}
y_i=y_{\ol{i}}.\label{yp}
\end{equation}
By Proposition~\ref{p16}, the partition function of dimer configurations on the
graph $\mathcal{R}(\Omega,\check{a})$ can be computed by the following formula
\begin{equation*}
Z_N=\left[\prod_{i\in I_2}\prod_{t=i+1}^{N}(1+y_ix_t)\right]
s_{\lambda(N)}(x_1,\ldots,x_N).
\end{equation*}

When the edge weights satisfy~\eqref{xp} and~\eqref{yp}, we may compute the
free energy as follows, distinguishing the cases where $x_i=x_j$ or not in the
expression above:
\begin{multline*}
\mathcal{F}:=\lim_{N\rightarrow\infty}\frac{1}{N^2}\log Z_N=\\
\frac{1}{n^2}\left[\sum_{1\leq i<j\leq
n}\!\!\log\left(\frac{x_i^M-x_j^M}{x_i-x_j}\right)+\frac{1}{2}\sum_{1\leq i\leq
n}\log(M x_i^{M-1})
+\frac{1}{2}\!\!\sum_{i\in
I_2\cap\{1,2,\ldots,n\}}\sum_{t=1}^n\log(1+y_ix_t)\right].
\end{multline*}

Here we assumed that all the weights $x_i$'s in a fundamental domain are
distinct. The case when some weights coincide is obtained by continuity.

We can also compute the Schur generating function for the random partition
corresponding to the random dimer configuration on  each row of the graph,
and obtain in this particular case an explicit analogue of
Proposition~\ref{p213} giving the moments of the limiting distribution of
particles $\mathbf{m}^{\kappa}$, at macroscopic height $\kappa\in[0,1]$.


Let $\rho^k$ be the distribution of the random partition corresponding to the random dimer configuration on the $k$th row. By Lemma~\ref{lm212}, we have
\begin{equation*}
\mathcal{S}_{\rho^k,X^{(N-t)}}(u_1,\ldots,u_{N-t})=
\frac{s_{\lambda(N)}(u_1,\ldots,u_{N-t},x_1,\ldots,x_t)}{s_{\lambda(N)}(x_1,\ldots,x_N)}
\!\!\!\!\!
\prod_{i\in\{1,2,\ldots,t/t+1\}\cap I_2}
\prod_{j=1}^{N-t}\Bigl(\frac{1+y_iu_j}{1+y_ix_{t+j}}\Bigr).
\end{equation*}
for $k=2t+1$ or $k=2t+2$.
When the edge weights are assigned periodically as in \eqref{xp} and
\eqref{yp}, and letting $N\to\infty$, $\frac{t}{N}\to \kappa\in[0,1)$, we have
\begin{multline*}
\lim_{(1-\kappa)N\rightarrow\infty}
\frac{1}{(1-\kappa)N} \log\mathcal{S}_{\rho^k,X^{(N-t)}}(u_1,\ldots,u_{\ell},x_{t+1+\ell},\ldots,x_{N-t})\\
=\sum_{1\leq i\leq \ell}\left[Q_{\kappa}(u_i)-Q_{\kappa}(x_{t+i})\right],
\end{multline*}
where
\begin{equation*}
  Q_{\kappa}(u)=\frac{1}{1-\kappa}\left[%
    \frac{1}{n}\sum_{1\leq j\leq n}
    \log\left(%
    \frac{u^M-x_j^M}{u-x_j}
    \right)
    +\frac{\kappa}{n}
    \sum_{i\in\{1,2,\ldots,n\}\cap I_2}\log(1+y_i u)
  \right].
\end{equation*}

Let $p\geq 1$ be a positive integer. Let $\rho_{\lfloor (1-\kappa )N\rfloor}$ be
the probability measure on the row of the square-hexagon lattice with $\lfloor
(1-\kappa) N \rfloor$ present $V$-edges, and let $\mathbf{m}_{\rho_{\lfloor
(1-\kappa) N\rfloor}}$ be the corresponding random counting measure. Let $\mathcal{N}=\lfloor (1-\kappa) N \rfloor$. Let $U=(u_1,\ldots,u_N)$ and $X=(x_1,\ldots,x_N)$ satisfying $x_{i\mod n}=x_i$.

Following similar computations as the proof of~\cite[Theorem~5.1]{bg}, we have
that the leading term for
\begin{equation*}
\mathcal{N}^{p+1}\int_{\RR}x^p\mathbf{m}_{\rho_{\lfloor(1-\kappa) N\rfloor}}
\end{equation*}
is given by
\begin{equation*}
\mathcal{M}_{p,\mathcal{N}}:=
\lim_{U\rightarrow X}\sum_{i=1}^{\mathcal{N}}\sum_{l=0}^{p}\mathcal{N}^{p-l}
\binom{p}{l}
u_i^p[Q'_{\kappa}(u_i)]^{p-l}\left(\sum_{j\in
  \{1,2,\ldots\mathcal{N}\}\setminus\{i\}}\frac{1}{u_i-u_j}\right)^{l}.
\end{equation*}

 First we assume that the edge weights $x_1,\ldots,x_n$ are pairwise distinct. Let
\begin{equation*}
S_{\mathcal{N}}(i)=\{j\in \{1,2,\ldots, \mathcal{N}\}: j\mod n=i\}=
\{an+i; 0\leq a \leq \lfloor \mathcal{N}/n\rfloor\}.
\end{equation*}
Then
  \begin{multline*}
\mathcal{M}_{p,\mathcal{N}}=\lim_{U\rightarrow X}\sum_{i=1}^{\mathcal{N}}\sum_{l=0}^{p}\mathcal{N}^{p-l}\left(\begin{array}{c}p\\l\end{array}\right)u_i^p[Q'_{\kappa}(u_i)]^{p-l}\\
\times\left[\sum_{k=0}^{l}\left(\begin{array}{c}l\\k\end{array}\right)\left(\sum_{j\in\{1,2,\ldots,\mathcal{N}\}\setminus
S_{\mathcal{N}}(i)}\frac{1}{u_i-u_j}\right)^{l-k}\left(\sum_{j\in
S_{\mathcal{N}}(i)\setminus\{i\}}\frac{1}{u_i-u_j}\right)^{k}\right],
\end{multline*}
and
\begin{multline*}
\lim_{N\rightarrow\infty}\frac{\mathcal{M}_{p,\mathcal{N}}}{\mathcal{N}^{p+1}}=\lim_{U\rightarrow
X}\frac{1}{n}\sum_{i=1}^{n}\sum_{k=0}^{p}\frac{p!}{k!(p-k)!}\frac{1}{n^k(k+1)!}\\
\times
\left.%
\frac{%
\partial^k
}{
\partial u^k
}
\left[%
  u^p\left(Q_{\kappa}'(u)+
  \frac{1}{n} \sum_{1\leq j\leq n,j\neq i}\frac{1}{u-x_j}\right)^{p-k}
\right]
\right|_{u=x_i}.
\end{multline*}


Note that
\begin{equation*}
Q_{\kappa}'(u)=\frac{1}{n(1-\kappa)}\sum_{1\leq j\leq n}\left(\frac{M u^{M-1}}{u^M-x_j^M}-\frac{1}{u-x_j}\right)+\frac{\kappa}{n(1-\kappa)}\sum_{i\in\{1,2,\ldots,n\}\cap I_n}\frac{y_i}{1+y_iu}
\end{equation*}
Using residue and following similar computations, we have
\begin{equation*}
\int_{\RR}x^{p}\textbf{m}^{\kappa}(dx)=
\frac{1}{2(p+1)\pi \mathbf{i}}\oint_{C_{x_1,\ldots,x_n}}\frac{dz}{z}\left(zQ_{\kappa}'(z)+\sum_{j=1}^{n}\frac{z}{n(z-x_{j})}\right)^{p+1},
\end{equation*}
where $C_{x_1,\ldots,x_n}$ is a simple, closed, positively oriented, contour containing only the poles $x_1,\ldots,x_n$ of the integrand, and no other singularities. 

In the case that some of the edge weights $x_1,\dots, x_n$ may be equal, one has
to be separate terms in $\mathcal{M}_{p,\mathcal{N}}$ and introduce instead of
$S_\mathcal{N}(i)$ the set
$T_{\mathcal{N}}(i)=\{j\in \{1,2,\ldots, \mathcal{N}\}: x_j=x_i\}$
but at the end, we arrive to the same expression for the moment of
$\mathbf{m}^\kappa$.

In \cite{zl3}, it is proved that when the boundary partition differs from (\ref{scb}) by at most one component at the beginning, the limit shape is the same as when the boundary partition is given by (\ref{scb}).

Define $F_{\kappa,M}(z)=zQ'_\kappa(z)+\sum_{j=1}^n \frac{z}{n(z-x_j)}$.
Adapting again the computations in~\cite{bg},
we can compute the Stieltjes transform of the measure $\mathbf{m}^{\kappa}$ when
$x$ is in a neighborhood of infinity by 
\begin{align*}
  \mathrm{St}_{\mathbf{m}^{\kappa}}(x)
  &=\sum_{j=0}^{\infty}x^{-(j+1)} \int_{\RR}y^j\mathbf{m}^{\kappa}(dy)\\
&=\sum_{j=0}^{\infty}
\frac{1}{2(j+1)\pi\mathbf{i}}\oint_{C_{x_1,\ldots,x_n}}\left(\frac{F_{\kappa,M}(z)}{x}\right)^{j+1}\frac{dz}{z}\\
&=-\frac{1}{2\pi\mathbf{i}}\oint_{C_{x_1,\ldots,x_n}}\log\left(1-\frac{F_{\kappa,M}(z)}{x}\right)\frac{dz}{z}.
\end{align*}
Integration by parts gives
\begin{equation*}
\mathrm{St}_{\mathbf{m}^{\kappa}}(x)=
\frac{1}{2\pi\mathrm{i}}\left[\oint_{C_{x_1,\ldots,x_n}}\!\!\!\!\!\!\log
z\frac{\frac{d}{dz}\left(1-\frac{F_{\kappa,M}(z)}{x}\right)}{1-\frac{F_{\kappa,M}(z)}{x}}dz
-\oint_{C_{x_1,\ldots,x_n}} \!\!\!\!\!\!
d\left(\log z\log \left(1-\frac{F_{\kappa,M}(z)}{x}\right)\right)\right].
\end{equation*}

Because $F_{\kappa,M}(z)$ has a Laurent series expansion in a neighborhood of $x_i$ given by
\begin{equation*}
F_{\kappa,M}(z)=\frac{x_j}{n(z-x_j)}+\sum_{k=0}^{\infty}\alpha_k(z-x_j)^k,
\end{equation*}
$F_{\kappa,M}(z)=x$ has exactly one root in a neighborhood of $x_i$ for $1\leq
i\leq n$, and thus, we can can find a unique composite inverse Laurent series
given by
\begin{equation*}
G_{\kappa,M,j}(w)=x_j+\sum_{i=1}^{\infty}\frac{\beta_i^{(j)}}{w^i},
\end{equation*}
such that $F_{\kappa,M}(G_{\kappa,M,j}(w))=w$ when $w$ is in a neighborhood of
infinity. Then
\begin{equation}
  \label{eq:rootFG}
z_i(x)=G_{\kappa,M,j}(x)
\end{equation}
is the unique root of $F_{\kappa,M}(z)=x$ in a neighborhood of $x_i$.

Since $1-\frac{F_{\kappa,M}}{x}$ has exactly one zero $z_i(x)$ and one pole $x_i$ in a neighborhood of $x_i$, we have
\begin{equation*}
\oint_{x_i}d\left(\log z\log \left(1-\frac{F_{\kappa,M}(z)}{x}\right)\right)=0;
\end{equation*}
and therefore
\begin{equation*}
\mathrm{St}_{\mathbf{m}_{\kappa}}(x)=\sum_{j=0}^{n}\log(z_j(x))\label{sjl}
\end{equation*}
when $x$ is in a neighborhood of infinity. By the complex analyticity of both
sides of \eqref{sjl}, we infer that \eqref{sjl} holds whenever $x$ is outside
the support of $\mathbf{m}_{\kappa}$.

\begin{lemma}
  Assume the liquid region is nonempty, and assume that for any $x\in \RR$,
  $F_{\kappa, M}(z)$ has at most one pair of complex conjugate roots. Then for
  any point $(\chi,\kappa)$ lying on the frozen boundary, the equation
  $F_{\kappa,M}=\frac{\chi}{1-\kappa}$ has double roots.
\end{lemma}
\begin{proof}
  By Lemma~\ref{l36}, the continuous density $f_{\mathbf{m}_{\kappa}}(x)$ of the
  measure $\mathbf{m}_{\kappa}(x)$ with respect to the Lebesgue measure is given
  by
  \begin{equation*}
    f_{\mathbf{m}_{\kappa}}(x)=-\lim_{\epsilon\rightarrow 0+}\frac{1}{\pi}\Im
    [\mathrm{St}_{\mathbf{m}_{\kappa}}(x+\mathbf{i}\epsilon)]
  \end{equation*}
  By~\eqref{sjl}, we have
  \begin{equation*}
    f_{\mathbf{m}_{\kappa}}(x)=-\lim_{\epsilon\rightarrow 0+}\frac{1}{\pi}\mathrm{Arg}(\prod_{i=1}^n z_i(x+\mathbf{i}\epsilon)).
  \end{equation*}
  If the liquid region is nonempty, and for any $x\in \RR$, $F_{\kappa,M}=x$ has
  at most one pair of complex conjugate roots, then for each point $(\chi,\kappa)$
  in the liquid region, there is exactly one of roots
  $z_{j}\left(\frac{\chi}{1-\kappa}+\mathbf{i}\epsilon\right)$
  from~\eqref{eq:rootFG} converging to a non-real root of
  $F_{\kappa,M}(z)=\frac{\chi}{1-\kappa}$; and all the others converge to real
  roots of $F_{\kappa,M}(z)=\frac{\chi}{1-\kappa}$. Then, the lemma follows from
  the complex analyticity of the density of the limit measure with respect to
  $(\chi,\kappa)$.
\end{proof}
Let us now be more specific with the particular cases when $M$ is equal to 1 or
2.

\subsubsection{$M=1$}
When $M=1$, all the vertices on the bottom row are $V$-vertices. Let
\begin{multline*}
  F_{\kappa}(z):=F_{\kappa,M=1}(z)=\frac{\kappa
  z}{n(1-\kappa)}\sum_{i\in\{1,2,\ldots,n\}\cap
I_2}\frac{y_i}{1+y_i z}+\sum_{j=1}^{n}\frac{z}{n(z-x_{j})}\\
= K- \frac{\kappa}{n(1-\kappa)} 
\sum_{i=1}^m \frac{n_i \gamma_i}{z+\gamma_i}+\frac{1}{n}\sum_{j=1}^n\frac{x_j}{z+x_j}
\end{multline*}
where $\gamma_1,\dots,\gamma_m$ are the distinct values of
$\frac{1}{y_i},i\in I_2\cap\{1,\dots,n\}$, with respective multiplicities
$n_1,\ldots,n_m$, and $K$ the constant
$\frac{1}{1-\kappa}-(n-r)\frac{\kappa}{1-\kappa}$.

Let us call $m'$ the number of distinct values of $x_j$'s for
$j\in\{1,\ldots,n\}$.
The equation
\begin{equation}
F_{\kappa}(z)=\frac{\chi}{1-\kappa}
\label{fk1}
\end{equation}
has at least a real solution for $z$ between two consecutive $-\gamma_i$, and
between two consecutive (distinct) $x_j$. This gives at least $m+m'-2$ real
roots. As it can be written as a polynomial equation in $z$ of degree $m+m'$, we
obtain then that Equation~\eqref{fk1} as at most a pair of complex conjugate
roots.

The frozen boundary $C$ is given by the condition that the two complex conjugate
roots of \eqref{fk1} merge to a double root. More precisely, let
\begin{equation*}
U(z)=\frac{z}{n}\sum_{i\in\{1,2,\ldots,n\}\cap I_2}\frac{y_i}{1+y_iz},
\qquad
V(z)=\sum_{j=1}^{n}\frac{z}{n(z-x_{j})}.
\end{equation*}
Then the frozen boundary has a parametric equation (with parameter $z$) as follows
\begin{equation*}
  \begin{cases}
\chi&=\kappa U(z)+(1-\kappa)V(z)\\
0&=\kappa U'(z)+(1-\kappa)V'(z).
\end{cases}
\end{equation*}
Hence we have
\begin{equation}
(\chi,\kappa)=\left(\frac{U(z)V'(z)-U'(z)V(z)}{V'(z)-U'(z)},\frac{V'(z)}{V'(z)-U'(z)}\right);
\label{fbm1}
\end{equation}
for $(\chi,\kappa)$ on the frozen boundary. The dual curve of the frozen boundary has parametric equation
\begin{equation}
(x,y)=\left(-\frac{1}{V(z)},\frac{U(z)-V(z)}{V(z)}\right).\label{dcm1}
\end{equation}
The values of $z$ for which the ratio $x/y$ is 0 correspond to horizontal
tangency points of the frozen boundary.

For $1\leq j\leq m$, $z=-\gamma_j$ corresponds to the point
$(V(-\gamma_j),0)$ of the frozen boundary, and the slope $x/y$
for the corresponding point on the dual curve is 0, because $U(-\gamma_j)=\infty$.
This gives thus $m$ points on the line $\kappa=0$ with horizontal tangent for
the frozen boundary.

When $z$ is one of the $x_i$'s, the corresponding point on the frozen boundary
is $(U(x_i),1)$ and the corresponding tangent line is also horizontal. If $m'$
is the number of distinct $x_i$'s in a fundamental domain, then we get $m'$
horizontal tangency points on the line $\kappa=1$ for the frozen boundary.

 By slightly adapting the proof of Proposition~\ref{prop:cloud1}, and checking
 the definition of cloud curve for this case, one obtains the following
 \begin{proposition}
   The frozen boundary given by the parametric equation \eqref{fbm1} is a cloud
   curve of rank $m+m'$.
 \end{proposition}

\begin{figure}[htbp]
  \centering
\includegraphics*[width=0.8\hsize]{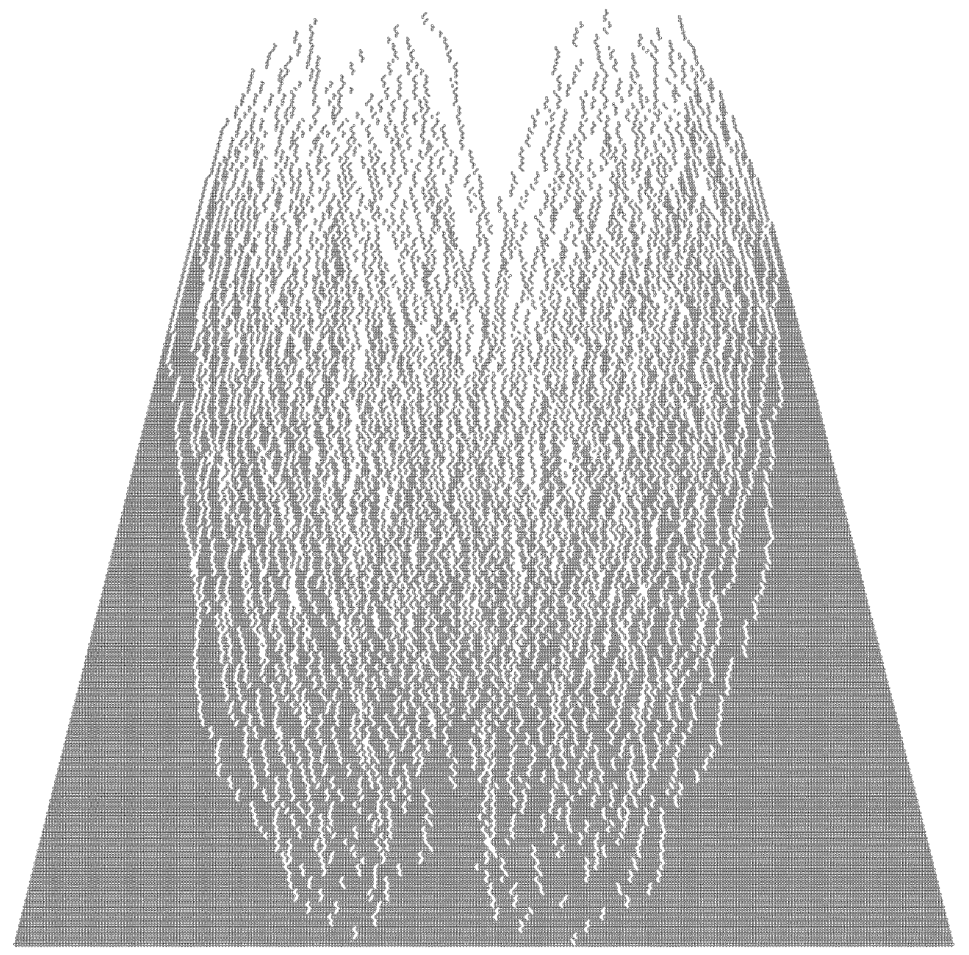}
\caption{Limit shape of perfect matchings on the square-hexagon lattice with
  weights $y_1=3,x_1=2ma,x_2=0.8,y_3=0.5,x_3=1.4,x_4=1.8$ and
$M=1$}\label{fig:shu1}
 \end{figure}


\subsubsection{$M=2$}
When $M=2$, the counting measure on the bottom row converges to the uniform
measure on $[0,2]$ when $N\rightarrow\infty$. Let
\begin{multline*}
  F_{\kappa}(z):=F_{\kappa,M=2}(z)=
  \frac{\kappa z}{n(1-\kappa)}\sum_{i\in\{1,2,\ldots,n\}\cap
  I_2}\frac{y_i}{1+y_iz}+\\
  \sum_{j=1}^{n}\frac{z}{n(z-x_{j})}+
  \frac{z}{n(1-\kappa)}\sum_{1\leq j\leq n}\frac{1}{z+x_j}.
\end{multline*}
As in the case $M=1$ discussed above, Equation~\eqref{fk1} for $M-2$ has at
most one pair of complex conjugate roots,
and
the parameters $z$ for which the two complex conjugate roots merge
to a double root correspond to the frozen boundary.
More precisely, let
\begin{equation*}
W(z)=\frac{z}{n}\sum_{1\leq j\leq n}\frac{1}{z+x_j};
\end{equation*}
then
\begin{equation*}
(\chi,\kappa)=\left(\frac{W'(z)U(z)+V'(z)U(z)-U'(z)V(z)-W'(z)V(z)}{V'(z)-U'(z)}+W(z),\frac{V'(z)+W'(z)}{V'(z)-U'(z)}\right);
\end{equation*}
for $(\chi,\kappa)$ on the frozen boundary. The dual curve of the frozen
boundary has parametric equation
\begin{equation*}
(x,y)=\left(-\frac{1}{V(z)+W(z)},\frac{U(z)-V(z)}{V(z)+W(z)}\right).
\end{equation*}
If we have $m'$ distinct values of $x_i$'s in the fundamental domain,
then for $z=x_i$, we get that the points
$\left(U(x_j)+W(x_j),1\right)$ are
$m'$ tangent points of the frozen boundary
to the line $\kappa=1$.

 \begin{figure}[htbp] 
   \centering
\includegraphics*[width=0.8\hsize]{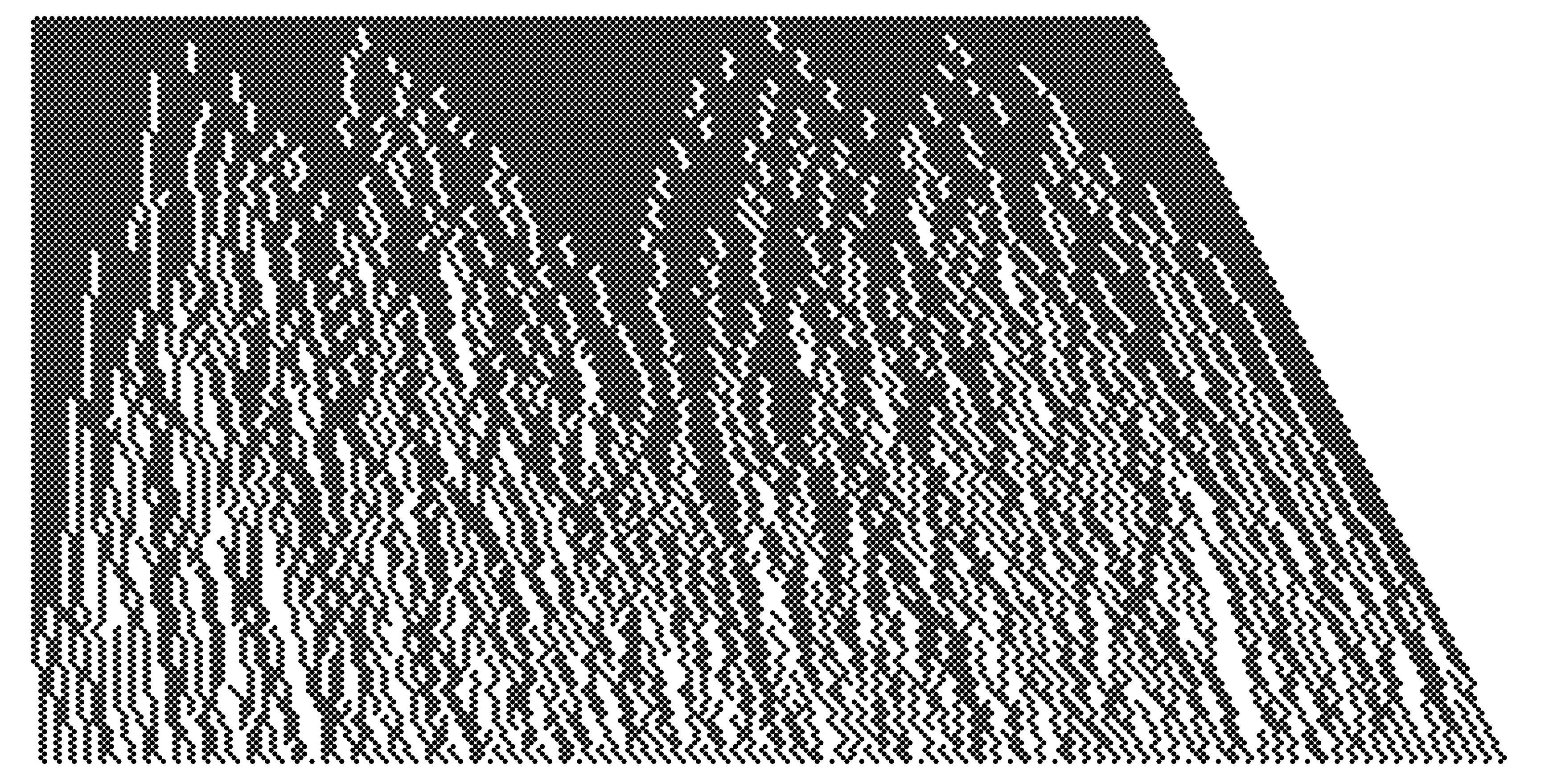}
\caption{Limit shape of perfect matchings on the square-hexagon lattice with
  weights $y_1=3,x_1=10,x_2=0.1,y_3=0.5,x_3=3.0,x_4=0.3$ and
$M=2$}\label{fig:shu2}
 \end{figure}
 
 \subsection{General Case.}
 In general, it is possible that the equation $f_{\kappa,M}(z)=\frac{\chi}{1-\kappa}$ has more than one pair of complex conjugate roots. For example, let $M=3$, $I_2=\emptyset$, $x_1=1$ and $x_2=2$. In this case, we have
 \begin{eqnarray*}
 F_{\kappa,3}(z)=\frac{z}{2}\left(\frac{1}{z-1}+\frac{1}{z-2}\right)+\frac{z}{2(1-\kappa)}\left(\frac{2z+1}{z^2+z+1}+\frac{2z+2}{z^2+2z+4}\right).
 \end{eqnarray*} 
For $\chi=1$ and $\kappa=0.5$, it is not hard to check that the equation
$F_{\kappa,3}=\frac{\chi}{1-\kappa}$ has 2 real roots and 4 non-real roots.
\bigskip

\bibliographystyle{cdraifplain}
\bibliography{sq}
\end{document}